\documentclass[twoside,reqno]{amsart}
\usepackage{amssymb,amsmath,amsfonts,amsthm}
\usepackage{enumerate}
\usepackage{amsfonts}
\usepackage{array}
\usepackage{color}

\newtheorem{theorem}{Theorem}
\newtheorem{remark}[theorem]{Remark}

\newtheorem{proposition}[theorem]{Proposition}
\newtheorem{lemma}[theorem]{Lemma}

\newcommand{\B}{{\mathbb B}}

\newcommand{\R}{{\mathbb R}}

\newcommand{\HH}{\mathbb H}

\def\dt{\partial_{t}}

\def\div{{\rm div\, }}

\def\u{{\bf u}}
\def\B{{\bf B}}

\def\vphi{\varphi}
\def\eps{\varepsilon}

\def\th{\theta}

\begin{document}

\title[Current-vortex sheets]{Approximate current-vortex sheets\\ near the onset of instability}
\author[A.~Morando] {Alessandro Morando}
\address{DICATAM, Sezione di Matematica, \newline \indent
Universit\`a di Brescia,\newline \indent Via Valotti, 9, 25133 BRESCIA, Italy}
\email{alessandro.morando@unibs.it, paolo.secchi@unibs.it, paola.trebeschi@unibs.it}

\author[P.~Secchi]{Paolo Secchi}

\author[P.~Trebeschi]{Paola Trebeschi}

\date{\today}

\begin{abstract}
The paper is concerned with the free boundary problem for 2D current-vortex sheets in ideal incompressible
magneto-hydrodynamics near the transition point between the linearized stability and instability.
In order to study the dynamics of
the discontinuity near the onset of the instability, Hunter and Thoo \cite{hunter-thoo} have introduced an asymptotic quadratically nonlinear integro-differential equation for the amplitude of small perturbations of the planar discontinuity. We study such amplitude equation and prove its nonlinear well-posedness under a stability condition given in terms of a longitudinal strain of the fluid along the discontinuity.
\end{abstract}


\thanks{{The authors are supported by the national research project PRIN 2012 \lq\lq Nonlinear Hyperbolic Partial Differential Equations, Dispersive and Transport Equations: theoretical and applicative aspects\rq\rq.}}

\maketitle

\section{Introduction and main results}\label{int}

We consider the equations of 2-dimensional incompressible magneto-hydrodynamics (MHD)
\begin{equation}
\label{mhd1}
\begin{cases}
\partial_t \u +\nabla \cdot
(\u \otimes \u-\B\otimes \B) +\nabla q =0 \, ,
\\
\partial_t \B -\nabla \times
(\u\times \B) =0 \, ,
\\
\div  \u=0\, ,\;\div \B=0\, \qquad\qquad \text{in }(0,T)\times\R^2,
\end{cases}
\end{equation}
where
$\u=(u_1,u_2)$ denotes the velocity field
and
$\B=(B_1,B_2)$ the magnetic field,
$p$ is the
pressure, $q= p+\frac{1}{2}|\B|^2$ the
total pressure
(for simplicity the density $\rho\equiv1$).

Let us consider {current-vortex sheets} solutions of \eqref{mhd1} (also called \lq\lq tangential discontinuities\rq\rq), that is weak solutions that are smooth on either side of a smooth hypersurface
$$\Gamma(t)=\{y=f(t,x)\}, \qquad  \mbox{where } t\in[0,T], \, (x,y)\in\R^2,\;
$$
and such that at $\Gamma(t)$ satisfy the boundary conditions
\begin{equation}
\label{bc}
\dt f =\u^\pm \cdot N \, ,\quad \B^\pm \cdot N=0 \, ,\quad [q]=0  \, ,
\end{equation}
with $N:=(-\partial_x f, 1)$. In \eqref{bc} $(\u^\pm,\B^\pm,q^\pm)$ denote the values of $(\u,\B,q)$ on the two sides of $\Gamma (t)$, and
$[q]=q^+_{|\Gamma}-q^-_{|\Gamma}$ the jump across
$\Gamma (t)$.

From \eqref{bc} the discontinuity front $\Gamma (t)$
is a tangential discontinuity, namely the plasma does not flow through the discontinuity front and the magnetic field is tangent to $\Gamma (t)$.
The possible jump of the tangential velocity and tangential magnetic field gives a concentration of current and vorticity on the front $\Gamma (t)$.
{Current-vortex sheets are fundamental waves in MHD and play an important role in plasma physics and astrophysics. The existence of current-vortex sheets solutions is known for compressible fluids \cite{ChenWang,trakhinin09arma}, but, as far as we know, is still an open problem for incompressible fluids, see \cite{cmst,morandotrakhinintrebeschi} for partial results.}

The necessary and sufficient linear stability condition for planar (constant coefficients) current-vortex
sheets was found a long time ago, see \cite{axford,michael,syrovatskii}.
To introduce it, let us consider a stationary solution of \eqref{mhd1}, \eqref{bc} with interface located at $\{y=0\}$ given by the constant states
\begin{equation}
\begin{array}{ll}\label{constant}
\u^\pm=(U^\pm,0)^T, \qquad \B^\pm=(B^\pm,0)^T
\end{array}
\end{equation}
in the $x$-direction.
The {necessary} and {sufficient stability condition} for the stationary solution is
\begin{equation}
\label{syrovatskii}
|U^+-U^-|^2 < 2 \, \Big( |B^+| ^2+ |B^-|^2 \Big) \, ,
\end{equation}
see \cite{axford,michael,syrovatskii}.
Equality in \eqref{syrovatskii} corresponds to the transition to
{\it violent} instability, i.e. {ill-posedness} of the linearized problem.

Let $U=(U^+,U^-), B=(B^+,B^-)$ and define
\[
\Delta(U,B):= \frac12 \, \Big( |B^+| ^2+ |B^-|^2 \Big) -\frac14|U^+-U^-|^2.
\]
According to \eqref{syrovatskii}, stability/instability occurs when $\Delta(U,B)\gtrless0$.

Hunter and Thoo investigated in \cite{hunter-thoo} the transition to instability when $\Delta(U,B)=0$.
Assume that $U^\pm,B^\pm$ depend on a small positive parameter $\eps$ and
\[
U^\pm=U^\pm_0+\eps U^\pm_1+O(\eps^2), \qquad
B^\pm=B^\pm_0+\eps B^\pm_1+O(\eps^2)
\]
as $\eps\to0^+$, where
\[
\Delta(U_0,B_0)=0\,.
\]
Then
\begin{equation}
\begin{array}{ll}\label{Delta}

\Delta(U,B)=\eps\mu+O(\eps^2)
\end{array}
\end{equation}
as $\eps\to0^+$, where
\[
\mu=B^+_0B^+_1+B^-_0B^-_1 - \frac12\left( U^+_0-U^-_0  \right)\left( U^+_1-U^-_1  \right)
\,.
\]
From \eqref{Delta}, $\mu$ plays the role of a scaled bifurcation parameter:
for small $\eps>0$, if $\mu>0$ the stationary solution \eqref{constant} is linearly stable,
while if $\mu<0$, it is linearly unstable.

It is proved in \cite{hunter-thoo} that the perturbed location of the interface  has the asymptotic expansion
\begin{equation*}
\begin{array}{ll}\label{}
y=f(t,x;\eps)=\eps\vphi( \tau,\th ) + O(\eps^{3/2})  \qquad \mbox{as } \eps\to0^+,
\end{array}
\end{equation*}
where $\tau=\eps^{1/2} t$ is a \lq\lq slow\rq\rq time variable and
$\th=x-\lambda_0 t$  is a new spatial variable in a reference frame moving with the surface wave, $\lambda_0=( U^+_0+U^-_0  )/2$.

As shown in \cite{hunter-thoo}, after a rescaling, and writing again $(t,x)$ for $(\tau,\th)$, the first order term $\vphi$ satisfies the quadratically nonlinear amplitude equation
\begin{equation}\label{onde_integro_diff}
\varphi_{tt}-\mu\varphi_{xx}=\left(\frac12\mathbb H[\phi^2]_{xx}+\phi\varphi_{xx}\right)_{\!\!x}\,,\qquad\phi=\mathbb H[\varphi]\,,
\end{equation}
where the unknown is the scalar function $\varphi=\varphi(t,x)$, whereas $\mathbb H$ denotes the Hilbert transform with respect to $x$.

\eqref{onde_integro_diff} is an integro-differential equation of order two: in fact, it may also be written as
\begin{equation*}
\begin{array}{ll}\label{equ1bis}
\vphi_{tt}-\mu\vphi_{xx} = \left(  [ \HH;\phi
]\partial_x  \phi
_{x} + \HH[\phi
^2_x]\right)_x \,,
\end{array}
\end{equation*}
where $[ \HH;\phi
]\partial_x$ is a pseudo-differential operator of order zero.
In \cite{hunter-thoo} the authors discuss the linearized well-posedness of \eqref{onde_integro_diff}. Linearizing the operator
\[
\mathbb L[\varphi]:=\varphi_{tt}-\mu\varphi_{xx}-\left( \left[\mathbb H  \,;\phi
\right]\phi
_{xx} + \mathbb H[\phi
^2_x] \right)_{x}\,
\]
about a given basic state $\vphi_0$ gives
\begin{equation}\label{linearizzata}
\mathbb{L}^\prime[\varphi_0]\varphi^\prime=\varphi^\prime_{tt} - {\left(\mu-2\phi
_{0,x}\right)}\varphi^\prime_{xx}\\
-2\left[\mathbb H\,;\phi
_{0,x}\right]\phi
^\prime_{xx}-2\mathbb H[\phi
_{0,xx}\phi
^\prime_x]\\
+\big(\left[\phi
^\prime;\mathbb H\right]\phi
_{0,xx}+{ \left[\phi
_0;\mathbb H\right]\phi
^\prime_{xx}\big)_x}\,,
\end{equation}
where
$\phi
'=\HH[\vphi'], \phi
_0=\HH[\vphi_0]$.
{Assume that the last term in \eqref{linearizzata} may be disregarded, even if of order 2 in $\vphi'$.}
When
\begin{equation}
\begin{array}{ll}\label{extstab2}
\mu-2\phi
_{0,x}<0
\end{array}
\end{equation}
the operator $\mathbb{L}^\prime[\varphi_0]$ is elliptic and \eqref{onde_integro_diff} is locally linearly ill-posed in any Sobolev space.
On the contrary, when
\begin{equation}
\begin{array}{ll}\label{extstab}

\mu-2\phi
_{0,x}>0

\end{array}
\end{equation}
the operator $\mathbb{L}^\prime[\varphi_0]$ is hyperbolic and \eqref{onde_integro_diff} is locally linearly well-posed. In a sense we can think of  \eqref{onde_integro_diff} as a nonlinear perturbation of the wave equation.
In \cite{hunter-thoo} the reader may also find a physical explanation of condition \eqref{extstab2} yielding the linearized ill-posedness, or alternatively \eqref{extstab} for well-posedness, which is given in terms of a longitudinal strain of the fluid along the discontinuity.

For simplicity, in the sequel it is assumed that $\varphi^\prime$ and $\varphi_0$ are periodic functions in $x$ (cf. Theorem \ref{nonlin_th} below). In this case, the periodicity of $\varphi_0$ implies that $\varphi_{0, x}$ has spatial mean equal to zero; since $\varphi_0$ and $\varphi_{0,x}$ are also real-valued then $\phi_{0,x}=\mathbb H[\varphi_{0, x}]$ is still real-valued with zero spatial mean (see the results collected in the next sections \ref{molt_fourier}, \ref{trasf_hilbert}). Therefore $\phi_{0, x}$ (if not identically zero\footnote{Because the spatial mean of $\varphi_{0}$ is zero, $\phi_{0, x}$ identically zero should imply that $\varphi_{0}$ is identically zero too.}) should attain either positive or negative values; consequently inequality \eqref{extstab} yields $\mu>0$ (providing {\it linear stability} of \eqref{onde_integro_diff}), while the opposite inequality implies $\mu<0$ (which gives {\it linear instability}). It is therefore somehow natural to regard \eqref{extstab} as a {\it stability condition}, under which we investigate the {\it nonlinear well-posedness} of the equation \eqref{onde_integro_diff}.

It is interesting to observe that the same quadratic operator of \eqref{onde_integro_diff} appears in the first order nonlocal amplitude equation
\begin{equation}\label{amplie}
\varphi_{t}+\frac12\mathbb H[\phi^2]_{xx}+\phi\varphi_{xx}=0\,,\qquad\phi=\mathbb H[\varphi]\,,
\end{equation}
for nonlinear Rayleigh waves  \cite{hamilton-et-al} and surface waves on current-vortex sheets and plasma-vacuum interfaces in incompressible MHD \cite{ali-hunter,ali-hunter-parker,secchi-quart}.
{Equation \eqref{amplie} is considered a canonical model equation for nonlinear surface wave solutions of hyperbolic conservation laws, analogous to the inviscid Burgers equation for bulk waves.}

\bigskip
\bigskip

%

In this paper we are mainly interested in the nonlinear well-posedness of \eqref{onde_integro_diff} under assumption \eqref{extstab}. More specifically, we will study the local-in-time existence of solutions to the initial value problem for  \eqref{onde_integro_diff}  with sufficiently smooth initial data
\begin{equation}\label{id}
\varphi_{\vert\,t=0}=\varphi^{(0)}\,,\qquad \partial_t\varphi_{\vert\,t=0}=\varphi^{(1)}\,,
\end{equation}
satisfying the following ``stability'' condition
\begin{equation*}\label{stability_nl}
\mu-2\phi^{(0)}_{x}>0\,,\qquad\phi^{(0)}:=\mathbb H[\varphi^{(0)}]\,,
\end{equation*}
which  must be understood as a {\it smallness} assumption on the size of the initial data $\varphi^{(0)}$ in \eqref{id}.

For the sake of convenience, in the paper the unknown $\varphi=\varphi(t,x)$ is a scalar function of  the time $t\in\mathbb R^+$  and the space variable $x$, ranging on the one-dimensional torus $\mathbb T$ (that is $\varphi$ is periodic in $x$). For all notation we refer to the following Section \ref{prbt}.
The main result of the paper is given by the following theorem.
\begin{theorem}\label{nonlin_th}
\begin{itemize}
\item[(1)] Assume that $\varphi^{(0)}\in H^{11}(\mathbb T)$, $\varphi^{(1)}\in H^{10}(\mathbb T)$ and
\begin{equation}\label{sign-cond}
\mu- 2\mathbb{H}[\varphi^{(0)}]_x\geq \delta>0 \quad {\rm in}\,\,\mathbb{T},
\end{equation}
with some positive constant $\delta$.
Then there exists $T>0$, depending only on $\Vert\varphi^{(0)}\Vert_{H^{11}(\mathbb T)}$, $\Vert\varphi^{(1)}\Vert_{H^{10}(\mathbb T)}$ and $\delta$, such that the initial value problem \eqref{onde_integro_diff}, \eqref{id} with initial data $\varphi^{(0)}$, $\varphi^{(1)}$ admits a unique solution $\varphi$ on $[0,T]$ satisfying
    \begin{equation*}
    \varphi\in L^2(0, T; H^{9}(\mathbb T))\cap H^1(0, T; H^{8}(\mathbb T))\cap H^2(0, T; H^{7}(\mathbb T))\,,
    \end{equation*}
\begin{equation*}
\mu- 2\mathbb{H}[\varphi]_x\geq \delta/2 \quad {\rm in}\,\,[0,T]\times\mathbb{T}\, .
    \end{equation*}
\item[(2)] If $\nu>10$ and $\varphi^{(0)}\in H^{\nu+1}(\mathbb T)$, $\varphi^{(1)}\in H^{\nu}(\mathbb T)$ satisfy condition \eqref{sign-cond} then the solution $\varphi$ of \eqref{onde_integro_diff}, \eqref{id} with initial data $\varphi^{(0)}$, $\varphi^{(1)}$, considered in the statement (1), satisfies
    \begin{equation*}
    \varphi\in L^2(0, T; H^{\nu-1}(\mathbb T))\cap H^1(0, T; H^{\nu-2}(\mathbb T))\cap H^2(0, T; H^{\nu-3}(\mathbb T))\,.
    \end{equation*}
\end{itemize}
\end{theorem}
\begin{remark}\label{rmk:4}
It is worth pointing out that the time interval of the existence of the solution $\varphi=\varphi(t,x)$ to the initial value problem \eqref{onde_integro_diff}, \eqref{id}, computed from the initial data by Theorem \ref{nonlin_th}, is the same in both the statements (1) and (2). Even though in (2) the initial data $\varphi^{(0)}$, $\varphi^{(1)}$ have an additional Sobolev regularity $\nu>10$, with respect to the minimal regularity that is required in (1), the final time $T$ depends on those data only through the lower order norms $\Vert\varphi^{(0)}\Vert_{H^{11}(\mathbb T)}$, $\Vert\varphi^{(1)}\Vert_{H^{10}(\mathbb T)}$, see Subsection \ref{proof-theorem}.
\end{remark}

The paper is organized as follows. After the following Section \ref{prbt} about notations and basic tools, in Section \ref{equazione_lineare} we study the linearized equation about a given reference state $\vphi_0$. The main result of this section is the basic energy estimate \eqref{stima_apriori_1} obtained by the standard approach of multiplication and integration by parts. Even if the approach is quite natural, in our opinion the result is not at all obvious. In fact, it follows from a very careful analysis, in particular the study of some critical terms containing derivative of higher order. Here we use in a crucial way some fine properties of the Hilbert transform and new commutator estimates.

In Section \ref{stima_tame} we prove a tame estimate in Sobolev spaces of any order for the solution $\vphi'$ to the linearized equation. The proof follows from the basic energy a priori estimate \eqref{stima_apriori_1} and suitable commutator estimates involving the Hilbert transform and derivatives of higher order. An important feature is the following one. In the basic energy a priori estimate \eqref{stima_apriori_1}, there is no loss of regularity in the sense that the $L^2$-norm of the source term $g=\mathbb{L}^\prime[\varphi_0]$ controls the $L^2$-norms of both $\vphi_t'$ and $\vphi_x'$ (as for the standard wave equation). Thus, one would hope to prove an estimate in Sobolev spaces of higher order again with no loss of derivative, and consequently to solve the nonlinear problem by a standard method like the implicit function theorem or the contraction principle.

Unfortunately, here we find a serious difficulty. Our tame estimate \eqref{stima_tame_1} still has no loss of regularity from  the source term $g$ to the solution, but it contains the loss of 2 spatial derivatives in the inversion of the operator $\mathbb{L}^\prime[\varphi_0]$,
from the given basic state $\varphi_0$ to $\varphi^\prime$.

For this reason, we can't apply a standard method for the resolution of the nonlinear problem. Instead, \eqref{onde_integro_diff} is solved by applying the Nash-Moser's theorem. This is done in Section \ref{nlpb} where we give an equivalent formulation of the Cauchy problem for \eqref{onde_integro_diff} as an abstract equation in a suitable functional setting, and verify all the assumptions needed for the application of Nash-Moser's theorem.

Finally, in \ref{stima_commutatore} we prove our commutator estimates involving the Hilbert transform and give other useful estimates. In \ref{sec_N-M}, for reader's convenience, we recall the assumptions and the statement of the Nash-Moser's theorem.
%


\section{Preliminary results and basic tools}\label{prbt}
\subsection{Notations}\label{not}
Throughout the whole paper, the partial derivative of a function $f(t,x)$ with respect to $t$ or $x$ will be denoted appending to the function the subscript $t$ or $x$ as
\begin{equation*}
f_t:=\frac{\partial f}{\partial t}\,,\qquad f_x:=\frac{\partial f}{\partial x}\,.
\end{equation*}
(The notations $\partial_t f$, $\partial_x f$ will be also used.) Higher order derivatives in $(t,x)$ will be denoted by the repeated indices; for instance $f_{tt}$ and $f_{tx}$ will stand respectively for second order derivatives of $f$ with respect to $t$ twice and $t$, $x$.

Let $\mathbb T$ denote the one-dimensional torus defined as
\begin{equation*}
\mathbb T:=\mathbb R/(2\pi\mathbb Z)\,,
\end{equation*}
that is the set of equivalence classes of real numbers with respect to the equivalence relation $\sim$ defined as
\begin{equation*}
x\sim y\qquad\mbox{if and only if}\qquad x-y\in 2\pi\mathbb Z\,.
\end{equation*}
It is customary to identify functions that are defined on $\mathbb T$ with $2\pi-$periodic functions on $\mathbb R$. According to this convention, it will be usual referring to $f:\mathbb T\rightarrow\mathbb C$ as a ``periodic function''.

All periodic functions $f:\mathbb T\rightarrow\mathbb C$ can be expanded in terms of Fourier series as
\begin{equation*}
f(x)=\frac{1}{2\pi}\sum\limits_{k\in\mathbb Z}\widehat{f}(k)e^{ikx}\,,
\end{equation*}
where $\left\{\widehat{f}(k)\right\}_{k\in\mathbb Z}$ are the Fourier coefficients defined by
\begin{equation}\label{coeff_fourier}
\widehat{f}(k):=\int_{\mathbb T}f(x)e^{-ikx}\,dx\,,\qquad k\in\mathbb Z\,.
\end{equation}
For $1\le p\le +\infty$, we denote by $L^p(\mathbb T)$ the usual Lebesgue space of exponent $p$ on $\mathbb T$, defined as the set of (equivalence classes of) measurable functions $f:\mathbb T\rightarrow\mathbb C$ such that the norm
\begin{equation*}
\Vert f\Vert_{L^p(\mathbb T)}:=
\begin{cases}\left(\int_{\mathbb T}\vert f(x)\vert^p\,dx\right)^{1/p}\,,\quad\mbox{if}\,\,p<+\infty\,\\  \mbox{ess sup}_{x\in\mathbb T}\vert f(x)\vert\,,\quad\mbox{if}\,\,p=+\infty
\end{cases}
\end{equation*}
is finite. We denote
\begin{equation*}
(f,g)_{L^2(\mathbb T)}:=\int_{\mathbb T}f(x)\overline{g(x)}\,dx
\end{equation*}
the inner product of two functions $f, g\in L^2(\mathbb T)$ ($\overline z$ denotes the conjugate of $z\in\mathbb C$).

For all $s\in\mathbb R$, $H^s(\mathbb T)$ will denote the Sobolev space of order $s$ on $\mathbb T$, defined to be the set of periodic functions\footnote{The word ``function'' is used here, and in the rest of the paper, in a wide sense. To be more precise, one should speak about ``periodic distributions'' on the torus, instead of ``periodic functions'', when dealing with real order Sobolev spaces. However, for the sake of simplicity, here we prefer to avoid the precise framework of distributions. We refer the reader to the monograph \cite{ruzhansky-turunen} for a thorough presentation of the periodic setting.} $f:\mathbb T\rightarrow\mathbb C$ such that
\begin{equation}\label{normaHs}
\Vert f\Vert_{H^s(\mathbb T)}^2:=\frac1{2\pi}\sum\limits_{k\in\mathbb Z}(1+|k|)^{2s}\vert\widehat{f}(k)\vert^2<+\infty\,.
\end{equation}
The function $\Vert\cdot\Vert_{H^s(\mathbb T)}$ defines a norm on $H^s(\mathbb T)$, associated to the inner product
\begin{equation*}
(f,g)_{H^s(\mathbb T)}:=\frac1{2\pi}\sum\limits_{k\in\mathbb Z}(1+|k|)^{2s}\widehat{f}(k)\overline{\widehat{g}(k)}\,,
\end{equation*}
which turns $H^s(\mathbb T)$ into a Hilbert space.

Because of the relation between differentiation and Fourier coefficients, it is obvious that when $s$ is a positive integer $H^s(\mathbb T)$ reduces to the space of periodic functions $f:\mathbb T\rightarrow\mathbb C$ such that
\begin{equation*}
\partial_x^k f\in L^2(\mathbb T)\,,\quad\mbox{for}\,\,0\le k\le s
\end{equation*}
and
\begin{equation*}
\sum\limits_{k=0}^s\left\Vert \partial_x^k f\right\Vert_{L^2(\mathbb T)}
\end{equation*}
defines a norm in $H^s(\mathbb T)$ equivalent to \eqref{normaHs}\footnote{Even though the functions $f$ involved here depend on $x\in \mathbb T$ alone, the partial derivative notation $\partial_x^k:=\partial_x\dots\partial_x$ ($k$ times) is used just in order to be consistent with the notations adopted in the subsequent sections, where functions will also depend on time.}.

In the following, we are mainly concerned with real-valued periodic functions $f:\mathbb T\rightarrow\mathbb R$ with zero spatial mean, that is such that
\begin{equation*}
\int_{\mathbb T}f(x)\,dx=0\,.
\end{equation*}
For such functions the Fourier coefficients \eqref{coeff_fourier} obey the additional constraints
\begin{equation}\label{coeff_fourier_condizioni}
\widehat{f}(0)=0\,,\qquad \overline{\widehat{f}(k)}=\widehat{f}(-k)\,,\,\,\forall\,k\in\mathbb Z\,.
\end{equation}
In view of \eqref{coeff_fourier_condizioni}, for zero mean periodic functions on $\mathbb T$ a norm in $H^s(\mathbb T)$ equivalent to \eqref{normaHs} is provided by
\begin{equation*}
\Vert f\Vert_{s}^2:=\frac1{2\pi}\sum\limits_{k\in\mathbb Z\setminus\{0\}}|k|^{2s}\vert\widehat{f}(k)\vert^2\,.
\end{equation*}
For $s\in\mathbb N$, from the well-known formula
\begin{equation*}
\widehat{\partial^s_xf}(k)=(ik)^s\widehat{f}(k)
\end{equation*}
and Parseval's identity it follows that
\begin{equation}\label{ident_norme}
\Vert f\Vert_{s}=\Vert\partial^s_x f\Vert_{L^2(\mathbb T)}\,.
\end{equation}

\vspace{.5cm}
Hereafter, we will deal with spaces of functions that depend even on time $t$. It will be convenient to regard real-valued functions $f=f(t,x)$, depending on time and space, as vector-valued functions of $t$ alone taking values in some Banach space $\mathcal X$ of functions depending on $x\in\mathbb T$. For technical reasons, the time variable $t$ will be allowed to run through the whole real line $\mathbb R_t$ (or even a real interval $(-\infty,T)$ for given $T>0$), and the integrability properties of functions in time will be referred to the {\it weighted} measure $e^{-\gamma t}\,dt$ on $\mathbb R_t$ (or $(-\infty,T)$), being $\gamma$ a real positive parameter.

For every Banach space $\mathcal X$ (with norm $\Vert\cdot\Vert_{\mathcal X}$) and $\gamma\ge 1$, we denote by $L^2_\gamma(\mathbb R; \mathcal X)$ the space of measurable functions $f:\mathbb R\rightarrow\mathcal X$, such that the real-valued function
\begin{equation*}
t\mapsto e^{-\gamma t}\Vert f(t)\Vert_{\mathcal X}=\Vert f_\gamma(t)\Vert_{\mathcal X}
\end{equation*}
is square integrable on $\mathbb R$; here it is set
\begin{equation}\label{f_gamma}
f_\gamma:=e^{-\gamma t}f\,.
\end{equation}
The space $L^2_\gamma(\mathbb R; \mathcal X)$ is provided with the norm
\begin{equation}\label{norma_l2gamma}
\Vert f\Vert_{L^2_\gamma(\mathbb R; \mathcal X)}^2:=\int_\mathbb R \Vert f_\gamma(t)\Vert_{\mathcal X}^2\,dt\,.
\end{equation}
In the sequel, the Banach space $\mathcal X$ will be always some Sobolev space $H^m(\mathbb T)$ of integer order $m$. When in particular $\mathcal X=L^2(\mathbb T)$ then the space $L^2_\gamma(\mathbb R; L^2(\mathbb T))$ reduces to be the set of functions $f=f(t,x)$, that are measurable both on $t$ and $x$ and such that $e^{-\gamma t}f$ is square integrable in $\mathbb R\times\mathbb T$; accordingly we set $L^2_\gamma(\mathbb R\times\mathbb T)=L^2_\gamma(\mathbb R; L^2(\mathbb T))$, and denote by $\Vert\cdot\Vert_{L^2_\gamma(\mathbb R\times\mathbb T)}$ the related norm in \eqref{norma_l2gamma}. For $\gamma=0$, the weighted spaces above reduce of course to the usual Lebesgue spaces that will be simply denoted as $L^2(\mathbb R;\mathcal X)$.

When $\mathcal X$ is an Hilbert space (that will be always the case in the following), being $\Vert\cdot\Vert_{\mathcal X}$ the norm associated to the natural inner product, then integration by parts yields that the following identity
\begin{equation*}
\Vert \partial_t f\Vert^2_{L^2_\gamma(\mathbb R;\mathcal X)}=\gamma^2\Vert f\Vert^2_{L^2_\gamma(\mathbb R; \mathcal X)}+\Vert\partial_t f_\gamma\Vert^2_{L^2(\mathbb R;\mathcal X)}
\end{equation*}
holds true for every $\gamma\ge 1$ and all sufficiently smooth functions $f$. In particular, it follows that
\begin{equation}\label{dis_gamma}
\gamma\Vert f\Vert_{L^2_\gamma(\mathbb R; \mathcal X)}\le \Vert\partial_t f\Vert_{L^2_\gamma(\mathbb R;\mathcal X)}\,.
\end{equation}
More in general for $k\in\mathbb N$ we set $H^k_\gamma(\mathbb R;\mathcal X)$ the space of measurable functions $f:\mathbb R\rightarrow\mathcal X$ such that
\begin{equation*}
\partial^j_t f\in L^2_\gamma(\mathbb R; \mathcal X)\,,\quad\mbox{for}\,\,j=0,\dots,k\,,
\end{equation*}
endowed with the weighted norm $\Vert\cdot\Vert_{H^k_\gamma(\mathbb R; \mathcal X)}$ defined by
\begin{equation}\label{normaHk}
\Vert f\Vert_{H^k_\gamma(\mathbb R; \mathcal X)}^2=\sum\limits_{j=0}^k\gamma^{2(k-j)}\Vert\partial^j_t f\Vert^2_{L^2_\gamma(\mathbb R; \mathcal X)}\,.
\end{equation}
Applying repeatedly the inequality \eqref{dis_gamma}, one can see that for every $k\in\mathbb N$, $\gamma\ge 1$ and $f\in H^k_\gamma(\mathbb R; \mathcal X)$ the following holds
\begin{equation*}
\Vert f\Vert_{H^k_\gamma(\mathbb R; \mathcal X)}\le \sqrt {k+1}\Vert\partial_t^k f\Vert_{L^2_\gamma(\mathbb R; \mathcal X)}\,.
\end{equation*}
The previous inequality shows that $f\mapsto \Vert \partial_t^k f\Vert_{L^2_\gamma(\mathbb R; \mathcal X)}$ defines a norm in $H^k_\gamma(\mathbb R; \mathcal X)$ equivalent (uniformly in $\gamma$) to the norm $\Vert\cdot\Vert_{H^k_\gamma(\mathbb R; \mathcal X)}$ defined in \eqref{normaHk}.

We denote by $L^\infty(\mathbb R; \mathcal X)$ the space of measurable functions $f:\mathbb R\rightarrow \mathcal X$ such that
$$
t\mapsto\Vert f(t)\Vert_{\mathcal X}
$$
is bounded in $\mathbb R$, provided with the norm
\begin{equation*}
\Vert f\Vert_{L^\infty(\mathbb R;\mathcal X)}:={\rm ess\,sup}_{t\in\mathbb R}\Vert f(t)\Vert_{\mathcal X}\,.
\end{equation*}
More in general, for $k\in\mathbb N$, $W^{k,\infty}(\mathbb R; \mathcal X)$ will be the space of measurable functions $f:\mathbb R\rightarrow \mathcal X$ such that
\begin{equation*}
\partial^j_t f\in L^\infty(\mathbb R; \mathcal X)\,,\quad\mbox{for}\,\,j=0,\dots,k\,,
\end{equation*}
with the norm
\begin{equation*}
\Vert f\Vert_{W^{k,\infty}(\mathbb R; \mathcal X)}:=\sup\limits_{0\le j\le k}\Vert\partial^j_t f\Vert_{L^\infty(\mathbb R; \mathcal X)}\,.
\end{equation*}
By replacing the real line with the real interval $(-\infty,T)$ for $T>0$ in all the above definitions, the vector-valued spaces $H^k_\gamma(-\infty,T;\mathcal X)$, $W^{k,\infty}(-\infty,T;\mathcal X)$ can be defined exactly in the same manner as before, with similar properties.
We recall the following vector-valued counterpart of the standard Sobolev Imbedding Theorem (see \cite{morando-trebeschi2013}).
\begin{lemma}\label{sobolev-imb}
For any $T>0$ and $\gamma\ge 1$, $H^1_\gamma(-\infty,T; \mathcal X)\hookrightarrow L^\infty(-\infty,T; \mathcal X)$ and there exists a positive constant $C_T$, independent of $\gamma$, such that
\begin{equation*}
\Vert u\Vert_{L^\infty(-\infty, T;\mathcal X)}\le\frac{C_T}{\sqrt{\gamma}}\Vert u\Vert_{H^1_\gamma(-\infty, T; \mathcal X)}\,,\quad\forall\,u\in H^1_\gamma(-\infty,T;\mathcal X)\,.
\end{equation*}
\end{lemma}
\subsection{Some reminds on periodic Fourier multipliers}\label{molt_fourier}
For a given sequence of real (or complex) numbers $\{A(k)\}_{k\in\mathbb Z}$, we denote by $A$ the linear operator defined on periodic functions $f:\mathbb T\rightarrow\mathbb C$ by setting
\begin{equation}\label{operatore1}
Af(x):=\frac{1}{2\pi}\sum\limits_{k\in\mathbb Z}A(k)\widehat{f}(k)e^{ikx}\,,\qquad x\in\mathbb T\,,
\end{equation}
or equivalently, on the Fourier side, by its Fourier coefficients
\begin{equation}\label{operatore2}
\widehat{Af}(k)=A(k)\widehat{f}(k)\,,\qquad\forall\,k\in\mathbb Z\,.
\end{equation}
We refer to the sequence $\{A(k)\}_{k\in\mathbb Z}$ as the {\it symbol} of the operator $A$.

The following continuity result will be useful in the sequel.
\begin{proposition}\label{prop_molt_sobolev}
Let the sequence  $\{A(k)\}_{k\in\mathbb Z}$ satisfy the following assumption
\begin{equation}\label{stima_simbolo}
(1+\vert k\vert)^{-m}\vert A(k)\vert\le C\,,\qquad\forall\,k\in\mathbb Z\,,
\end{equation}
with suitable constants $m\in\mathbb R$, $C>0$; then the operator $A$ with symbol $\{A(k)\}_{k\in\mathbb Z}$, defined by \eqref{operatore1}, extends as a linear bounded operator
\begin{equation*}
A:H^{s}(\mathbb T)\rightarrow H^{s-m}(\mathbb T)\,,
\end{equation*}
for all $s\in\mathbb R$; more precisely
\begin{equation*}
\Vert Af\Vert_{H^{s-m}(\mathbb T)}\le C\Vert f\Vert_{H^{s}(\mathbb T)}\,,\qquad\forall\,f\in H^{s}(\mathbb T)\,,
\end{equation*}
where $C$ is the same constant involved in \eqref{stima_simbolo}.
\end{proposition}
\begin{proof}
The proof is a straightforward application of the estimates \eqref{stima_simbolo} and the definition of the Sobolev norm in \eqref{normaHs}; indeed we compute (see \eqref{operatore2}):
\begin{align*}
\Vert Af\Vert^2_{H^{s-m}(\mathbb T)}= &\frac1{2\pi}\sum\limits_{k\in\mathbb Z}(1+\vert k\vert)^{2(s-m)}\vert\widehat{Af}(k)\vert^2=\frac1{2\pi}\sum\limits_{k\in\mathbb Z}(1+\vert k\vert)^{2(s-m)}\vert A(k)\widehat{f}(k)\vert^2\\
\\
&\le \frac{C^2}{2\pi}\sum\limits_{k\in\mathbb Z}(1+\vert k\vert)^{2s}\vert\widehat{f}(k)\vert^2=C^2\Vert f\Vert^2_{H^s(\mathbb T)}\,.
\end{align*}
\end{proof}
We will refer to an operator $A$, under the assumptions of Proposition \ref{prop_molt_sobolev}, as a {\it Fourier multiplier} of order $m$. Such an operator transforms periodic functions with mean zero into functions of the same type, as it is easily seen by observing that
$\widehat{Af}(0)=A(0)\widehat{f}(0)=0,
$ as long as $\widehat f(0)=0$.
\begin{remark}\label{remark_prodotto}
As a straightforward consequence of formulas \eqref{operatore1}, \eqref{operatore2}, it even follows that the composition $AB$ of two Fourier multipliers $A$ and $B$, whose symbols are respectively  $\{A(k)\}_{k\in\mathbb Z}$ and  $\{B(k)\}_{k\in\mathbb Z}$, is again a Fourier multiplier whose symbol is given by  $\{A(k)B(k)\}_{k\in\mathbb Z}$ (the order of $AB$ being the sum of the orders of $A$ and $B$ separately, because of \eqref{stima_simbolo}). We have in particular that $AB=BA$.
\end{remark}

An example of a Fourier multiplier of order one is provided by the $x-$derivative, i.e. $Af=f_x$, since indeed
\begin{equation*}
\widehat{Af}(k)=\widehat{f_x}(k)=ik\widehat{f}(k)\,,\qquad\forall\,k\in\mathbb Z\,.
\end{equation*}
Another relevant example of a Fourier multiplier is considered in the next section.
\subsection{Discrete Hilbert transform}\label{trasf_hilbert}
The discrete Hilbert transform of a periodic function $f:\mathbb T\rightarrow\mathbb C$, denoted by $\mathbb H[f]$, is defined on the Fourier side by setting
\begin{equation}\label{hilbert1}
\widehat{\mathbb H[f]}(k)=-i\,{\rm sgn}\,k\widehat{f}(k)\,,\qquad\forall\,k\in\mathbb Z\,,
\end{equation}
where
\begin{equation}\label{segno}
{\rm sgn}\,k:=
\begin{cases}
1\,,\quad\mbox{if}\,\,k>0\,,\\ 0\,,\quad\mbox{if}\,\,k=0\,,\\ -1\,,\quad\mbox{if}\,\,k<0\,.
\end{cases}
\end{equation}
It is clear that, in view of Proposition \ref{prop_molt_sobolev}, the Hilbert transform provides a Fourier multiplier of order zero, the condition \eqref{stima_simbolo} being satisfied by $A(k)=-i\,{\rm sgn}\,k$ with $m=0$ and $C=1$; then after Proposition \ref{prop_molt_sobolev} we conclude that
\begin{equation*}
\mathbb H:H^s(\mathbb T)\rightarrow H^s(\mathbb T)
\end{equation*}
is a linear bounded operator and
\begin{equation}\label{stima_hilbert}
\Vert\mathbb H[f]\Vert_{H^s(\mathbb T)}\le\Vert f\Vert_{H^s(\mathbb T)}\,,\qquad\forall\,f\in H^s(\mathbb T)
\end{equation}
for all $s\in\mathbb R$.

Since $\vert{\rm sgn}\,k\vert=1$ for $k\neq 0$, it holds in particular that for every periodic function $f\in H^s(\mathbb T)$ with mean zero (i.e. $\widehat{f}(0)=0$), one has
\begin{equation*}
\Vert\mathbb H[f]\Vert^2_{s}=\frac1{2\pi}\sum\limits_{k\in\mathbb Z\setminus\{0\}}\vert k\vert^{2s}\vert\widehat{\mathbb H[f]}(k)\vert^2=\frac1{2\pi}\sum\limits_{k\in\mathbb Z\setminus\{0\}}\vert k\vert^{2s}\vert-i\,{\rm sgn}\,k\widehat{f}(k)\vert^2=\frac1{2\pi}\sum\limits_{k\in\mathbb Z\setminus\{0\}}\vert k\vert^{2s}\vert\widehat{f}(k)\vert^2=\Vert f\Vert^2_{s}\,.
\end{equation*}
Here below we collect a few elementary properties of the Hilbert transform that will be useful in the sequel.
\begin{itemize}
\item[1.] The Hilbert transform commutes with the $x-$derivative. It is a particular case of the property recalled in Remark \ref{remark_prodotto};
\item[2.] For all periodic functions $f,g:\mathbb T\rightarrow\mathbb C$ there holds
\begin{equation}\label{prodotto_hilbert}
\mathbb H\left[fg-\mathbb H[f]\mathbb H[g]\right]=f\mathbb H[g]+\mathbb H[f]g\,.
\end{equation}
\item[3.] For every periodic function $f:\mathbb T\rightarrow\mathbb C$, with zero mean, and $k\in\mathbb Z$, the following formulas of calculus hold true \footnote{Notice that, according to the convention ${\rm sgn}\,0=0$ (see \eqref{segno}), the Hilbert transform $\mathbb H[f]$ of any periodic function $f$ on $\mathbb T$ has zero mean. Hence, the first formula in \eqref{calcolo} is not true when the mean of  $f$ is different from zero. In the latter case, that formula should be replaced by $\mathbb H^2[f]=-f +\widehat{f}(0)$.}
\begin{equation}\label{calcolo}
\mathbb H^2[f]=-f\,,\qquad \mathbb H\left[e^{ik\cdot}\right](x)=-i\,{\rm sgn}\,k\,e^{ikx}\,.
\end{equation}
\item[4.] For all periodic functions $f,g\in L^2(\mathbb T)$ there holds
\begin{equation}\label{integraleH}
\left(\mathbb H[f], g\right)_{L^2(\mathbb T)}=\left(f, -\mathbb H[g]\right)_{L^2(\mathbb T)}.
\end{equation}
\item[5.] For all periodic functions $f,g\in L^2(\mathbb T)$ and $h\in L^\infty(\mathbb T)$ there holds
\begin{equation}\label{aggiunto}
\left(\left[h;\mathbb H\right]f , g\right)_{L^2(\mathbb T)}=\left(f ,\left[h;\mathbb H\right]g\right)_{L^2(\mathbb T)},
\end{equation}
where $\left[h;\mathbb H\right]$ denotes the commutator between the multiplication by the function $h$ and the Hilbert transform $\mathbb H$.
\end{itemize}
\section{The linearized equation}\label{equazione_lineare}
It is well-known that a first step in proving the local-in-time existence for the nonlinear equation \eqref{onde_integro_diff} is the study of the well-posedness of the linearization of this equation about a sufficiently smooth state $\varphi_0=\varphi_0(t,x)$; according to \eqref{extstab} we assume that the reference state $\varphi_0$ satisfies
\begin{equation*}
\mu-2\phi_{0,x}>0\,,\qquad\phi_0:=\mathbb H[\varphi_0]\,.
\end{equation*}
The above condition ensures the ``leading part'' $\varphi_{tt}-(\mu-2\phi_{0,x})\varphi_{xx}$ of the linearized equation to be of hyperbolic type (see \eqref{equazione_lin1}).

In order to linearize the equation \eqref{onde_integro_diff}, it is firstly convenient to rewrite it in the following equivalent form
\begin{equation}\label{onde_1}
\varphi_{tt}-\mu\varphi_{xx}=\left(\mathbb H[\phi^2_x]-\left[\phi\,;\mathbb H\right]\phi_{xx}\right)_{x}\,,
\end{equation}
where
\begin{equation*}
\left[\phi\,;\mathbb H\right]:=\phi\mathbb H-\mathbb H\phi
\end{equation*}
denotes the commutator of the multiplication operator by $\phi$ and the Hilbert transform $\mathbb H$. Here we have used the first formula in \eqref{calcolo} (it is assumed that $\varphi$ has zero spatial mean), which implies $\varphi_{xx}=-\mathbb H[\phi_{xx}]$, and that $\mathbb H$ commutes with the $x-$differentiation.

Let $\varphi_0=\varphi_0(t,x)$ be a given basic state, with zero spatial mean, obeying suitable regularity assumptions and set $\phi_0=\mathbb H[\varphi_0]$. If we let $\mathbb L[\cdot]$ denote the nonlinear operator
\begin{equation}\label{operatore_nonlin}
\mathbb L[\varphi]:=\varphi_{tt}-\mu\varphi_{xx}-\left(\mathbb H[\phi^2_x]-\left[\phi\,;\mathbb H\right]\phi_{xx}\right)_{x}\,,
\end{equation}
then the linearization of the equation \eqref{onde_1} about $\varphi_0$ is defined by
\begin{equation}\label{equazione_lin}
\mathbb{L}^\prime[\varphi_0]\varphi^\prime:=\frac{d}{d\varepsilon}\left.\mathbb L[\varphi_0+\varepsilon\varphi^\prime]\right\vert_{\varepsilon=0}=g\,.
\end{equation}
In the linear equation above the unknown $\varphi^\prime=\varphi^\prime(t,x)$ represents some small perturbation of the basic state $\varphi_0$, with zero spatial mean, while $g=g(t,x)$ is some given nonzero forcing term taking account of lower order perturbation errors arising from the linearization of \eqref{onde_1}.

An explicit computation leads to the following form of the linear operator $\mathbb{L}^\prime[\varphi_0]$ in \eqref{equazione_lin}.
\begin{equation}\label{operatore_lin1}
\mathbb{L}^\prime[\varphi_0]\varphi^\prime=\varphi^\prime_{tt}-\mu\varphi^\prime_{xx}-\left(2\mathbb H[\phi_{0,x}\phi^\prime_x]-\left[\phi^\prime;\mathbb H\right]\phi_{0,xx}-\left[\phi_0;\mathbb H\right]\phi^\prime_{xx}\right)_x\,,\qquad \phi^\prime:=\mathbb H[\varphi^\prime]\,,
\end{equation}
where again $[A\,;B]:=AB-BA$ is the commutator of the operators $A$ and $B$.
\subsection{$L^2-$a priori estimate}\label{stima_L2}
Following the approach developed in \cite{benzoni-serre}, our first goal is to associate to the linear equation \eqref{equazione_lin} an $L^2$-a priori energy estimate of the first order derivatives of any sufficiently smooth solution of such an equation by the forcing term $g$. In this context, the time $t$ will be allowed to span the whole real line $\mathbb R$, and the energy estimate we are looking for will be of {\it weighted type}, in the sense that the integrability of the involved functions will be measured in the weighted space $L^2_\gamma(\mathbb R\times\mathbb T)$.

In the following, for a real $\gamma>0$ and every function $f=f(t,x)$ on $\mathbb R\times\mathbb T$, according to \eqref{f_gamma}, we use the shortcut
\begin{equation}\label{shortcut}
f_\gamma:=e^{-\gamma t}f\,,\qquad f_{t,\gamma}:=e^{-\gamma t}f_t\,,\qquad f_{x,\gamma}:=e^{-\gamma t}f_x\,.
\end{equation}
For the reader convenience, let us recall here below the following useful, though trivial, identities
\begin{equation}\label{identita_1}
f_{x,\gamma}=\partial_x f_\gamma\,,\qquad f_{t,\gamma}=\gamma f_\gamma+\partial_t f_\gamma\,.
\end{equation}
The relation
\begin{equation}\label{identita}
\partial_x f_{t,\gamma}=\gamma f_{x,\gamma}+\partial_t f_{x,\gamma}
\end{equation}
follows at once from a combination of \eqref{shortcut} and \eqref{identita_1}; the latter will be repeatedly used later on.

From the second equality in \eqref{identita_1} and the integration by parts in $t$, the following estimate can also be proved, see \cite{metivier2} for the proof:
\begin{equation}\label{stima_metivier}
\gamma\Vert f\Vert_{L^2_\gamma(\mathbb R\times\mathbb T)}\le\Vert f_{t}\Vert_{L^2_\gamma(\mathbb R\times\mathbb T)}\,.
\end{equation}
Actually, inequality \eqref{stima_metivier} is a particular case of \eqref{dis_gamma} for $\mathcal{X}=L^2(\mathbb T)$.

This section is devoted to the proof of the following result.
\begin{theorem}\label{teorema_1}
Let the basic state $\varphi_0:\mathbb R\times\mathbb T\rightarrow\mathbb R$, with zero spatial mean, satisfy
\begin{equation}\label{regolarita_phi0}
\varphi_0\in L^\infty(\mathbb R; H^3(\mathbb T))\,,\quad\varphi_{0,t}\in L^\infty(\mathbb R; H^2(\mathbb T))
\end{equation}
and
\begin{equation}\label{stability_unif}
\mu-2\phi_{0,x}\ge\delta/2\quad\mbox{in}\,\,\,\mathbb R\times\mathbb T\,,
\end{equation}
with $\delta>0$ assigned in Theorem \ref{nonlin_th}. Then there exist constants $\gamma_0\ge 1$ depending only on $\delta$ and $\varphi_0$ through the norms $\Vert\varphi_{0}\Vert_{L^\infty(\mathbb R;H^3(\mathbb T))}$, $\Vert\varphi_{0,t}\Vert_{L^\infty(\mathbb R; H^2(\mathbb T))}$, and $C_0>0$ depending only on $\delta$, such that for all $\gamma\ge \gamma_0$ and every sufficiently smooth function $\varphi^\prime:\mathbb R\times\mathbb T\rightarrow\mathbb R$, with zero spatial mean, the following a priori estimate
\begin{equation}\label{stima_apriori_1}
\gamma\left\{\Vert\varphi^\prime_{t}\Vert^2_{L^2_\gamma(\mathbb R\times\mathbb T)}+\Vert\varphi^\prime_{x}\Vert^2_{L^2_\gamma(\mathbb R\times\mathbb T)}\right\}\le\frac{C_0}{\gamma}\Vert g\Vert^2_{L^2_\gamma(\mathbb R\times\mathbb T)}
\end{equation}
is satisfied, where $g:=\mathbb{L}^\prime[\varphi_0]\varphi^\prime$.
\end{theorem}
\begin{proof}
Let $\varphi^\prime(t,x)$ be a sufficiently smooth function, with zero spatial mean, according to the statement of Theorem \ref{teorema_1} and set $g:=\mathbb{L}^\prime[\varphi_0]\varphi^\prime$. To simplify the notation, in the following we drop the superscript ${}^\prime$ in the unknown function $\varphi^\prime$.
\newline
Just in order to outline the quantity $\mu-2\phi_{0,x}$, involved in \eqref{stability_unif}, as coefficient of $\varphi_{xx}$, it is convenient to expand the last term in the right-hand side of \eqref{operatore_lin1}; by the rules of calculus collected in Section \ref{trasf_hilbert} we get
\begin{align*}
\left(2\mathbb H\right.&\left.[\phi_{0,x}\phi_x]-\left[\phi;\mathbb H\right]\phi_{0,xx}-\left[\phi_0;\mathbb H\right]\phi_{xx}\right)_x\\
&=2\mathbb H[\phi_{0,x}\phi_{xx}]+2\mathbb H[\phi_{0,xx}\phi_x]-\left(\left[\phi;\mathbb H\right]\phi_{0,xx}+\left[\phi_0;\mathbb H\right]\phi_{xx}\right)_x\\
&=2\phi_{0,x}\mathbb H[\phi_{xx}]+2\left[\mathbb H\,;\phi_{0,x}\right]\phi_{xx}+2\mathbb H[\phi_{0,xx}\phi_x]-\left(\left[\phi;\mathbb H\right]\phi_{0,xx}+\left[\phi_0;\mathbb H\right]\phi_{xx}\right)_x\\
&=-2\phi_{0,x}\varphi_{xx}+2\left[\mathbb H\,;\phi_{0,x}\right]\phi_{xx}+2\mathbb H[\phi_{0,xx}\phi_x]-\left(\left[\phi;\mathbb H\right]\phi_{0,xx}+\left[\phi_0;\mathbb H\right]\phi_{xx}\right)_x\,.
\end{align*}
Then substituting the last expression into \eqref{operatore_lin1} gives
\begin{equation}\label{operatore_lin2}
\mathbb{L}[\varphi_0]\varphi=\varphi_{tt}-\left(\mu-2\phi_{0,x}\right)\varphi_{xx}-2\left[\mathbb H\,;\phi_{0,x}\right]\phi_{xx}-2\mathbb H[\phi_{0,xx}\phi_x]+\left(\left[\phi;\mathbb H\right]\phi_{0,xx}+\left[\phi_0;\mathbb H\right]\phi_{xx}\right)_x\,.
\end{equation}
In view of \eqref{operatore_lin2}, the linearized equation \eqref{equazione_lin} takes the form
\begin{equation}\label{equazione_lin1}
\varphi_{tt}-\left(\mu-2\phi_{0,x}\right)\varphi_{xx}=2\left[\mathbb H\,;\phi_{0,x}\right]\phi_{xx}+2\mathbb H[\phi_{0,xx}\phi_x]-\left(\left[\phi;\mathbb H\right]\phi_{0,xx}+\left[\phi_0;\mathbb H\right]\phi_{xx}\right)_x+g\,.
\end{equation}
Since the estimate \eqref{stima_apriori_1} involves the weighted $L^2-$norms of $\varphi_t$, $\varphi_x$, it is also convenient to restate the equation \eqref{equazione_lin1} in terms of $\varphi_{t,\gamma}$, $\varphi_{x,\gamma}$. Then \eqref{equazione_lin1} becomes equivalent to
\begin{equation}\label{equazione_lin1_w}
\begin{split}
\gamma\varphi_{t,\gamma}&+\partial_t\varphi_{t,\gamma}-\left(\mu-2\phi_{0,x}\right)\partial_x\varphi_{x,\gamma}=2\left[\mathbb H\,;\phi_{0,x}\right]\partial_x\phi_{x,\gamma}+2\mathbb H[\phi_{0,xx}\phi_{x,\gamma}]\\
&-\left(\left[\phi_\gamma;\mathbb H\right]\phi_{0,xx}+\left[\phi_0;\mathbb H\right]\partial_x\phi_{x,\gamma}\right)_x+g_\gamma\,.
\end{split}
\end{equation}
We multiply by $\varphi_{t,\gamma}$ the equation \eqref{equazione_lin1_w} and integrate over $\mathbb R\times\mathbb T$ to get
\begin{equation}\label{equazione_lin2}
\begin{split}
\gamma\Vert\varphi_{t}\Vert^2_{L^2_\gamma(\mathbb R\times\mathbb T)}&+\int_{\mathbb R\times\mathbb T}\partial_t\varphi_{t,\gamma}\varphi_{t,\gamma}\,dx\,dt-\int_{\mathbb R\times\mathbb T}\left(\mu-2\phi_{0,x}\right)\partial_x\varphi_{x,\gamma}\varphi_{t,\gamma}\,dx\,dt\\
&=2\int_{\mathbb R\times\mathbb T}\left[\mathbb H\,;\phi_{0,x}\right]\partial_x\phi_{x,\gamma}\varphi_{t,\gamma}\,dx\,dt+2\int_{\mathbb R\times\mathbb T}\mathbb H[\phi_{0,xx}\phi_{x,\gamma}]\varphi_{t,\gamma}\,dx\,dt\\
&-\int_{\mathbb R\times\mathbb T}\left(\left[\phi_\gamma;\mathbb H\right]\phi_{0,xx}+\left[\phi_0;\mathbb H\right]\partial_x\phi_{x,\gamma}\right)_x\varphi_{t,\gamma}\,dx\,dt+\int_{\mathbb R\times\mathbb T}g_\gamma\varphi_{t,\gamma}\,dx\,dt\,.
\end{split}
\end{equation}
As for the two integrals in the left-hand side, integration by parts gives
\begin{equation}\label{int_1.1}
\int_{\mathbb R\times\mathbb T}\partial_t\varphi_{t,\gamma}\varphi_{t,\gamma}\,dx\,dt=0;
\end{equation}
\begin{equation}\label{int_1.2}
\begin{split}
-\int_{\mathbb R\times\mathbb T}&\left(\mu-2\phi_{0,x}\right)\partial_x\varphi_{x,\gamma}\varphi_{t,\gamma}\,dx\,dt=\int_{\mathbb R\times\mathbb T}\left(\left(\mu-2\phi_{0,x}\right)\varphi_{t,\gamma}\right)_x\varphi_{x,\gamma}\,dx\,dt\\
&
=\int_{\mathbb R\times\mathbb T}\left(\mu-2\phi_{0,x}\right)\partial_x\varphi_{t,\gamma}\varphi_{x,\gamma}\,dx\,dt-\int_{\mathbb R\times\mathbb T}2\phi_{0,xx}\varphi_{t,\gamma}\varphi_{x,\gamma}\,dx\,dt\,.
\end{split}
\end{equation}
We use the identity \eqref{identita}, with $f=\varphi$, in \eqref{int_1.2} to get
\begin{equation}\label{int_1.2.1}
\begin{split}
-\int_{\mathbb R\times\mathbb T}&\left(\mu-2\phi_{0,x}\right)\partial_x\varphi_{x,\gamma}\varphi_{t,\gamma}\,dx\,dt\\
&=\gamma\int_{\mathbb R\times\mathbb T}\left(\mu-2\phi_{0,x}\right)\vert\varphi_{x,\gamma}\vert^2\,dx\,dt+\int_{\mathbb R\times\mathbb T}\left(\mu-2\phi_{0,x}\right)\partial_t\varphi_{x,\gamma}\varphi_{x,\gamma}-2\int_{\mathbb R\times\mathbb T}\phi_{0,xx}\varphi_{t,\gamma}\varphi_{x,\gamma}\,dx\,dt\,.
\end{split}
\end{equation}
As for the second integral in the right-hand side of \eqref{int_1.2.1}, integration by parts in $t$ gives
\begin{equation*}
\int_{\mathbb R\times\mathbb T}\left(\mu-2\phi_{0,x}\right)\partial_t\varphi_{x,\gamma}\varphi_{x,\gamma}=\int_{\mathbb R\times\mathbb T}\phi_{0,xt}\vert\varphi_{x,\gamma}\vert^2\,dx\,dt\,.
\end{equation*}
Replacing the latter into \eqref{int_1.2.1} then gives
\begin{equation}\label{int_1.2.2}
\begin{split}
-\int_{\mathbb R\times\mathbb T}&\left(\mu-2\phi_{0,x}\right)\partial_x\varphi_{x,\gamma}\varphi_{t,\gamma}\,dx\,dt\\
&=\gamma\int_{\mathbb R\times\mathbb T}\left(\mu-2\phi_{0,x}\right)\vert\varphi_{x,\gamma}\vert^2\,dx\,dt+\int_{\mathbb R\times\mathbb T}\phi_{0,xt}\vert\varphi_{x,\gamma}\vert^2\,dx\,dt-2\int_{\mathbb R\times\mathbb T}\phi_{0,xx}\varphi_{t,\gamma}\varphi_{x,\gamma}\,dx\,dt\,.
\end{split}
\end{equation}
Replacing \eqref{int_1.1}, \eqref{int_1.2.2} into \eqref{equazione_lin2} we get
\begin{equation}\label{equazione_lin3}
\begin{split}
\gamma\Vert\varphi_{t}&\Vert^2_{L^2_\gamma(\mathbb R\times\mathbb T)}+\gamma\int_{\mathbb R\times\mathbb T}\left(\mu-2\phi_{0,x}\right)\vert\varphi_{x,\gamma}\vert^2\,dx\,dt\\
&=-\int_{\mathbb R\times\mathbb T}\phi_{0,xt}\vert\varphi_{x,\gamma}\vert^2\,dx\,dt+2\int_{\mathbb R\times\mathbb T}\phi_{0,xx}\varphi_{t,\gamma}\varphi_{x,\gamma}\,dx\,dt\\
&+2\int_{\mathbb R\times\mathbb T}\left[\mathbb H\,;\phi_{0,x}\right]\partial_x\phi_{x,\gamma}\varphi_{t,\gamma}\,dx\,dt+2\int_{\mathbb R\times\mathbb T}\mathbb H[\phi_{0,xx}\phi_{x,\gamma}]\varphi_{t,\gamma}\,dx\,dt\\
&-\int_{\mathbb R\times\mathbb T}\left(\left[\phi_\gamma;\mathbb H\right]\phi_{0,xx}+\left[\phi_0;\mathbb H\right]\partial_x\phi_{x,\gamma}\right)_x\varphi_{t,\gamma}\,dx\,dt+\int_{\mathbb R\times\mathbb T}g_\gamma\varphi_{t,\gamma}\,dx\,dt\,.
\end{split}
\end{equation}
Now we are going to provide a suitable estimate for each of the integral terms that appear in the right-hand side of the identity above.
Throughout the following, $C$ will always denote some numerical positive constant that may be possibly different from
line to line.
\newline
Combining the Sobolev imbedding $H^1(\mathbb T)\hookrightarrow L^\infty(\mathbb T)$ and Poincar\'{e}'s inequality for functions with zero spatial mean, the following functional inequalities
\begin{equation}\label{sobolev-poincare}
\Vert\psi\Vert_{L^\infty(\mathbb T)}\le C\Vert\psi\Vert_{H^1(\mathbb T)}\le C\Vert\psi_x\Vert_{L^2(\mathbb T)}\,,
\end{equation}
can be easily established for all periodic functions $\psi\in H^1(\mathbb T)$ with spatial mean equal to zero. They will be repeatedly used in the following calculations.
\newline
Let us come back to the estimate of the right-hand side of \eqref{equazione_lin3}. H\"{o}lder's inequality gives
\begin{equation*}
-\int_{\mathbb R\times\mathbb T}\phi_{0,xt}\vert\varphi_{x,\gamma}\vert^2\,dx\,dt\le \int_{\mathbb R}\Vert \phi_{0,xt}\Vert_{L^\infty(\mathbb T)}\Vert\varphi_{x,\gamma}\Vert^2_{L^2(\mathbb T)}\,dt\,;
\end{equation*}
then we apply \eqref{sobolev-poincare} to $\phi_{0,xt}$ (recall that $\varphi_0$ and $\phi_0=\mathbb H[\varphi_0]$, as well as all their derivatives in $x$ and $t$, have zero spatial mean) and the $L^2-$continuity of the Hilbert transform (see \eqref{stima_hilbert} for $s=0$) to find
\begin{equation}\label{stima_int_2.1'}
-\int_{\mathbb R\times\mathbb T}\phi_{0,xt}\vert\varphi_{x,\gamma}\vert^2\,dx\,dt\le C\int_{\mathbb R}\Vert \varphi_{0,xxt}\Vert_{L^2(\mathbb T)}\Vert\varphi_{x,\gamma}\Vert^2_{L^2(\mathbb T)}\,dt\le C\Vert \varphi_{0,t}\Vert_{L^\infty(\mathbb R;H^2(\mathbb T))}\Vert\varphi_{x}\Vert^2_{L^2_\gamma(\mathbb R\times\mathbb T)}\,.
\end{equation}
Similarly, by H\"{o}lder and Young's inequalities and making use of \eqref{sobolev-poincare} and the continuity of the Hilbert transform, one gets
\begin{equation}\label{stima_int_2.2}
\begin{split}
2\int_{\mathbb R\times\mathbb T}&\phi_{0,xx}\varphi_{t,\gamma}\varphi_{x,\gamma}\,dx\,dt\le 2\int_{\mathbb R}\Vert\phi_{0,xx}\Vert_{L^\infty(\mathbb T)}\Vert\varphi_{t,\gamma}\Vert_{L^2(\mathbb T)}\Vert\varphi_{x,\gamma}\Vert_{L^2(\mathbb T)}\,dt\\
&\le C\int_\mathbb R\Vert\varphi_{0,xxx}\Vert_{L^2(\mathbb T)}\left\{\Vert\varphi_{t,\gamma}\Vert_{L^2(\mathbb T)}^2+\Vert\varphi_{x,\gamma}\Vert_{L^2(\mathbb T)}^2\right\}\,dt\\
&\le C\Vert\varphi_{0}\Vert_{L^\infty(\mathbb R;H^3(\mathbb T))}\left\{\Vert\varphi_{t}\Vert_{L^2_\gamma(\mathbb R\times\mathbb T)}^2+\Vert\varphi_{x}\Vert_{L^2_\gamma(\mathbb R\times\mathbb T)}^2\right\}\,.
\end{split}
\end{equation}
To provide an estimate of the integral $\displaystyle 2\int_{\mathbb R\times\mathbb T}\mathbb H[\phi_{0,xx}\phi_{x,\gamma}]\varphi_{t,\gamma}\,dx\,dt$ in the right-hand side of \eqref{equazione_lin3} we use the properties of the Hilbert transform collected in Section \ref{trasf_hilbert}, together with  H\"{o}lder and Young's inequalities and \eqref{sobolev-poincare}, to get
\begin{equation}\label{stima_int_2.4}
\begin{split}
2\int_{\mathbb R\times\mathbb T}&\mathbb H[\phi_{0,xx}\phi_{x,\gamma}]\varphi_{t,\gamma}\,dx\,dt=-2\int_{\mathbb R\times\mathbb T}\phi_{0,xx}\phi_{x,\gamma}\mathbb H[\varphi_{t,\gamma}]\,dx\,dt=-2\int_{\mathbb R\times\mathbb T}\phi_{0,xx}\phi_{x,\gamma}\phi_{t,\gamma}\,dx\,dt\\
&\le 2\int_{\mathbb R} \Vert\phi_{0,xx}\Vert_{L^\infty(\mathbb T)}\Vert\phi_{t,\gamma}\Vert_{L^2(\mathbb T)}\Vert\phi_{x,\gamma}\Vert_{L^2(\mathbb T)}\,dt\le C\int_{\mathbb R}\Vert\varphi_{0,xxx}\Vert_{L^2(\mathbb T)}\Vert\varphi_{t,\gamma}\Vert_{L^2(\mathbb T)}\Vert\varphi_{x,\gamma}\Vert_{L^2(\mathbb T)}\,dt\\
&\le C\Vert\varphi_{0}\Vert_{L^\infty(\mathbb R;H^3(\mathbb T))}\left\{\Vert\varphi_{t}\Vert_{L^2_\gamma(\mathbb R\times\mathbb T)}^2+\Vert\varphi_{x}\Vert_{L^2_\gamma(\mathbb R\times\mathbb T)}^2\right\}\,.
\end{split}
\end{equation}
To obtain an estimate of $\displaystyle 2\int_{\mathbb R\times\mathbb T}\left[\mathbb H\,;\phi_{0,x}\right]\partial_x\phi_{x,\gamma}\varphi_{t,\gamma}\,dx\,dt$ we use H\"{o}lder's inequality, \eqref{sobolev-poincare} and the estimate \eqref{stima_comm_2} of Lemma \ref{lemma_comm} with $s=1$, $v=\phi_{0,x}$ and $f=\phi_{x,\gamma}$ (and again the properties of the Hilbert transform) to get
\begin{equation}\label{stima_int_2.3}
\begin{split}
\displaystyle 2\int_{\mathbb R\times\mathbb T}&\left[\mathbb H\,;\phi_{0,x}\right]\partial_x\phi_{x,\gamma}\varphi_{t,\gamma}\,dx\,dt\le 2\int_{\mathbb R}\Vert \left[\mathbb H\,;\phi_{0,x}\right]\partial_x\phi_{x,\gamma}\Vert_{L^2(\mathbb T)}\Vert\varphi_{t,\gamma}\Vert_{L^2(\mathbb T)}\,dt\\
&\le C\int_{\mathbb R}\Vert \phi_{0,xx}\Vert_{H^1(\mathbb T)}\Vert \phi_{x,\gamma}\Vert_{L^2(\mathbb T)}\Vert\varphi_{t,\gamma}\Vert_{L^2(\mathbb T)}\le C \int_{\mathbb R}\Vert \varphi_{0,xx}\Vert_{H^1(\mathbb T)}\Vert \varphi_{x,\gamma}\Vert_{L^2(\mathbb T)}\Vert\varphi_{t,\gamma}\Vert_{L^2(\mathbb T)}\\
&\le C\Vert \varphi_{0}\Vert_{L^\infty(\mathbb R;H^3(\mathbb T))}\left\{\Vert \varphi_{x}\Vert_{L^2_\gamma(\mathbb R\times\mathbb T)}^2+\Vert \varphi_{t}\Vert_{L^2_\gamma(\mathbb R\times\mathbb T)}^2\right\}\,.
\end{split}
\end{equation}
As for the integral term involving the source $g$, H\"{o}lder and Young's inequalities yield
\begin{equation}\label{stima_g}
\int_{\mathbb R\times\mathbb T}g_\gamma\,\varphi_{t,\gamma}\,dx\,dt\le\Vert g\Vert_{L^2_\gamma(\mathbb R\times\mathbb T)}\Vert\varphi_{t}\Vert_{L^2_\gamma(\mathbb R\times\mathbb T)}\le\frac{\gamma}{2}\Vert\varphi_{t}\Vert_{L^2_\gamma(\mathbb R\times\mathbb T)}^2+\frac{1}{2\gamma}\Vert g\Vert_{L^2_\gamma(\mathbb R\times\mathbb T)}^2\,.
\end{equation}
It remains now to treat the last term in the right-hand side of \eqref{equazione_lin3}, that is
\begin{equation}\label{int_2.5}
\mathcal I:=-\int_{\mathbb R\times\mathbb T}\left(\left[\phi_\gamma;\mathbb H\right]\phi_{0,xx}+\left[\phi_0;\mathbb H\right]\partial_x\phi_{x,\gamma}\right)_x\varphi_{t,\gamma}\,dx\,dt\,.
\end{equation}
Let us firstly decompose $\mathcal I$ above as the sum
\begin{equation}\label{decomp_int_2.5}
\mathcal I=\mathcal I_1+\mathcal I_2\,,
\end{equation}
where
\begin{align}
&\mathcal I_1:=-\int_{\mathbb R\times\mathbb T}\left(\left[\phi_\gamma;\mathbb H\right]\phi_{0,xx}\right)_x\varphi_{t,\gamma}\,dx\,dt\,,\label{int_2.5.1}\\
&\mathcal I_2:=-\int_{\mathbb R\times\mathbb T}\left(\left[\phi_0;\mathbb H\right]\partial_x\phi_{x,\gamma}\right)_x\varphi_{t,\gamma}\,dx\,dt\,.\label{int_2.5.2}
\end{align}
{\it The estimate of $\mathcal I_1$:} we use Leibniz's formula, the definition of the commutator and formula \eqref{integraleH} to rewrite $\mathcal I_1$ as
\begin{equation}\label{rappres_int_2.5.1}
\begin{split}
\mathcal I_1=&-\int_{\mathbb R\times\mathbb T}\left[\phi_{x,\gamma};\mathbb H\right]\phi_{0,xx}\varphi_{t,\gamma}\,dx\,dt-\int_{\mathbb R\times\mathbb T}\left[\phi_{\gamma};\mathbb H\right]\phi_{0,xxx}\varphi_{t,\gamma}\,dx\\
&=-\int_{\mathbb R\times\mathbb T}\phi_{x,\gamma}\mathbb H[\phi_{0,xx}]\varphi_{t,\gamma}\,dx\,dt+\int_{\mathbb R\times\mathbb T}\mathbb H\left[\phi_{x,\gamma}\phi_{0,xx}\right]\varphi_{t,\gamma}\,dx\,dt\\
&-\int_{\mathbb R\times\mathbb T}\phi_\gamma\mathbb H[\phi_{0,xxx}]\varphi_{t,\gamma}\,dx\,dt+\int_{\mathbb R\times\mathbb T}\mathbb H\left[\phi_\gamma\phi_{0,xxx}\right]\varphi_{t,\gamma}\,dx\,dt\\
&=-\int_{\mathbb R\times\mathbb T}\phi_{x,\gamma}\mathbb H[\phi_{0,xx}]\varphi_{t,\gamma}\,dx\,dt-\int_{\mathbb R\times\mathbb T}\phi_{x,\gamma}\phi_{0,xx}\mathbb H[\varphi_{t,\gamma}]\,dx\,dt\\
&-\int_{\mathbb R\times\mathbb T}\phi_\gamma\mathbb H[\phi_{0,xxx}]\varphi_{t,\gamma}\,dx\,dt-\int_{\mathbb R\times\mathbb T}\phi_\gamma\phi_{0,xxx}\mathbb H[\varphi_{t,\gamma}]\,dx\,dt\\
&=\int_{\mathbb R\times\mathbb T}\phi_{x,\gamma}\varphi_{0,xx}\varphi_{t,\gamma}\,dx\,dt-\int_{\mathbb R\times\mathbb T}\phi_{x,\gamma}\phi_{0,xx}\phi_{t,\gamma}\,dx\,dt\\
&+\int_{\mathbb R\times\mathbb T}\phi_\gamma\varphi_{0,xxx}\varphi_{t,\gamma}\,dx\,dt-\int_{\mathbb R\times\mathbb T}\phi_\gamma\phi_{0,xxx}\phi_{t,\gamma}\,dx\,dt\,.
\end{split}
\end{equation}
By using H\"{older} and Young's inequalities, \eqref{sobolev-poincare} and the $L^2-$continuity of the Hilbert transform (cf. Section \ref{trasf_hilbert}), the first and the second integrals involved in the representation \eqref{rappres_int_2.5.1} can be easily estimated by
\begin{equation}\label{stima_int_2.5.1.1-2}
\begin{array}{ll}
\begin{split}
\int_{\mathbb R\times\mathbb T}&\varphi_{0,xx}\phi_{x,\gamma}\varphi_{t,\gamma}\,dx\,dt\le\int_{\mathbb R}\Vert \varphi_{0,xx}\Vert_{L^\infty(\mathbb T)}\Vert\phi_{x,\gamma}\Vert_{L^2(\mathbb T)}\Vert\varphi_{t,\gamma}\Vert_{L^2(\mathbb T)}\,dt\\
&\le C\int_{\mathbb R}\Vert \varphi_{0,xxx}\Vert_{L^2(\mathbb T)}\left\{\Vert\varphi_{x\gamma}\Vert_{L^2(\mathbb T)}^2+\Vert\varphi_{t,\gamma}\Vert_{L^2(\mathbb T)}^2\right\}\,dt\\
&\le C\Vert \varphi_{0}\Vert_{L^\infty(\mathbb R;H^3(\mathbb T))}\left\{\Vert\varphi_{x}\Vert_{L^2_\gamma(\mathbb R\times\mathbb T)}^2+\Vert\varphi_{t}\Vert_{L^2_\gamma(\mathbb R\times\mathbb T)}^2\right\}\,,
\end{split}\\
\\
\begin{split}
-\int_{\mathbb R\times\mathbb T}&\phi_{0,xx}\phi_{x,\gamma}\phi_{t,\gamma}\,dx\,dt\le\int_{\mathbb R}\Vert \phi_{0,xx}\Vert_{L^\infty(\mathbb T)}\Vert\phi_{x,\gamma}\Vert_{L^2(\mathbb T)}\Vert\phi_{t,\gamma}\Vert_{L^2(\mathbb T)}\,dt\\
&\le C\int_{\mathbb R}\Vert \varphi_{0,xxx}\Vert_{L^2(\mathbb T)}\left\{\Vert\varphi_{x,\gamma}\Vert_{L^2(\mathbb T)}^2+\Vert\varphi_{t,\gamma}\Vert_{L^2(\mathbb T)}^2\right\}\,dt\\
&\le C\Vert\varphi_{0}\Vert_{L^\infty(\mathbb R;H^3(\mathbb T))}\left\{\Vert\varphi_{x}\Vert_{L^2_\gamma(\mathbb R\times\mathbb T)}^2+\Vert\varphi_{t}\Vert_{L^2_\gamma(\mathbb R\times\mathbb T)}^2\right\}\,.
\end{split}
\end{array}
\end{equation}
Concerning the third and fourth integrals in \eqref{rappres_int_2.5.1}, firstly we use again H\"{older} and Young's inequalities to get
\begin{equation}\label{stima_int_2.5.1.3-4}
\begin{array}{ll}
\displaystyle\int_{\mathbb R\times\mathbb T}\varphi_{0,xxx}\phi_\gamma\varphi_{t,\gamma}\,dx\,dt\le\int_{\mathbb R}\Vert \varphi_{0,xxx}\Vert_{L^2(\mathbb T)}\Vert\phi_\gamma\Vert_{L^\infty(\mathbb T)}\Vert\varphi_{t,\gamma}\Vert_{L^2(\mathbb T)}\,dt\,,\\
\\
\begin{split}
-\int_{\mathbb R\times\mathbb T}&\phi_{0,xxx}\phi_\gamma\phi_{t,\gamma}\,dx\,dt\le\int_{\mathbb R}\Vert \phi_{0,xxx}\Vert_{L^2(\mathbb T)}\Vert\phi_\gamma\Vert_{L^\infty(\mathbb T)}\Vert\phi_{t,\gamma}\Vert_{L^2(\mathbb T)}\,dt\\
&\le \int_{\mathbb R}\Vert \varphi_{0,xxx}\Vert_{L^2(\mathbb T)}\Vert\phi_\gamma\Vert_{L^\infty(\mathbb T)}\Vert\varphi_{t,\gamma}\Vert_{L^2(\mathbb T)}\,dt\,.
\end{split}
\end{array}
\end{equation}
Since the spatial mean of $\phi$ is zero, we apply \eqref{sobolev-poincare} to $\phi_\gamma$ and the $L^2-$continuity of the Hilbert transform to estimate
 \begin{equation}\label{sobolev-poincare-phi}
\Vert\phi_\gamma\Vert_{L^\infty(\mathbb T)}\le C\Vert\phi_{x,\gamma}\Vert_{L^2(\mathbb T)}\le C\Vert\varphi_{x,\gamma}\Vert_{L^2(\mathbb T)}\,.
\end{equation}
Combining \eqref{stima_int_2.5.1.3-4}, \eqref{sobolev-poincare-phi} and using once again Young's inequality and the Hilbert transform properties in Section \ref{trasf_hilbert} we get
\begin{equation}\label{stima_int_2.5.1.3-4'}
\begin{array}{ll}
\begin{split}
\int_{\mathbb R\times\mathbb T}&\varphi_{0,xxx}\phi_\gamma\varphi_{t,\gamma}\,dx\,dt\le C\int_\mathbb R\Vert \varphi_{0,xxx}\Vert_{L^2(\mathbb T)}\Vert\varphi_{x,\gamma}\Vert_{L^2(\mathbb T)}\Vert\varphi_{t,\gamma}\Vert_{L^2(\mathbb T)}\,dt\\
&\le C\Vert \varphi_{0}\Vert_{L^\infty(\mathbb R;H^3(\mathbb T))}\left\{\Vert\varphi_{x}\Vert_{L^2_\gamma(\mathbb R\times\mathbb T)}^2+\Vert\varphi_{t}\Vert_{L^2_\gamma(\mathbb R\times\mathbb T)}^2\right\}\,,
\end{split}\\
\\
\begin{split}
-\int_{\mathbb R\times\mathbb T}&\phi_{0,xxx}\phi_\gamma\phi_{t,\gamma}\,dx\,dt\le C \int_{\mathbb R}\Vert \varphi_{0,xxx}\Vert_{L^2(\mathbb T)}\Vert\varphi_{x,\gamma}\Vert_{L^2(\mathbb T)}\Vert\varphi_{t,\gamma}\Vert_{L^2(\mathbb T)}\,dt\\
&\le C\Vert\varphi_{0}\Vert_{L^\infty(\mathbb R;H^3(\mathbb T))}\left\{\Vert\varphi_{x}\Vert_{L^2_\gamma(\mathbb R\times\mathbb T)}^2+\Vert\varphi_{t}\Vert_{L^2_\gamma(\mathbb R\times\mathbb T)}^2\right\}\,.
\end{split}
\end{array}
\end{equation}
Gathering estimates \eqref{stima_int_2.5.1.1-2}, \eqref{stima_int_2.5.1.3-4'}, we obtain the following estimate of $\mathcal I_1$
\begin{equation}\label{stima_int_2.5.1}
\mathcal I_1\le C\Vert\varphi_{0}\Vert_{L^\infty(\mathbb R;H^3(\mathbb T))}\left\{\Vert\varphi_{x}\Vert_{L^2_\gamma(\mathbb R\times\mathbb T)}^2+\Vert\varphi_{t}\Vert_{L^2_\gamma(\mathbb R\times\mathbb T)}^2\right\}\,.
\end{equation}
{\it The estimate of $\mathcal I_2$}: The integral $\mathcal I_2$ is the most difficult to handle; indeed, even after integration by parts, it still contains second order derivatives of $\varphi$ that we can't estimate directly. We will manage to represent $\mathcal I_2$ as the sum of integral terms involving only first order derivatives of $\varphi$ to which operators of the kind considered in Lemma \ref{lemma_comm} are applied.

We first integrate by parts in $x$ and use the identity \eqref{identita} (with $f=\varphi$) to get
\begin{equation}\label{rappres_int_2.5.2_prel}
\mathcal I_2=\int_{\mathbb R\times\mathbb T}\left[\phi_0;\mathbb H\right]\partial_x\phi_{x,\gamma}\partial_x\varphi_{t,\gamma}\,dx\,dt
=\mathcal I_{2,1}+\mathcal I_{2,2}\,,
\end{equation}
where
\begin{align}
&\mathcal I_{2,1}:=\gamma\int_{\mathbb R\times\mathbb T}\left[\phi_0;\mathbb H\right]\partial_x\phi_{x,\gamma}\varphi_{x,\gamma}\,dx\,dt\,,\label{I_21}\\
&\mathcal I_{2,2}:=\int_{\mathbb R\times\mathbb T }\left[\phi_0;\mathbb H\right]\partial_x\phi_{x,\gamma}\partial_t\varphi_{x,\gamma}\,dx\,dt\,.\label{I_22}
\end{align}
Let us consider $\mathcal I_{2,2}$; integration by parts and Leibniz's rule give
\begin{equation}\label{rappres_int_2.5.2}
\begin{split}
\mathcal I_{2,2}&=-\int_{\mathbb R\times\mathbb T}\partial_t\left(\left[\phi_0;\mathbb H\right]\partial_x\phi_{x,\gamma}\right)\varphi_{x,\gamma}\,dx\,dt\\
&=-\int_{\mathbb R\times\mathbb T}\left[\phi_{0,t};\mathbb H\right]\partial_x\phi_{x,\gamma}\varphi_{x,\gamma}\,dx\,dt-\int_{\mathbb R\times\mathbb T}\left[\phi_0;\mathbb H\right]\partial_t\partial_x\phi_{x,\gamma}\varphi_{x,\gamma}\,dx\,dt\\
&=-\int_{\mathbb R\times\mathbb T}\left[\phi_{0,t};\mathbb H\right]\partial_x\phi_{x,\gamma}\varphi_{x,\gamma}\,dx\,dt-\int_{\mathbb R\times\mathbb T}\partial_x\left(\left[\phi_0;\mathbb H\right]\partial_t\phi_{x,\gamma}\right)\varphi_{x,\gamma}\,dx\,dt\\
&\quad +\int_{\mathbb R\times\mathbb T}\left[\phi_{0,x};\mathbb H\right]\partial_t\phi_{x,\gamma}\varphi_{x,\gamma}\,dx\,dt\\
&=-\int_{\mathbb R\times\mathbb T}\left[\phi_{0,t};\mathbb H\right]\partial_x\phi_{x,\gamma}\varphi_{x,\gamma}\,dx\,dt+\int_{\mathbb R\times\mathbb T}\left[\phi_0;\mathbb H\right]\partial_t\phi_{x,\gamma}\partial_x\varphi_{x,\gamma}\,dx\,dt\\
&\quad +\int_{\mathbb R\times\mathbb T}\left[\phi_{0,x};\mathbb H\right]\partial_t\phi_{x,\gamma}\varphi_{x,\gamma}\,dx\,dt\,.
\end{split}
\end{equation}
Let us focus on the second integral $\displaystyle\int_{\mathbb R\times\mathbb T}\left[\phi_0;\mathbb H\right]\partial_t\phi_{x,\gamma}\partial_x\varphi_{x,\gamma}\,dx\,dt$ in the above representation of $\mathcal I_{2,2}$. Using formulas \eqref{integraleH} and \eqref{aggiunto} we get
\begin{align*}
\int_{\mathbb R\times\mathbb T}&\left[\phi_0;\mathbb H\right]\partial_t\phi_{x,\gamma}\partial_x\varphi_{x,\gamma}\,dx\,dt\\
&=\int_{\mathbb R\times\mathbb T}\partial_t\phi_{x,\gamma}\left[\phi_0\,;\,\mathbb H\right]\partial_x\varphi_{x,\gamma}\,dx\,dt=-\int_{\mathbb R\times\mathbb T}\partial_t\varphi_{x,\gamma}\mathbb H\left[\left[\phi_0\,;\,\mathbb H\right]\partial_x\varphi_{x,\gamma}\right]\,dx\,dt\\
&=-\int_{\mathbb R\times\mathbb T}\partial_t\varphi_{x,\gamma}\left[\phi_0\,;\,\mathbb H\right]\mathbb H\left[\partial_x\varphi_{x,\gamma}\right]\,dx\,dt-\int_{\mathbb R\times\mathbb T}\partial_t\varphi_{x,\gamma}\left[\mathbb H\,;\,\left[\phi_0\,;\,\mathbb H\right]\right]\partial_x\varphi_{x,\gamma}\,dx\,dt\\
&=-\int_{\mathbb R\times\mathbb T}\partial_t\varphi_{x,\gamma}\left[\phi_0\,;\,\mathbb H\right]\partial_x\phi_{x,\gamma}\,dx\,dt-\int_{\mathbb R\times\mathbb T}\partial_t\varphi_{x,\gamma}\left[\mathbb H\,;\,\left[\phi_0\,;\,\mathbb H\right]\right]\partial_x\varphi_{x,\gamma}\,dx\,dt\,.
\end{align*}
Then we use once again the identity \eqref{identita} to exchange the $x$ and $t-$derivatives in $\partial_t\varphi_{x,\gamma}$ involved in the two integrals just above and find, after integration by parts in $x$ (see also \eqref{int_2.5.2}, \eqref{I_21}),
\begin{equation}\label{rappres_int_2.5.2_bis}
\begin{split}
\int_{\mathbb R\times\mathbb T}&\left[\phi_0;\mathbb H\right]\partial_t\phi_{x,\gamma}\partial_x\varphi_{x,\gamma}\,dx\,dt\\
&=-\int_{\mathbb R\times\mathbb T}(-\gamma\varphi_{x,\gamma}+\partial_x\varphi_{t,\gamma})\left[\phi_0\,;\,\mathbb H\right]\partial_x\phi_{x,\gamma}\,dx\,dt\\
&\quad -\int_{\mathbb R\times\mathbb T}(-\gamma\varphi_{x,\gamma}+\partial_x\varphi_{t,\gamma})\left[\mathbb H\,;\,\left[\phi_0\,;\,\mathbb H\right]\right]\partial_x\varphi_{x,\gamma}\,dx\,dt\\
&=\gamma\int_{\mathbb R\times\mathbb T}\varphi_{x,\gamma}\left[\phi_0\,;\,\mathbb H\right]\partial_x\phi_{x,\gamma}\,dx\,dt+\gamma\int_{\mathbb R\times\mathbb T}\varphi_{x,\gamma}\left[\mathbb H\,;\,\left[\phi_0\,;\,\mathbb H\right]\right]\partial_x\varphi_{x,\gamma}\,dx\,dt\\
&\quad +\int_{\mathbb R\times\mathbb T}\varphi_{t,\gamma}\left(\left[\phi_0\,;\,\mathbb H\right]\partial_x\phi_{x,\gamma}\right)_x\,dx\,dt
+\int_{\mathbb R\times\mathbb T}\varphi_{t,\gamma}\left(\left[\mathbb H\,;\,\left[\phi_0\,;\,\mathbb H\right]\right]\partial_x\varphi_{x,\gamma}\right)_x\,dx\,dt\\
&=\mathcal I_{2,1}+\gamma\int_{\mathbb R\times\mathbb T}\varphi_{x,\gamma}\left[\mathbb H\,;\,\left[\phi_0\,;\,\mathbb H\right]\right]\partial_x\varphi_{x,\gamma}\,dx\,dt\\
&\quad -\mathcal I_2+\int_{\mathbb R\times\mathbb T}\varphi_{t,\gamma}\left(\left[\mathbb H\,;\,\left[\phi_0\,;\,\mathbb H\right]\right]\partial_x\varphi_{x,\gamma}\right)_x\,dx\,dt\,.
\end{split}
\end{equation}
Substituting \eqref{rappres_int_2.5.2_bis} in \eqref{rappres_int_2.5.2} we get
\begin{equation}\label{rappres_int_2.5.2_ter}
\begin{split}
\mathcal I_{2,2}&=-\int_{\mathbb R\times\mathbb T}\left[\phi_{0,t};\mathbb H\right]\partial_x\phi_{x,\gamma}\varphi_{x,\gamma}\,dx\,dt+\mathcal I_{2,1}\\
&\quad +\gamma\int_{\mathbb R\times\mathbb T}\varphi_{x,\gamma}\left[\mathbb H\,;\,\left[\phi_0\,;\,\mathbb H\right]\right]\partial_x\varphi_{x,\gamma}\,dx\,dt-\mathcal I_2\\
&\quad +\int_{\mathbb R\times\mathbb T}\varphi_{t,\gamma}\left(\left[\mathbb H\,;\,\left[\phi_0\,;\,\mathbb H\right]\right]\partial_x\varphi_{x,\gamma}\right)_x\,dx\,dt +\int_{\mathbb R\times\mathbb T}\left[\phi_{0,x};\mathbb H\right]\partial_t\phi_{x,\gamma}\varphi_{x,\gamma}\,dx\,dt\,.
\end{split}
\end{equation}
Then substituting \eqref{rappres_int_2.5.2_ter} into \eqref{rappres_int_2.5.2_prel} gives
\begin{equation*}
\begin{split}
\mathcal I_2&=2\mathcal I_{2,1}-\int_{\mathbb R\times\mathbb T}\left[\phi_{0,t};\mathbb H\right]\partial_x\phi_{x,\gamma}\varphi_{x,\gamma}\,dx\,dt+\gamma\int_{\mathbb R\times\mathbb T}\varphi_{x,\gamma}\left[\mathbb H\,;\,\left[\phi_0\,;\,\mathbb H\right]\right]\partial_x\varphi_{x,\gamma}\,dx\,dt\\
&\quad -\mathcal I_2+\int_{\mathbb R\times\mathbb T}\varphi_{t,\gamma}\left(\left[\mathbb H\,;\,\left[\phi_0\,;\,\mathbb H\right]\right]\partial_x\varphi_{x,\gamma}\right)_x\,dx\,dt +\int_{\mathbb R\times\mathbb T}\left[\phi_{0,x};\mathbb H\right]\partial_t\phi_{x,\gamma}\varphi_{x,\gamma}\,dx\,dt\,,
\end{split}
\end{equation*}
hence, using also \eqref{identita} with $f=\phi$ and integrating by parts in $x$,
\begin{equation}\label{rappres_int_2.5.2'}
\begin{split}
\displaystyle\mathcal I_2&=\mathcal I_{2,1}-\frac12\int_{\mathbb R\times\mathbb T}\left[\phi_{0,t};\mathbb H\right]\partial_x\phi_{x,\gamma}\varphi_{x,\gamma}\,dx\,dt+\frac{\gamma}{2}\int_{\mathbb R\times\mathbb T}\varphi_{x,\gamma}\left[\mathbb H\,;\,\left[\phi_0\,;\,\mathbb H\right]\right]\partial_x\varphi_{x,\gamma}\,dx\,dt\\
&\quad+\frac12\int_{\mathbb R\times\mathbb T}\varphi_{t,\gamma}\left(\left[\mathbb H\,;\,\left[\phi_0\,;\,\mathbb H\right]\right]\partial_x\varphi_{x,\gamma}\right)_x\,dx\,dt+\frac12\int_{\mathbb R\times\mathbb T}\left[\phi_{0,x};\mathbb H\right]\partial_t\phi_{x,\gamma}\varphi_{x,\gamma}\,dx\,dt\\
&=\mathcal I_{2,1}-\frac12\int_{\mathbb R\times\mathbb T}\left[\phi_{0,t};\mathbb H\right]\partial_x\phi_{x,\gamma}\varphi_{x,\gamma}\,dx\,dt-\frac{\gamma}{2}\int_{\mathbb R\times\mathbb T}\varphi_{\gamma}\left(\left[\mathbb H\,;\,\left[\phi_0\,;\,\mathbb H\right]\right]\partial_x\varphi_{x,\gamma}\right)_x\,dx\,dt\\
&\quad+\frac12\int_{\mathbb R\times\mathbb T}\varphi_{t,\gamma}\left(\left[\mathbb H\,;\,\left[\phi_0\,;\,\mathbb H\right]\right]\partial_x\varphi_{x,\gamma}\right)_x\,dx\,dt-\frac{\gamma}2\int_{\mathbb R\times\mathbb T}\left[\phi_{0,x};\mathbb H\right]\phi_{x,\gamma}\varphi_{x,\gamma}\,dx\,dt\\
&\quad+\frac12\int_{\mathbb R\times\mathbb T}\left[\phi_{0,x};\mathbb H\right]\partial_x\phi_{t,\gamma}\varphi_{x,\gamma}\,dx\,dt\,.
\end{split}
\end{equation}
Let us now rewrite the integral $\mathcal I_{2,1}$, see \eqref{I_21}. In fact, the coefficient $\gamma$ in front of it does not allow a direct estimate.
We expand the commutator $\left[\phi_0;\mathbb H\right]$, then use \eqref{calcolo}, \eqref{integraleH} and integration by parts to write
\begin{equation}\label{rappres_int_2.5.2.1'}
\begin{split}
\int_{\mathbb R\times\mathbb T}&\left[\phi_0;\mathbb H\right]\partial_x\phi_{x,\gamma}\varphi_{x,\gamma}\,dx\,dt\\
&=\int_{\mathbb R\times\mathbb T}\phi_0\,\mathbb H[\partial_x\phi_{x,\gamma}]\varphi_{x,\gamma}\,dx\,dt-\int_{\mathbb R\times\mathbb T}\mathbb H\left[\phi_0\,\partial_x\phi_{x,\gamma}\right]\varphi_{x,\gamma}\,dx\,dt\\
&=-\int_{\mathbb R\times\mathbb T}\phi_0\,\partial_x\varphi_{x,\gamma}\varphi_{x,\gamma}\,dx\,dt+\int_{\mathbb R\times\mathbb T}\phi_0\,\partial_x\phi_{x,\gamma}\phi_{x,\gamma}\,dx\,dt\\
&=-\frac12\int_{\mathbb R\times\mathbb T}\phi_0\,\partial_x\left(\left\vert\varphi_{x,\gamma}\right\vert^2\right)\,dx\,dt+\frac12\int_{\mathbb R\times\mathbb T}\phi_0\,\partial_x\left(\left\vert\phi_{x,\gamma}\right\vert^2\right)\,dx\,dt\\
&=\frac12\int_{\mathbb R\times\mathbb T}\phi_{0,x}\,\left(
\left\vert\varphi_{x,\gamma}\right\vert^2
-\left\vert\phi_{x,\gamma}
\right\vert^2
\right)\,dx\,dt\,.
\end{split}
\end{equation}
We use \eqref{prodotto_hilbert} with $f=g=\varphi_{x,\gamma}$, together with the first formula in \eqref{calcolo}  (notice that  $fg-\mathbb H[f]\mathbb H[g]$ has zero mean as long as $f$ and $g$ have zero mean), to get
\begin{equation*}
\left\vert\phi_{x,\gamma}\right\vert^2=\left\vert\mathbb H[\varphi_{x,\gamma}]\right\vert^2=\left\vert\varphi_{x,\gamma}\right\vert^2+2\mathbb H\left[\varphi_{x,\gamma}\phi_{x,\gamma}\right]\,.
\end{equation*}
Thus substituting into \eqref{rappres_int_2.5.2.1'} we find
\begin{equation*}
\begin{split}
\int_{\mathbb R\times\mathbb T}&\left[\phi_0;\mathbb H\right]\partial_x\phi_{x,\gamma}\varphi_{x,\gamma}\,dx\,dt
=-\int_{\mathbb R\times\mathbb T}\phi_{0,x}\mathbb H\left[\varphi_{x,\gamma}\phi_{x,\gamma}\right]\,dx\,dt\\
&=\int_{\mathbb R\times\mathbb T}\mathbb H[\phi_{0,x}]\varphi_{x,\gamma}\phi_{x,\gamma}\,dx\,dt
=-\int_{\mathbb R\times\mathbb T}\varphi_{0,x}\varphi_{x,\gamma}\phi_{x,\gamma}\,dx\,dt\,.
\end{split}
\end{equation*}
We may further rewrite (and it is convenient doing so) the last integral above in terms of a commutator operator; indeed we may compute
\begin{equation*}
\begin{split}
-\int_{\mathbb R\times\mathbb T}&\varphi_{0,x}\varphi_{x,\gamma}\phi_{x,\gamma}\,dx\,dt\\
&=-\int_{\mathbb R\times\mathbb T}\varphi_{0,x}\varphi_{x,\gamma}\mathbb H\left[\varphi_{x,\gamma}\right]\,dx\,dt+\int_{\mathbb R\times\mathbb T}\varphi_{x,\gamma}\mathbb  H\left[\varphi_{0,x}\varphi_{x,\gamma}\right]\,dx\,dt-\int_{\mathbb R\times\mathbb T}\varphi_{x,\gamma}\mathbb  H\left[\varphi_{0,x}\varphi_{x,\gamma}\right]\,dx\,dt\\
&=\int_{\mathbb R\times\mathbb T}\left[\mathbb H;\varphi_{0,x}\right]\partial_x\varphi_{\gamma}\varphi_{x,\gamma}\,dx\,dt+\int_{\mathbb R\times\mathbb T}\varphi_{0,x}\varphi_{x,\gamma}\phi_{x,\gamma}\,dx\,dt
\end{split}
\end{equation*}
hence
\begin{equation*}
-\int_{\mathbb R\times\mathbb T}\varphi_{0,x}\varphi_{x,\gamma}\phi_{x,\gamma}\,dx\,dt=\frac12\int_{\mathbb R\times\mathbb T}\left[\mathbb H;\varphi_{0,x}\right]\partial_x\varphi_{\gamma}\varphi_{x,\gamma}\,dx\,dt
\end{equation*}
and
\begin{equation*}
\int_{\mathbb R\times\mathbb T}\left[\phi_0;\mathbb H\right]\partial_x\phi_{x,\gamma}\varphi_{x,\gamma}\,dx\,dt=-\int_{\mathbb R\times\mathbb T}\varphi_{0,x}\varphi_{x,\gamma}\phi_{x,\gamma}\,dx\,dt=\frac12\int_{\mathbb R\times\mathbb T}\left[\mathbb H;\varphi_{0,x}\right]\partial_x\varphi_{\gamma}\varphi_{x,\gamma}\,dx\,dt\,.
\end{equation*}
In conclusion we obtain
\begin{equation}\label{rappres_I_21}
\mathcal I_{2,1}=\frac{\gamma}2\int_{\mathbb R\times\mathbb T}\left[\mathbb H;\varphi_{0,x}\right]\partial_x\varphi_{\gamma}\varphi_{x,\gamma}\,dx\,dt\,,
\end{equation}
then substituting \eqref{rappres_I_21} into \eqref{rappres_int_2.5.2'} we get
\begin{equation}\label{rappres_int_2.5.2''}
\begin{split}
\displaystyle\mathcal I_2&=\frac{\gamma}2\int_{\mathbb R\times\mathbb T}\left[\mathbb H;\varphi_{0,x}\right]\partial_x\varphi_{\gamma}\varphi_{x,\gamma}\,dx\,dt-\frac12\int_{\mathbb R\times\mathbb T}\left[\phi_{0,t};\mathbb H\right]\partial_x\phi_{x,\gamma}\varphi_{x,\gamma}\,dx\,dt\\
&-\frac{\gamma}{2}\int_{\mathbb R\times\mathbb T}\varphi_{\gamma}\left(\left[\mathbb H\,;\,\left[\phi_0\,;\,\mathbb H\right]\right]\partial_x\varphi_{x,\gamma}\right)_x\,dx\,dt+\frac12\int_{\mathbb R\times\mathbb T}\varphi_{t,\gamma}\left(\left[\mathbb H\,;\,\left[\phi_0\,;\,\mathbb H\right]\right]\partial_x\varphi_{x,\gamma}\right)_x\,dx\,dt\\
&-\frac{\gamma}2\int_{\mathbb R\times\mathbb T}\left[\phi_{0,x};\mathbb H\right]\partial_x\phi_{\gamma}\varphi_{x,\gamma}\,dx\,dt+\frac12\int_{\mathbb R\times\mathbb T}\left[\phi_{0,x};\mathbb H\right]\partial_x\phi_{t,\gamma}\varphi_{x,\gamma}\,dx\,dt=\sum\limits_{j=1}^6 I_j\,.
\end{split}
\end{equation}
We are now in the position to provide a suitable estimate of $\mathcal I_2$.

Concerning the first integral $I_1:=\displaystyle\frac{\gamma}2\int_{\mathbb R\times\mathbb T}\left[\mathbb H;\varphi_{0,x}\right]\partial_x\varphi_{\gamma}\varphi_{x,\gamma}\,dx\,dt$ in the right-hand side of \eqref{rappres_int_2.5.2''}, we apply H\"{o}lder and Young's inequalities and use the estimates \eqref{stima_comm_2} (with $s=1$, $v=\varphi_{0,x}$ and $f=\varphi_\gamma$), \eqref{stima_metivier} to get
\begin{equation*}
\begin{split}
I_1&\le\frac{\gamma}{2}\int_{\mathbb R}\left\Vert\left[\mathbb H\,;\,\varphi_{0,x}\right]\partial_x\varphi_{\gamma}\right\Vert_{L^2(\mathbb T)}\Vert\varphi_{x,\gamma}\Vert_{L^2(\mathbb T)}\,dt\\
&\le C\gamma\int_{\mathbb R}\Vert\varphi_{0,xx}\Vert_{H^1(\mathbb T)}\Vert\varphi_\gamma\Vert_{L^2(\mathbb T)}\Vert\varphi_{x,\gamma}\Vert_{L^2(\mathbb T)}\,dt\\
&\le C\int_{\mathbb R}\Vert\varphi_{0,xx}\Vert_{H^1(\mathbb T)}\Vert\varphi_{t,\gamma}\Vert_{L^2(\mathbb T)}\Vert\varphi_{x,\gamma}\Vert_{L^2(\mathbb T)}\,dt\\
&\le C\Vert\varphi_{0}\Vert_{L^\infty(\mathbb R; H^3(\mathbb T))}\left\{\Vert\varphi_{t}\Vert^2_{L^2_\gamma(\mathbb R\times\mathbb T)}+\Vert\varphi_{x}\Vert^2_{L^2_\gamma(\mathbb R\times\mathbb T)}\right\}\,.
\end{split}
\end{equation*}
To estimate the fifth integral $I_5:=\displaystyle-\frac{\gamma}2\int_{\mathbb R\times\mathbb T}\left[\phi_{0,x};\mathbb H\right]\partial_x\phi_{\gamma}\varphi_{x,\gamma}\,dx\,dt$ in the right-hand side of \eqref{rappres_int_2.5.2''} we use exactly the same arguments as before to obtain
\begin{equation*}
I_5\le C\Vert\varphi_{0}\Vert_{L^\infty(\mathbb R; H^3(\mathbb T))}\left\{\Vert\varphi_{t}\Vert^2_{L^2_\gamma(\mathbb R\times\mathbb T)}+\Vert\varphi_{x}\Vert^2_{L^2_\gamma(\mathbb R\times\mathbb T)}\right\}\,.
\end{equation*}
The integrals $I_2:=\displaystyle-\frac12\int_{\mathbb R\times\mathbb T}\left[\phi_{0,t};\mathbb H\right]\partial_x\phi_{x,\gamma}\varphi_{x,\gamma}\,dx\,dt$ and $I_6:=\displaystyle\frac12\int_{\mathbb R\times\mathbb T}\left[\phi_{0,x};\mathbb H\right]\partial_x\phi_{t,\gamma}\varphi_{x,\gamma}\,dx\,dt$ in the right-hand side of \eqref{rappres_int_2.5.2''} can be treated exactly in the same manner as above, applying H\"{o}lder and Young's inequalities, using the estimate \eqref{stima_comm_2} (with $s=1$, and $v=\phi_{0,t}$, $f=\phi_{x,\gamma}$ or $v=\phi_{0,x}$, $f=\phi_{t,\gamma}$ respectively in $I_2$ or $I_6$) and the $L^2-$continuity of $\mathbb H$ to get
\begin{equation*}
\begin{split}
\begin{split}
I_2&\le\frac12\int_{\mathbb R}\Vert\left[\phi_{0,t};\mathbb H\right]\partial_x\phi_{x,\gamma}\Vert_{L^2(\mathbb T)}\Vert\varphi_{x,\gamma}\Vert_{L^2(\mathbb T)}\,dt\le C\int_{\mathbb R}\Vert\phi_{0,tx}\Vert_{H^1(\mathbb T)}\Vert\phi_{x,\gamma}\Vert_{L^2(\mathbb T)}\Vert\varphi_{x,\gamma}\Vert_{L^2(\mathbb T)}\,dt\\
&\le C\Vert\varphi_{0,t}\Vert_{L^\infty(\mathbb R;H^2(\mathbb T))}\Vert\varphi_{x}\Vert^2_{L^2_\gamma(\mathbb R\times\mathbb T)}\,;
\end{split}\\
\\
\begin{split}
I_6&\le\frac12\int_{\mathbb R}\Vert\left[\phi_{0,x};\mathbb H\right]\partial_x\phi_{t,\gamma}\Vert_{L^2(\mathbb T)}\Vert\varphi_{x,\gamma}\Vert_{L^2(\mathbb T)}\,dt\le C\int_{\mathbb R}\Vert\phi_{0,xx}\Vert_{H^1(\mathbb T)}\Vert\phi_{t,\gamma}\Vert_{L^2(\mathbb T)}\Vert\varphi_{x,\gamma}\Vert_{L^2(\mathbb T)}\,dt\\
&\le C\Vert\varphi_{0}\Vert_{L^\infty(\mathbb R;H^3(\mathbb T))}\left\{\Vert\varphi_{t}\Vert^2_{L^2_\gamma(\mathbb R\times\mathbb T)}+\Vert\varphi_{x}\Vert^2_{L^2_\gamma(\mathbb R\times\mathbb T)}\right\}\,.
\end{split}
\end{split}
\end{equation*}
To estimate
\begin{equation*}
I_3:=-\frac{\gamma}{2}\int_{\mathbb R\times\mathbb T}\varphi_{\gamma}\left(\left[\mathbb H\,;\,\left[\phi_0\,;\,\mathbb H\right]\right]\partial_x\varphi_{x,\gamma}\right)_x\,dx\,dt\,,\quad I_4:=\displaystyle\frac12\int_{\mathbb R\times\mathbb T}\varphi_{t,\gamma}\left(\left[\mathbb H\,;\,\left[\phi_0\,;\,\mathbb H\right]\right]\partial_x\varphi_{x,\gamma}\right)_x\,dx\,dt
\end{equation*}
in the right-hand side of \eqref{rappres_int_2.5.2''}, we use again H\"{o}lder's inequality, the estimates \eqref{stima_comm_3} (with $s=1$, $v=\phi_{0}$ and $f=\varphi_{x,\gamma}$), \eqref{stima_metivier}, \eqref{sobolev-poincare} and the $L^2-$continuity of $\mathbb H$ to get
\begin{equation*}
\begin{split}
\begin{split}
I_3&\le\frac{\gamma}2\int_{\mathbb R}\left\Vert \left(\left[\mathbb H\,;\,\left[\phi_0\,;\,\mathbb H\right]\right]\partial_x\varphi_{x,\gamma}\right)_x\right\Vert_{L^2(\mathbb T)}\Vert\varphi_{\gamma}\Vert_{L^2(\mathbb T)}\,dt\\
&\le C\gamma\int_{\mathbb R}\Vert\phi_{0,xx}\Vert_{H^1(\mathbb T)}\Vert\varphi_{x,\gamma}\Vert_{L^2(\mathbb T)}\Vert\varphi_{\gamma}\Vert_{L^2(\mathbb T)}\,dt\\
&\le C\int_{\mathbb R}\Vert\varphi_{0,xx}\Vert_{H^1(\mathbb T)}\Vert\varphi_{x,\gamma}\Vert_{L^2(\mathbb T)}\Vert\varphi_{t,\gamma}\Vert_{L^2(\mathbb T)}\,dt\\
&\le C\Vert\varphi_{0}\Vert_{L^\infty(\mathbb R;H^3(\mathbb T))}\left\{\Vert\varphi_{t}\Vert^2_{L^2_\gamma(\mathbb R\times\mathbb T)}+\Vert\varphi_{x}\Vert^2_{L^2_\gamma(\mathbb R\times\mathbb T)}\right\}\,;
\end{split}\\
\\
\begin{split}
I_4&\le\frac12\int_{\mathbb R}\left\Vert \left(\left[\mathbb H\,;\,\left[\phi_0\,;\,\mathbb H\right]\right]\partial_x\varphi_{x,\gamma}\right)_x\right\Vert_{L^2(\mathbb T)}\Vert\varphi_{t,\gamma}\Vert_{L^2(\mathbb T)}\,dt\\
&\le C\int_{\mathbb R}\Vert\varphi_{0,xx}\Vert_{H^1(\mathbb T)}\Vert\varphi_{x,\gamma}\Vert_{L^2(\mathbb T)}\Vert\varphi_{t,\gamma}\Vert_{L^2(\mathbb T)}\,dt\\
&\le C\Vert\varphi_{0}\Vert_{L^\infty(\mathbb R;H^3(\mathbb T))}\left\{\Vert\varphi_{t}\Vert^2_{L^2_\gamma(\mathbb R\times\mathbb T)}+\Vert\varphi_{x}\Vert^2_{L^2_\gamma(\mathbb R\times\mathbb T)}\right\}\,.
\end{split}
\end{split}
\end{equation*}
Gathering the preceding estimates yields the following estimate of $\mathcal I_2$
\begin{equation*}
\mathcal I_2\le C\left\{\Vert\varphi_{0,t}\Vert_{L^\infty(\mathbb R;H^2(\mathbb T))}+\Vert\varphi_{0}\Vert_{L^\infty(\mathbb R;H^3(\mathbb T))}\right\}\left\{\Vert\varphi_{t}\Vert^2_{L^2_\gamma(\mathbb R\times\mathbb T)}+\Vert\varphi_{x}\Vert^2_{L^2_\gamma(\mathbb R\times\mathbb T)}\right\}\,,
\end{equation*}
from which, in view of \eqref{decomp_int_2.5} and using also \eqref{stima_int_2.5.1}, we get for the integral $\mathcal I$ in \eqref{int_2.5} the estimate
\begin{equation}\label{stima_int_2.5}
\mathcal I\le C\left\{\Vert\varphi_{0,t}\Vert_{L^\infty(\mathbb R;H^2(\mathbb T))}+\Vert\varphi_{0}\Vert_{L^\infty(\mathbb R;H^3(\mathbb T))}\right\}\left\{\Vert\varphi_{t}\Vert^2_{L^2_\gamma(\mathbb R\times\mathbb T)}+\Vert\varphi_{x}\Vert^2_{L^2_\gamma(\mathbb R\times\mathbb T)}\right\}\,.
\end{equation}
Gathering the estimates \eqref{stima_int_2.1'}-\eqref{stima_g}, \eqref{stima_int_2.5}, from \eqref{equazione_lin3} we derive
\begin{equation}\label{stima_lineare1}
\begin{split}
\gamma &\Vert\varphi_{t}\Vert^2_{L^2_\gamma(\mathbb R\times\mathbb T)}+\gamma\int_{\mathbb R\times\mathbb T}\left(\mu-2\phi_{0,x}\right)\left\vert\varphi_{x,\gamma}\right\vert^2\,dx\,dt\\
&\le C\left\{\Vert\varphi_{0,t}\Vert_{L^\infty(\mathbb R;H^2(\mathbb T))}+\Vert\varphi_{0}\Vert_{L^\infty(\mathbb R;H^3(\mathbb T))}\right\}\left\{\Vert\varphi_{t}\Vert^2_{L^2_\gamma(\mathbb R\times\mathbb T)}+\Vert\varphi_{x}\Vert^2_{L^2_\gamma(\mathbb R\times\mathbb T)}\right\}\\
&\quad +\frac{\gamma}{2}\Vert\varphi_{t}\Vert^2_{L^2_\gamma(\mathbb R\times\mathbb T)}+\frac{1}{2\gamma}\Vert g\Vert^2_{L^2_\gamma(\mathbb R\times\mathbb T)}\,.
\end{split}
\end{equation}
The estimate \eqref{stima_apriori_1} follows at once from \eqref{stima_lineare1} by absorbing in the left-hand side the term $\displaystyle\frac{\gamma}{2}\Vert\varphi_{t}\Vert^2_{L^2_\gamma(\mathbb R\times\mathbb T)}$, then using the assumption \eqref{stability_unif} to bound from below the integral $\displaystyle\int_{\mathbb R\times\mathbb T}\left(\mu-2\phi_{0,x}\right)\left\vert\varphi_{x,\gamma}\right\vert^2\,dx\,dt$ in the left-hand side of \eqref{stima_lineare1} by
\begin{equation*}
\frac{\delta}{2}\Vert\varphi_{x}\Vert^2_{L^2_\gamma(\mathbb R\times\mathbb T)}\,,
\end{equation*}
and finally taking $\gamma_0=\gamma_0(\delta, \Vert\varphi_{0,t}\Vert_{L^\infty(\mathbb R;H^2(\mathbb T))}, \Vert\varphi_{0}\Vert_{L^\infty(\mathbb R;H^3(\mathbb T))})$ sufficiently large such that
\begin{equation*}
\min\left\{1,\delta\right\}\gamma-2C\left\{\Vert\varphi_{0,t}\Vert_{L^\infty(\mathbb R;H^2(\mathbb T))}+\Vert\varphi_{0}\Vert_{L^\infty(\mathbb R;H^3(\mathbb T))}\right\}\ge\frac12\min\left\{1,\delta\right\}\gamma\,.
\end{equation*}
We can take for instance $\gamma_0=\displaystyle\frac{4C\left\{\Vert\varphi_{0,t}\Vert_{L^\infty(\mathbb R;H^2(\mathbb T))}+\Vert\varphi_{0}\Vert_{L^\infty(\mathbb R;H^3(\mathbb T))}\right\}}{\min\left\{1,\delta\right\}}$ and $C_0=\displaystyle\frac{2}{\min\{1,\delta\}}$.
\end{proof}
\subsection{Well-posedness of the linearized problem}
Theorem \ref{teorema_1} only gives an a priori estimate for any sufficiently smooth function $\varphi^\prime$ in terms of $g:=\mathbb{L}^\prime[\varphi_0]\varphi^\prime$. Using standard methods we can also obtain the existence of the solution.
\begin{theorem}\label{teorema_2}
Let the basic state $\varphi_0:\mathbb R\times\mathbb T\rightarrow\mathbb R$, with zero spatial mean, satisfy
\eqref{regolarita_phi0}, \eqref{stability_unif}, and let $\gamma_0\ge 1$, $C_0>0$ be the constants of Theorem \ref{teorema_1}.
For every $\gamma\ge \gamma_0$ and $g\in{L^2_\gamma(\mathbb R\times\mathbb T)}$ there exists a unique solution $\varphi^\prime$ with zero spatial mean of \eqref{equazione_lin} (expanded form in \eqref{equazione_lin1}) such that $\varphi^\prime_{t}\in{L^2_\gamma(\mathbb R\times\mathbb T)}, \varphi^\prime_{x}\in{L^2_\gamma(\mathbb R\times\mathbb T)}$. Moreover, $\varphi^\prime$
satisfies
the a priori estimate
\eqref{stima_apriori_1}.
\end{theorem}
\begin{proof}
The proof follows from a standard Galerkin approximation. Given any function $f:\mathbb T\rightarrow\mathbb C$ expanded in terms of Fourier series as
\begin{equation*}\label{}
f(x)=\frac{1}{2\pi}\sum\limits_{k\in\mathbb Z}\widehat{f}(k)e^{ikx}\,,
\end{equation*}
we define the finite dimensional orthogonal projection
\begin{equation*}\label{}
f^N(x)=P_Nf(x)=\frac{1}{2\pi}\sum\limits_{|k|\leq N}\widehat{f}(k)e^{ikx}\,.
\end{equation*}
Let us consider the Galerkin approximation $\vphi^N=P_N\varphi^\prime$, defined as the solution of the approximate ODE
\begin{multline}\label{equazione_lin1approx}
\varphi^N_{tt}-P_N\left(  \left(\mu-2\phi_{0,x}\right)\varphi^N_{xx} \right)\\
=P_N\left\{  2\left[\mathbb H\,;\phi_{0,x}\right]\phi^N_{xx}+2\mathbb H[\phi_{0,xx}\phi^N_x]
-\left(\left[\phi^N;\mathbb H\right]\phi_{0,xx}+\left[\phi_0;\mathbb H\right]\phi^N_{xx}\right)_x \right\}+g^N\,.
\end{multline}
We can repeat for $\vphi^N$ the same calculations in the proof of Theorem \ref{teorema_1}, leading to the analogue of the a priori estimate \eqref{stima_apriori_1}, uniform in $N$. By standard arguments we can extract a subsequence $\{\vphi^N\}$ and pass to the limit in \eqref{equazione_lin1approx} to obtain a solution of \eqref{equazione_lin}. The uniqueness of the solution follows from \eqref{stima_apriori_1} and the linearity of the problem.

\end{proof}


\section{Tame estimate for the linearized equation}\label{stima_tame}
This section is devoted to associate to all sufficiently smooth solutions of the equation \eqref{equazione_lin} an appropriate a priori estimate in higher order Sobolev norms. In agreement with the above section, the norms involved in the found estimate are {\it weighted} in time by the exponential function $e^{-\gamma t}$.

The goal of this section is to prove the following result.
\begin{proposition}\label{prop_stima_tame}
Let $m\ge1$ be an integer. Assume that the basic state $\varphi_0:\mathbb R\times\mathbb T\rightarrow\mathbb R$, with zero spatial mean, satisfies
\begin{equation*}
\varphi_0\in L^\infty(\mathbb R; H^5(\mathbb T))\cap L^\infty(\mathbb R; H^{m+3}(\mathbb T))\,,\quad\varphi_{0,t}\in L^\infty(\mathbb R; H^2(\mathbb T))
\end{equation*}
and
condition \eqref{stability_unif}. Then there exist $\gamma_1\ge 1$, and $C>0$, depending boundedly and increasingly on $m$, $\delta$, $\Vert\varphi_{0}\Vert_{L^\infty(\mathbb R;H^5(\mathbb T))}$, $\Vert\varphi_{0,t}\Vert_{L^\infty(\mathbb R; H^2(\mathbb T))} $, such that for all $\gamma\ge \gamma_1$ and $g\in{L^2_\gamma(\mathbb R; H^m(\mathbb T))}$ there exists a unique solution $\varphi^\prime$ with zero spatial mean of \eqref{equazione_lin} (expanded form in \eqref{equazione_lin1}) such that $\varphi^\prime_{t}\in{{L^2_\gamma(\mathbb R; H^m(\mathbb T))}}, \varphi^\prime_{x}\in{{L^2_\gamma(\mathbb R; H^m(\mathbb T))}}$. Moreover,  the following a priori estimate holds true
\begin{equation}\label{stima_tame_1}
\begin{split}
\gamma & \left\{\Vert\varphi^\prime_{t}\Vert^2_{L^2_\gamma(\mathbb R; H^m(\mathbb T))}+\Vert\varphi^\prime_{x}\Vert^2_{L^2_\gamma(\mathbb R;H^m(\mathbb T))}\right\}\\
&\le \frac{C}{\gamma}\left\{\Vert\varphi_{0,x}\Vert^2_{L^\infty(\mathbb R;H^{m+2}(\mathbb T))}\Vert g\Vert^2_{L^2_\gamma(\mathbb R; H^2(\mathbb T))} +\Vert g\Vert^2_{L^2_\gamma(\mathbb R; H^m(\mathbb T))}\right\}\,.
\end{split}
\end{equation}
\end{proposition}
\begin{remark}\label{}
The a priori estimate \eqref{stima_tame_1} shows the loss of 2 spatial derivatives in the inversion of the operator $\mathbb{L}^\prime[\varphi_0]$,
from the given basic state $\varphi_0$ to $\varphi^\prime$. For this reason the equation \eqref{onde_integro_diff} cannot be solved by standard methods like the implicit function theorem or the contraction principle. Instead, \eqref{onde_integro_diff} will be solved by applying the Nash-Moser's theorem.
\end{remark}
\begin{proof}
In view of \eqref{ident_norme}, to obtain an estimate of the $L^2_\gamma(\mathbb R;H^m(\mathbb T))-$norm of  the derivatives $\varphi^\prime_t$ and $\varphi^\prime_x$ of a smooth solution $\varphi^\prime$ to the equation \eqref{equazione_lin}, with an arbitrary (sufficiently large) order $m\ge 1$, let us first concentrate on the derivative $\partial_x^m\varphi^\prime$. To simplify the notation, in the following we drop the superscript ${}^\prime$ in the unknown function $\varphi^\prime$.

Applying the derivative $\partial_x^m$ to \eqref{equazione_lin} (see also \eqref{equazione_lin1}), after some calculations we find that $\partial_x^m\varphi$ must solve the linear equation
\begin{equation*}
\begin{split}
\mathbb{L}^\prime[\varphi_0]&\left(\partial_x^m\varphi\right)=2\mathbb H\left[\left[\partial_x^m\,;\,\phi_{0,x}\right]\phi_{xx}\right]+2\mathbb H\left[\left[\partial_x^m\,;\,\phi_{0,xx}\right]\phi_x\right]\\
&-\left(\left[\partial_x^m\,;\,\mathbb H\left[\phi_{0,xx}\right]\right]\phi-\mathbb H\left[\left[\partial_x^m\,;\,\phi_{0,xx}\right]\phi\right]+\left[\left[\partial_x^m\,;\,\phi_0\right]\,;\,\mathbb H\right]\phi_{xx}\right)_x+\partial_x^mg\\
&=\sum\limits_{j=1}^5G_j+\partial_x^mg\,,
\end{split}
\end{equation*}
which looks like the original linear equation \eqref{equazione_lin} where the forcing term $g$ is replaced by $\sum\limits_{j=1}^5G_j+\partial_x^mg$.
Hence applying the $L^2-$estimate \eqref{stima_apriori_1} we get for $\gamma\ge\gamma_0$
\begin{equation}\label{stima_apriori_2_0}
\gamma\left\{\Vert\partial_x^m\varphi_{t}\Vert^2_{L^2_\gamma(\mathbb R\times\mathbb T)}+\Vert\partial_x^m\varphi_{x}\Vert^2_{L^2_\gamma(\mathbb R\times\mathbb T)}\right\}\le\frac{C_0}{\gamma}\left\{\sum\limits_{j=1}^5\Vert G_j\Vert^2_{L^2_\gamma(\mathbb R\times\mathbb T)}+\Vert \partial_x^m g\Vert^2_{L^2_\gamma(\mathbb R\times\mathbb T)}\right\}\,,
\end{equation}
where the threshold $\gamma_0$ and the constant $C_0$ are provided by Theorem \ref{teorema_1}.

We have now to provide suitable estimates for the $L^2_\gamma-$norms of the $G_j$'s. As in the preceding section, throughout the following $C$ will always denote a positive constant, only depending on $m$, that may be different from line to line.

\smallskip
\noindent
{\it The estimate of} $G_1:=2\mathbb H\left[\left[\partial_x^m\,;\,\phi_{0,x}\right]\phi_{xx}\right]$: for fixed $t\in\mathbb R$ we first compute, by using the commutator estimate \eqref{stima_comm_5}, the Sobolev continuity of $\mathbb H$ and the Sobolev imbedding $H^1(\mathbb T)\hookrightarrow L^\infty(\mathbb T)$
\begin{equation*}
\begin{split}
\Vert & G_{1,\,\gamma}(t)\Vert_{L^2(\mathbb T)}=2\Vert \mathbb H\left[\left[\partial_x^m\,;\,\phi_{0,x}\right]\phi_{xx,\,\gamma}\right]\Vert_{L^2(\mathbb T)}= 2\Vert \left[\partial_x^m\,;\,\phi_{0,x}\right]\phi_{xx,\,\gamma}\Vert_{L^2(\mathbb T)}\\
&\le C\left\{\Vert\phi_{0,xx}\Vert_{L^\infty(\mathbb T)}\Vert\phi_{xx,\gamma}\Vert_{H^{m-1}(\mathbb T)}+\Vert\phi_{xx,\gamma}\Vert_{L^\infty(\mathbb T)}\Vert\phi_{0,x}\Vert_{H^{m}(\mathbb T)}\right\}\\
&\le C\left\{\Vert\phi_{0,xx}\Vert_{L^\infty(\mathbb T)}\Vert\varphi_{x,\gamma}\Vert_{H^{m}(\mathbb T)}+\Vert\phi_{xx,\gamma}\Vert_{L^\infty(\mathbb T)}\Vert\varphi_{0,x}\Vert_{H^{m}(\mathbb T)}\right\}\\
&\le C\left\{\Vert\varphi_{0}\Vert_{H^3(\mathbb T)}\Vert\varphi_{x,\gamma}\Vert_{H^{m}(\mathbb T)}+\Vert\varphi_{x,\gamma}\Vert_{H^2(\mathbb T)}\Vert\varphi_{0,x}\Vert_{H^{m}(\mathbb T)}\right\}\,.
\end{split}
\end{equation*}
Integrating over $\mathbb R_t$ we get
\begin{equation}\label{stima_G1_fin}
\Vert G_{1}\Vert_{L^2_\gamma(\mathbb R\times\mathbb T)}\le C\left\{\Vert\varphi_{0}\Vert_{L^\infty(\mathbb R;H^3(\mathbb T))}\Vert\varphi_{x}\Vert_{L^2_\gamma(\mathbb R;H^{m}(\mathbb T))}+\Vert\varphi_{x}\Vert_{L^2_\gamma(\mathbb R;H^2(\mathbb T))}\Vert\varphi_{0,x}\Vert_{L^\infty(\mathbb R;H^{m}(\mathbb T))}\right\}\,.
\end{equation}

\smallskip
\noindent
{\it The estimate of} $G_2:=2\mathbb H\left[\left[\partial_x^m\,;\,\phi_{0,xx}\right]\phi_{x}\right]$: for fixed $t\in\mathbb R$ we first compute by using the commutator estimate \eqref{stima_comm_4} and, for the rest, arguing as before
\begin{equation*}
\begin{split}
\Vert & G_{2,\,\gamma}(t)\Vert_{L^2(\mathbb T)}=2\Vert \mathbb H\left[\left[\partial_x^m\,;\,\phi_{0,xx}\right]\phi_{x,\,\gamma}\right]\Vert_{L^2(\mathbb T)}= 2\Vert \left[\partial_x^m\,;\,\phi_{0,xx}\right]\phi_{x,\,\gamma}\Vert_{L^2(\mathbb T)}\\
&\le C\left\{\Vert\phi_{0,xx}\Vert_{L^\infty(\mathbb T)}\Vert\phi_{x,\gamma}\Vert_{H^{m}(\mathbb T)}+\Vert\phi_{x,\gamma}\Vert_{L^\infty(\mathbb T)}\Vert\phi_{0,xx}\Vert_{H^{m}(\mathbb T)}\right\}\\
&\le C\left\{\Vert\varphi_{0}\Vert_{H^3(\mathbb T)}\Vert\varphi_{x,\gamma}\Vert_{H^{m}(\mathbb T)}+\Vert\varphi_{x,\gamma}\Vert_{H^1(\mathbb T)}\Vert\varphi_{0,x}\Vert_{H^{m+1}(\mathbb T)}\right\}\,.
\end{split}
\end{equation*}
Integration over $\mathbb R_t$ then gives
\begin{equation}\label{stima_G2_fin}
\Vert G_{2}\Vert_{L^2_\gamma(\mathbb R\times\mathbb T)}\le C\left\{\Vert\varphi_{0}\Vert_{L^\infty(\mathbb R;H^3(\mathbb T))}\Vert\varphi_{x}\Vert_{L^2_\gamma(\mathbb R;H^{m}(\mathbb T))}+\Vert\varphi_{x}\Vert_{L^2_\gamma(\mathbb R;H^1(\mathbb T))}\Vert\varphi_{0,x}\Vert_{L^\infty(\mathbb R;H^{m+1}(\mathbb T))}\right\}\,.
\end{equation}

\smallskip
\noindent
{\it The estimate of} $G_3:=-\left(\left[\partial_x^m\,;\,\mathbb H\left[\phi_{0,xx}\right]\right]\phi\right)_x$: for fixed $t\in\mathbb R$ we first compute by using again the commutator estimate \eqref{stima_comm_4} and the Sobolev continuity of $\mathbb H$:
\begin{equation}\label{stima_G3_0}
\begin{split}
\Vert & G_{3,\,\gamma}(t)\Vert_{L^2(\mathbb T)}=\Vert -\left(\left[\partial_x^m\,;\,\mathbb H\left[\phi_{0,xx}\right]\right]\phi_\gamma\right)_x\Vert_{L^2(\mathbb T)}\\
&\le\Vert\left[\partial_x^m\,;\,\mathbb H\left[\phi_{0,xxx}\right]\right]\phi_\gamma\Vert_{L^2(\mathbb T)}+\Vert \left[\partial_x^m\,;\,\mathbb H\left[\phi_{0,xx}\right]\right]\phi_{x, \gamma}\Vert_{L^2(\mathbb T)}\\
&=\Vert\left[\partial_x^m\,;\,\varphi_{0,xxx}\right]\phi_\gamma\Vert_{L^2(\mathbb T)}+\Vert \left[\partial_x^m\,;\,\varphi_{0,xx}\right]\phi_{x, \gamma}\Vert_{L^2(\mathbb T)}\\
&\le C\left\{\Vert\varphi_{0,xxx}\Vert_{L^\infty(\mathbb T)}\Vert\phi_\gamma\Vert_{H^m(\mathbb T)}+\Vert\varphi_{0,xxx}\Vert_{H^m(\mathbb T)}\Vert\phi_\gamma\Vert_{L^\infty(\mathbb T)}\right.\\
&\quad\left.+\Vert\varphi_{0,xx}\Vert_{L^\infty(\mathbb T)}\Vert\phi_{x,\gamma}\Vert_{H^m(\mathbb T)}+\Vert\varphi_{0,xx}\Vert_{H^m(\mathbb T)}\Vert\phi_{x,\gamma}\Vert_{L^\infty(\mathbb T)}\right\}\\
&\le C\left\{\Vert\varphi_{0,xxx}\Vert_{L^\infty(\mathbb T)}\Vert\varphi_\gamma\Vert_{H^m(\mathbb T)}+\Vert\varphi_{0,x}\Vert_{H^{m+2}(\mathbb T)}\Vert\phi_\gamma\Vert_{L^\infty(\mathbb T)}\right.\\
&\quad\left.+\Vert\varphi_{0,xx}\Vert_{L^\infty(\mathbb T)}\Vert\varphi_{x,\gamma}\Vert_{H^m(\mathbb T)}+\Vert\varphi_{0,x}\Vert_{H^{m+1}(\mathbb T)}\Vert\phi_{x,\gamma}\Vert_{L^\infty(\mathbb T)}\right\}\,.
\end{split}
\end{equation}
On the other hand from Poincar\'e's inequality (recall that $\varphi$ has zero spatial mean) and the Sobolev imbedding $H^1(\mathbb T)\hookrightarrow L^\infty(\mathbb T)$ we get
\begin{eqnarray}
\Vert\varphi_\gamma\Vert^2_{H^m(\mathbb T)}=\sum\limits_{k=0}^m\Vert\partial_x^k\varphi_\gamma\Vert^2_{L^2(\mathbb T)}\le C\sum\limits_{k=0}^m\Vert\partial_x^{k+1}\varphi_\gamma\Vert^2_{L^2(\mathbb T)}=C\Vert\varphi_{x,\gamma}\Vert^2_{H^m(\mathbb T)}\,;\label{magg_1}\\
\Vert\phi_\gamma\Vert_{L^\infty(\mathbb T)}\le C\Vert\varphi_\gamma\Vert_{H^1(\mathbb T)}\le C\Vert\varphi_{x,\gamma}\Vert_{L^2(\mathbb T)}\le C\Vert\varphi_{x,\gamma}\Vert_{H^1(\mathbb T)}\,;\label{magg_2}\\
\Vert\varphi_{0,xxx}\Vert_{L^\infty(\mathbb T)}\le C\Vert\varphi_{0,xxx}\Vert_{H^1(\mathbb T)}\le C\Vert\varphi_{0}\Vert_{H^4(\mathbb T)}\,;\label{magg_3}\\
\Vert\varphi_{0,xx}\Vert_{L^\infty(\mathbb T)}\le C\Vert\varphi_{0,xx}\Vert_{H^1(\mathbb T)}\le C\Vert\varphi_{0}\Vert_{H^3(\mathbb T)}\,.\label{magg_4}
\end{eqnarray}
Hence gathering estimates \eqref{stima_G3_0}, \eqref{magg_1}--\eqref{magg_4} we get
\begin{equation*}
\begin{split}
\Vert & G_{3,\,\gamma}(t)\Vert_{L^2(\mathbb T)}\le C\left\{\Vert\varphi_{0}\Vert_{H^4(\mathbb T)}\Vert\varphi_{x,\gamma}\Vert_{H^m(\mathbb T)}+\Vert\varphi_{0,x}\Vert_{H^{m+2}(\mathbb T)}\Vert\varphi_{x,\gamma}\Vert_{H^1(\mathbb T)}\right\}\,.
\end{split}
\end{equation*}
Finally, integration over $\mathbb R_t$ yields
\begin{equation}\label{stima_G3_fin}
\Vert G_{3}\Vert_{L^2_\gamma(\mathbb R\times\mathbb T)}\le C\left\{\Vert\varphi_{0}\Vert_{L^\infty(\mathbb R;H^4(\mathbb T))}\Vert\varphi_{x}\Vert_{L^2_\gamma(\mathbb R;H^{m}(\mathbb T))}+\Vert\varphi_{x}\Vert_{L^2_\gamma(\mathbb R;H^1(\mathbb T))}\Vert\varphi_{0,x}\Vert_{L^\infty(\mathbb R;H^{m+2}(\mathbb T))}\right\}\,.
\end{equation}

\smallskip
\noindent
{\it The estimate of} $G_4:=\left(\mathbb H\left[\left[\partial_x^m\,;\,\phi_{0,xx}\right]\phi\right]\right)_x$: for fixed $t\in\mathbb R$ we compute once again
\begin{equation*}
\begin{split}
\Vert & G_{4,\,\gamma}(t)\Vert_{L^2(\mathbb T)}=\Vert \left(\mathbb H\left[\left[\partial_x^m\,;\,\phi_{0,xx}\right]\phi_\gamma\right]\right)_x\Vert_{L^2(\mathbb T)}=\Vert\left(\left[\partial_x^m\,;\,\phi_{0,xx}\right]\phi_\gamma\right)_x  \Vert_{L^2(\mathbb T)}\\
&\le\Vert\left[\partial_x^m\,;\,\phi_{0,xxx}\right]\phi_\gamma\Vert_{L^2(\mathbb T)}+\Vert \left[\partial_x^m\,;\,\phi_{0,xx}\right]\phi_{x, \gamma}\Vert_{L^2(\mathbb T)}\,.
\end{split}
\end{equation*}
Then we argue as for $G_3$ (where the derivatives of $\varphi_0$ are replaced by the same derivatives of $\phi_0$) to get
\begin{equation}\label{stima_G4_fin}
\Vert G_{4}\Vert_{L^2_\gamma(\mathbb R\times\mathbb T)}\le C\left\{\Vert\varphi_{0}\Vert_{L^\infty(\mathbb R;H^4(\mathbb T))}\Vert\varphi_{x}\Vert_{L^2_\gamma(\mathbb R;H^{m}(\mathbb T))}+\Vert\varphi_{x}\Vert_{L^2_\gamma(\mathbb R;H^1(\mathbb T))}\Vert\varphi_{0,x}\Vert_{L^\infty(\mathbb R;H^{m+2}(\mathbb T))}\right\}\,.
\end{equation}

\smallskip
\noindent
{\it The estimate of} $G_5:=\left(-\left[\left[\partial_x^m\,;\,\phi_0\right]\,;\,\mathbb H\right]\phi_{xx}\right)_x$: for fixed $t\in\mathbb R$ we compute
\begin{equation}\label{stima_G5_0}
\Vert G_{5,\gamma}(t)\Vert_{L^2(\mathbb  T)}\le \left\Vert\left[\left[\partial_x^m\,;\,\phi_{0,x}\right]\,;\,\mathbb H\right]\phi_{xx,\gamma}\right\Vert_{L^2(\mathbb T)}+\left\Vert\left[\left[\partial_x^m\,;\,\phi_0\right]\,;\,\mathbb H\right]\phi_{xxx,\gamma}\right\Vert_{L^2(\mathbb T)}\,.
\end{equation}
Now we treat separately the two $L^2-$norms in the right-hand side above. As for the first norm, we expand the commutator with $\mathbb H$ and use the commutator estimate \eqref{stima_comm_5} to get
\begin{equation}\label{stima_G5_norma1}
\begin{split}
\left\Vert\left[\left[\partial_x^m\,;\,\phi_{0,x}\right]\,;\,\mathbb H\right]\phi_{xx,\gamma}\right\Vert_{L^2(\mathbb T)}  & \le \Vert \left[\partial_x^m\,;\,\phi_{0,x}\right]\mathbb H\left[\phi_{xx,\gamma}\right]\Vert_{L^2(\mathbb T)}+\Vert\mathbb H\left[\left[\partial_x^m\,;\,\phi_{0,x}\right]\phi_{xx,\gamma}\right]\Vert_{L^2(\mathbb T)}\\
& =\Vert \left[\partial_x^m\,;\,\phi_{0,x}\right]\varphi_{xx,\gamma}\Vert_{L^2(\mathbb T)}+\Vert\left[\partial_x^m\,;\,\phi_{0,x}\right]\phi_{xx,\gamma}\Vert_{L^2(\mathbb T)}\\
& \le C\left\{\Vert\phi_{0,xx}\Vert_{L^\infty(\mathbb T)}\Vert\varphi_{xx,\gamma}\Vert_{H^{m-1}(\mathbb T)}+\Vert\phi_{0,x}\Vert_{H^{m}(\mathbb T)}\Vert\varphi_{xx,\gamma}\Vert_{L^\infty(\mathbb  T)}\right.\\
&\qquad\left. +\Vert\phi_{0,xx}\Vert_{L^\infty(\mathbb T)}\Vert\phi_{xx,\gamma}\Vert_{H^{m-1}(\mathbb T)}+\Vert\phi_{0,x}\Vert_{H^{m}(\mathbb T)}\Vert\phi_{xx,\gamma}\Vert_{L^\infty(\mathbb  T)} \right\}\\
&\le C\left\{\Vert\phi_{0,xx}\Vert_{L^\infty(\mathbb T)}\Vert\varphi_{x,\gamma}\Vert_{H^{m}(\mathbb T)}+\Vert\phi_{0,x}\Vert_{H^{m}(\mathbb T)}\Vert\varphi_{xx,\gamma}\Vert_{L^\infty(\mathbb  T)}\right.\\
&\qquad\left. +\Vert\phi_{0,x}\Vert_{H^{m}(\mathbb T)}\Vert\phi_{xx,\gamma}\Vert_{L^\infty(\mathbb  T)}\right\}\\
&\le C\left\{\Vert\varphi_{0}\Vert_{H^3(\mathbb T)}\Vert\varphi_{x,\gamma}\Vert_{H^{m}(\mathbb T)}+\Vert\varphi_{0,x}\Vert_{H^{m}(\mathbb T)}\Vert\varphi_{x,\gamma}\Vert_{H^2(\mathbb  T)}\right\}\,.
\end{split}
\end{equation}
Concerning the second $L^2-$norm in the right-hand side of \eqref{stima_G5_0}, we use Lemma \ref{lemma_comm_2} to get
\begin{equation}\label{stima_G5_norma2}
\left\Vert\left[\left[\partial_x^m\,;\,\phi_0\right]\,;\,\mathbb H\right]\phi_{xxx,\gamma}\right\Vert_{L^2(\mathbb T)}\le C\Vert\phi_{0,x}\Vert_{H^m(\mathbb T)}\Vert\phi_{xx,\gamma}\Vert_{H^1(\mathbb T)}\le C\Vert\varphi_{0,x}\Vert_{H^m(\mathbb T)}\Vert\varphi_{x,\gamma}\Vert_{H^2(\mathbb T)}.
\end{equation}
Summing up \eqref{stima_G5_norma1}, \eqref{stima_G5_norma2} we then obtain
\begin{equation*}
\Vert G_{5,\gamma}(t)\Vert_{L^2(\mathbb  T)}\le C\left\{\Vert\varphi_{0}\Vert_{H^3(\mathbb T)}\Vert\varphi_{x,\gamma}\Vert_{H^{m}(\mathbb T)}+\Vert\varphi_{0,x}\Vert_{H^{m}(\mathbb T)}\Vert\varphi_{x,\gamma}\Vert_{H^2(\mathbb  T)}\right\}
\end{equation*}
and integrating in $\mathbb R_t$
\begin{equation}\label{stima_G5_fin}
\Vert G_{5}\Vert_{L^2_\gamma(\mathbb R\times\mathbb  T)}\le C\left\{\Vert\varphi_{0}\Vert_{L^\infty(\mathbb R;H^3(\mathbb T))}\Vert\varphi_{x}\Vert_{L^2_\gamma(\mathbb R;H^{m}(\mathbb T))}+\Vert\varphi_{0,x}\Vert_{L^\infty(\mathbb R;H^{m}(\mathbb T))}\Vert\varphi_{x}\Vert_{L^2_\gamma(\mathbb R; H^2(\mathbb  T))}\right\}\,.
\end{equation}
Using \eqref{stima_G1_fin}, \eqref{stima_G2_fin}, \eqref{stima_G3_fin}, \eqref{stima_G4_fin}, \eqref{stima_G5_fin} to estimate the right-hand side of \eqref{stima_apriori_2_0} we have
\begin{equation}\label{stima_apriori_2_1}
\begin{split}
\gamma &\left\{\Vert\partial_x^m\varphi_{t}\Vert^2_{L^2_\gamma(\mathbb R\times\mathbb T)}+\Vert\partial_x^m\varphi_{x}\Vert^2_{L^2_\gamma(\mathbb R\times\mathbb T)}\right\}\le \frac{C}{\gamma}\left\{\Vert\varphi_{0}\Vert^2_{L^\infty(\mathbb R;H^4(\mathbb T))}\Vert\varphi_{x}\Vert^2_{L^2_\gamma(\mathbb R;H^{m}(\mathbb T))}\right.\\
&\qquad \left.+\Vert\varphi_{0,x}\Vert^2_{L^\infty(\mathbb R;H^{m+2}(\mathbb T))}\Vert\varphi_{x}\Vert^2_{L^2_\gamma(\mathbb R;H^{2}(\mathbb T))}+\Vert \partial_x^m g\Vert^2_{L^2_\gamma(\mathbb R\times\mathbb T)}\right\}\,,\qquad\forall\,\gamma\ge\gamma_0\,.
\end{split}
\end{equation}
Because of \eqref{ident_norme} and $\varphi$ has spatial zero mean, the left-hand side of \eqref{stima_apriori_2_1} is equivalent, uniformly in $\gamma$, to $\gamma\left\{\Vert\varphi_{t}\Vert^2_{L^2_\gamma(\mathbb R; H^m(\mathbb T))}+\Vert\varphi_{x}\Vert^2_{L^2_\gamma(\mathbb R;H^m(\mathbb T))}\right\}$, hence the estimate can be restated as
\begin{equation*}
\begin{split}
\gamma & \left\{\Vert\varphi_{t}\Vert^2_{L^2_\gamma(\mathbb R; H^m(\mathbb T))}+\Vert\varphi_{x}\Vert^2_{L^2_\gamma(\mathbb R;H^m(\mathbb T))}\right\}\le \frac{C}{\gamma}\left\{\Vert\varphi_{0}\Vert^2_{L^\infty(\mathbb R;H^4(\mathbb T))}\Vert\varphi_{x}\Vert^2_{L^2_\gamma(\mathbb R;H^{m}(\mathbb T))}\right.\\
&\qquad\left.+\Vert\varphi_{0,x}\Vert^2_{L^\infty(\mathbb R;H^{m+2}(\mathbb T))}\Vert\varphi_{x}\Vert^2_{L^2_\gamma(\mathbb R;H^{2}(\mathbb T))}+\Vert g\Vert^2_{L^2_\gamma(\mathbb R; H^m_\gamma(\mathbb T))}\right\}\,,\qquad\forall\,\gamma\ge\gamma_0\,.
\end{split}
\end{equation*}
Then the first term in the right-hand side can be absorbed into the left-hand side, provided that $\gamma$ is larger than some positive $\widetilde{\gamma}_0=\widetilde{\gamma}_0(\Vert\varphi_0\Vert_{L^\infty(\mathbb R; H^4(\mathbb T))})$, to get
\begin{equation}\label{stima_apriori_2_2}
\begin{split}
\gamma & \left\{\Vert\varphi_{t}\Vert^2_{L^2_\gamma(\mathbb R; H^m(\mathbb T))}+\Vert\varphi_{x}\Vert^2_{L^2_\gamma(\mathbb R;H^m(\mathbb T))}\right\}\\
&\le \frac{C}{\gamma}\left\{\Vert\varphi_{0,x}\Vert^2_{L^\infty(\mathbb R;H^{m+2}(\mathbb T))}\Vert\varphi_{x}\Vert^2_{L^2_\gamma(\mathbb R;H^{2}(\mathbb T))}+\Vert g\Vert^2_{L^2_\gamma(\mathbb R; H^m(\mathbb T))}\right\}\,,\qquad\forall\,\gamma\ge\max\{\gamma_0,\widetilde{\gamma}_0\}\,.
\end{split}
\end{equation}
Now we write \eqref{stima_apriori_2_2} for $m=2$
\begin{equation*}
\begin{split}
\gamma & \left\{\Vert\varphi_{t}\Vert^2_{L^2_\gamma(\mathbb R; H^2(\mathbb T))}+\Vert\varphi_{x}\Vert^2_{L^2_\gamma(\mathbb R;H^2(\mathbb T))}\right\}\\
&\le \frac{C}{\gamma}\left\{\Vert\varphi_{0,x}\Vert^2_{L^\infty(\mathbb R;H^{4}(\mathbb T))}\Vert\varphi_{x}\Vert^2_{L^2_\gamma(\mathbb R;H^{2}(\mathbb T))}+\Vert g\Vert^2_{L^2_\gamma(\mathbb R; H^2(\mathbb T))}\right\}\\
&\le \frac{C}{\gamma}\left\{\Vert\varphi_{0}\Vert^2_{L^\infty(\mathbb R;H^{5}(\mathbb T))}\Vert\varphi_{x}\Vert^2_{L^2_\gamma(\mathbb R;H^{2}(\mathbb T))}+\Vert g\Vert^2_{L^2_\gamma(\mathbb R; H^2(\mathbb T))}\right\}
\end{split}
\end{equation*}
and again we absorb into the left-hand side the first term in the right-hand side above for $\gamma\ge\widetilde{\gamma_1}=\widetilde{\gamma}_1(\Vert\varphi_{0}\Vert_{L^\infty(\mathbb R;H^{5}(\mathbb T))})$ to find that
\begin{equation}\label{stima_apriori_2_4}
\begin{split}
\gamma & \left\{\Vert\varphi_{t}\Vert^2_{L^2_\gamma(\mathbb R; H^2(\mathbb T))}+\Vert\varphi_{x}\Vert^2_{L^2_\gamma(\mathbb R;H^2(\mathbb T))}\right\}\le \frac{C}{\gamma}\Vert g\Vert^2_{L^2_\gamma(\mathbb R; H^2(\mathbb T))}\,.
\end{split}
\end{equation}
Then we use \eqref{stima_apriori_2_4} to estimate the norm $\Vert\varphi_{x}\Vert^2_{L^2_\gamma(\mathbb R;H^2(\mathbb T))}$ in the right-hand side of \eqref{stima_apriori_2_2} to find that
\begin{equation*}
\begin{split}
\gamma & \left\{\Vert\varphi_{t}\Vert^2_{L^2_\gamma(\mathbb R; H^m(\mathbb T))}+\Vert\varphi_{x}\Vert^2_{L^2_\gamma(\mathbb R;H^m(\mathbb T))}\right\}\\
&\le \frac{C}{\gamma}\left\{\Vert\varphi_{0,x}\Vert^2_{L^\infty(\mathbb R;H^{m+2}(\mathbb T))}\Vert g\Vert^2_{L^2_\gamma(\mathbb R; H^2(\mathbb T))} +\Vert g\Vert^2_{L^2_\gamma(\mathbb R; H^m(\mathbb T))}\right\}\,,\quad\forall\,\gamma\ge\max\{\gamma_0,\widetilde{\gamma}_0,\widetilde{\gamma}_1\}=:\gamma_1\,.
\end{split}
\end{equation*}
\end{proof}

Under the same assumptions of Proposition \ref{prop_stima_tame} we obtain an estimate of the $L^2_\gamma(\mathbb R; H^{m-1}(\mathbb T))$-norm of $\varphi^\prime_{tt}$, namely we prove the following
\begin{proposition}\label{prop_stima_tame_2}
Under the same assumptions of Proposition \ref{prop_stima_tame} there exists a positive constant $C_1=C_1\left(\mu,\delta, m\right)$ such that for all sufficiently smooth functions $\varphi^\prime:\mathbb R\times\mathbb T\rightarrow\mathbb R$, with zero spatial mean, and for all $\gamma\ge 1$ the following estimate holds true
\begin{equation}\label{stima_dm-1phitt}
\begin{split}
&\Vert\varphi^\prime_{tt}\Vert_{L^2_\gamma(\mathbb R;H^{m-1}(\mathbb T))}\le C_1\left\{\Vert\varphi^\prime_{x}\Vert_{L^2_\gamma(\mathbb R;H^m(\mathbb T))}+\Vert\varphi_0\Vert_{L^\infty(\mathbb R;H^3(\mathbb T))}\Vert\varphi^\prime_{x}\Vert_{L^2_\gamma(\mathbb R;H^m(\mathbb T))}\right.\\
&\qquad\qquad\qquad\quad \left.+\Vert\varphi_{0,x}\Vert_{L^\infty(\mathbb R;H^{m}(\mathbb T))}\Vert\varphi^\prime_{x}\Vert_{L^2_\gamma(\mathbb R;H^2(\mathbb T))}+\Vert g\Vert_{L^2_\gamma(\mathbb R;H^{m-1}(\mathbb T))}\right\}\,,
\end{split}
\end{equation}
where $g:=\mathbb{L}^\prime[\varphi_0]\varphi^\prime$.
\end{proposition}
\begin{proof}
As in the proof of Proposition \ref{prop_stima_tame}, to simplify the notation we drop the superscript ${}^\prime$ in the unknown function $\varphi^\prime$. Because of \eqref{ident_norme}, it is enough providing an estimate of the $L^2_\gamma-$norm of $\partial^{m-1}_x\varphi_{tt}$; this can be obtained by applying the differential operator $\partial_x^{m-1}$ to the linearized equation \eqref{equazione_lin1} and solving the resulting equation for $\partial^{m-1}_x\varphi_{tt}$.

Applying $\partial^{m-1}_x$ to \eqref{equazione_lin1} we get
\begin{equation}\label{eq_dm-1phitt}
\begin{split}
\partial^{m-1}_x&\varphi_{tt}=\mu\partial^m_x\varphi_x-2\phi_{0,x}\partial_x^m\varphi_x-2\left[\partial^{m-1}_x\,;\,\phi_{0,x}\right]\varphi_{xx}+2\left[\mathbb H\,;\,\phi_{0,x}\right]\partial_x^m\phi_x\\
&+2\left[\partial^{m-1}_x\,;\,\left[\mathbb H\,;\,\phi_{0,x}\right]\right]\phi_{xx}+2\mathbb H\left[\phi_{0,xx}\partial_x^{m-1}\phi_x+\left[\partial_x^{m-1}\,;\,\phi_{0,xx}\right]\phi_x\right]\\
&-\left(\left[\phi\,;\,\mathbb H\right]\partial^m_x\phi_{0,x}+\left[\partial^{m-1}_x\,;\,\left[\phi\,;\,\mathbb H\right]\right]\phi_{0,xx}+\left[\phi_0\,;\,\mathbb H\right]\partial_x^m\phi_x+\left[\partial_x^{m-1}\,;\,\left[\phi_0\,;\,\mathbb H\right]\right]\phi_{xx}\right)_x\\
&+\partial_x^{m-1}g=\sum\limits_{j=1}^{11}G'_j+\partial^{m-1}_x g\,.
\end{split}
\end{equation}
In order to obtain the desired estimate, we provide a suitable bound of each term in the right-hand side above. Let $t\in\mathbb R$ and $\gamma\ge 1$ be arbitrarily fixed. In all the following estimates the $L^2-$continuity of $\mathbb H$ is used.

{\it Estimate of $G'_1$}:
\begin{equation}\label{stima_1}
\Vert G'_{1,\gamma}(t)\Vert_{L^2(\mathbb T)}\le\vert\mu\vert\Vert\partial^m_x\varphi_{x,\gamma}\Vert_{L^2(\mathbb T)}\le \vert\mu\vert\Vert\varphi_{x,\gamma}\Vert_{H^m(\mathbb T)}\,.
\end{equation}

{\it Estimate of $G'_2$}: H\"{o}lder's inequality and Sobolev's imbedding $H^1(\mathbb T)\hookrightarrow L^\infty(\mathbb T)$ yield
\begin{equation}\label{stima_2}
\begin{split}
\Vert & G'_{2,\gamma}(t)\Vert_{L^2(\mathbb T)}=\Vert -2\phi_{0,x}\partial_x^m\varphi_{x,\gamma}\Vert_{L^2(\mathbb T)}\le C\Vert\phi_{0,x}\Vert_{L^\infty(\mathbb T)}\Vert\partial^m_x\varphi_{x,\gamma}\Vert_{L^2(\mathbb T)}\\
&\le C\Vert\varphi_{0,x}\Vert_{H^1(\mathbb T)}\Vert\varphi_{x,\gamma}\Vert_{H^m(\mathbb T)}\le C\Vert\varphi_0\Vert_{H^2(\mathbb T)}\Vert\varphi_{x,\gamma}\Vert_{H^m(\mathbb T)}\,.
\end{split}
\end{equation}

{\it Estimate of $G'_3$}: estimate \eqref{stima_comm_4} and Sobolev's imbedding $H^1(\mathbb T)\hookrightarrow L^\infty(\mathbb T)$ yield
\begin{equation}\label{stima_3}
\begin{split}
\Vert & G'_{3,\gamma}(t)\Vert_{L^2(\mathbb T)}=\Vert -2\left[\partial_x^{m-1}\,;\,\phi_{0,x}\right]\varphi_{xx,\gamma}\Vert_{L^2(\mathbb T)}\\
&\le C\left\{\Vert\phi_{0,x}\Vert_{L^\infty(\mathbb T)}\Vert\varphi_{xx,\gamma}\Vert_{H^{m-1}(\mathbb T)}+\Vert\phi_{0,x}\Vert_{H^{m-1}(\mathbb T)}\Vert\varphi_{xx,\gamma}\Vert_{L^\infty(\mathbb T)}\right\}\\
&\le C\left\{\Vert\varphi_{0,x}\Vert_{H^1(\mathbb T)}\Vert\varphi_{x,\gamma}\Vert_{H^{m}(\mathbb T)}+\Vert\varphi_{0,x}\Vert_{H^{m-1}(\mathbb T)}\Vert\varphi_{xx,\gamma}\Vert_{H^1(\mathbb T)}\right\}\\
&\le C\left\{\Vert\varphi_{0}\Vert_{H^2(\mathbb T)}\Vert\varphi_{x,\gamma}\Vert_{H^{m}(\mathbb T)}+\Vert\varphi_{0,x}\Vert_{H^{m-1}(\mathbb T)}\Vert\varphi_{x,\gamma}\Vert_{H^2(\mathbb T)}\right\}\,.
\end{split}
\end{equation}

{\it Estimate of $G'_4$}: applying estimate \eqref{stima_comm_1} gives
\begin{equation}\label{stima_4}
\begin{split}
\Vert & G'_{4,\gamma}(t)\Vert_{L^2(\mathbb T)}=\Vert 2\left[\mathbb H\,;\,\phi_{0,x}\right]\partial^m_x\phi_{x,\gamma}\Vert_{L^2(\mathbb T)}\le C\Vert\phi_{0,x}\Vert_{H^1(\mathbb T)}\Vert\partial^m_x\phi_{x,\gamma}\Vert_{L^{2}(\mathbb T)}\\
&\le C\Vert\varphi_{0}\Vert_{H^2(\mathbb T)}\Vert\varphi_{x,\gamma}\Vert_{H^{m}(\mathbb T)}\,.
\end{split}
\end{equation}

{\it Estimate of $G'_5$}: in view of Lemma \ref{lemma_comm_2} and using that $\partial^{m-1}_x$ and $\mathbb H$ commute, we get
\begin{equation}\label{stima_5}
\begin{split}
\Vert & G'_{5,\gamma}(t)\Vert_{L^2(\mathbb T)}=\Vert 2\left[\partial^{m-1}_x\,;\,\left[\mathbb H\,;\,\phi_{0,x}\right]\right]\phi_{xx,\gamma}\Vert_{L^2(\mathbb T)}=\Vert -2\left[\left[\partial^{m-1}_x\,;\,\phi_{0,x}\right]\,;\,\mathbb H\right]\phi_{xx,\gamma}\Vert_{L^2(\mathbb T)}\\
&\le C\Vert\phi_{0,xx}\Vert_{H^{m-1}(\mathbb T)}\Vert\phi_{x,\gamma}\Vert_{H^{1}(\mathbb T)}\le C\Vert\varphi_{0,x}\Vert_{H^{m}(\mathbb T)}\Vert\varphi_{x,\gamma}\Vert_{H^{1}(\mathbb T)}\,.
\end{split}
\end{equation}

{\it Estimate of $G'_6$}: H\"{o}lder's inequality and Sobolev's imbedding $H^1(\mathbb T)\hookrightarrow L^\infty(\mathbb T)$ yield
\begin{equation}\label{stima_6}
\begin{split}
\Vert & G'_{6,\gamma}(t)\Vert_{L^2(\mathbb T)}=\Vert 2\mathbb H\left[\phi_{0,xx}\partial^{m-1}_x\phi_{x,\gamma}\right]\Vert_{L^2(\mathbb T)}\le 2\Vert\phi_{0,xx}\partial^{m-1}_x\phi_{x,\gamma}\Vert_{L^2(\mathbb T)}\\
&\le 2\Vert\phi_{0,xx}\Vert_{L^{\infty}(\mathbb T)}\Vert\partial^{m-1}_x\phi_{x,\gamma}\Vert_{L^{2}(\mathbb T)}\le C\Vert\varphi_{0,xx}\Vert_{H^{1}(\mathbb T)}\Vert\varphi_{x,\gamma}\Vert_{H^{m-1}(\mathbb T)}\\
&\le C\Vert\varphi_{0}\Vert_{H^{3}(\mathbb T)}\Vert\varphi_{x,\gamma}\Vert_{H^{m-1}(\mathbb T)}\,.
\end{split}
\end{equation}

{\it Estimate of $G'_7$}: estimate \eqref{stima_comm_4} and Sobolev's imbedding $H^1(\mathbb T)\hookrightarrow L^\infty(\mathbb T)$ yield
\begin{equation}\label{stima_7}
\begin{split}
\Vert & G'_{7,\gamma}(t)\Vert_{L^2(\mathbb T)}=\Vert 2\mathbb H\left[\left[\partial_x^{m-1}\,;\,\phi_{0,xx}\right]\phi_{x,\gamma}\right]\Vert_{L^2(\mathbb T)}\le 2\Vert \left[\partial_x^{m-1}\,;\,\phi_{0,xx}\right]\phi_{x,\gamma}\Vert_{L^2(\mathbb T)}\\
&\le C\left\{\Vert\phi_{0,xx}\Vert_{L^\infty(\mathbb T)}\Vert\phi_{x,\gamma}\Vert_{H^{m-1}(\mathbb T)}+\Vert\phi_{0,xx}\Vert_{H^{m-1}(\mathbb T)}\Vert\phi_{x,\gamma}\Vert_{L^\infty(\mathbb T)}\right\}\\
&\le C\left\{\Vert\varphi_{0,xx}\Vert_{H^1(\mathbb T)}\Vert\varphi_{x,\gamma}\Vert_{H^{m-1}(\mathbb T)}+\Vert\varphi_{0,xx}\Vert_{H^{m-1}(\mathbb T)}\Vert\varphi_{x,\gamma}\Vert_{H^1(\mathbb T)}\right\}\\
&\le C\left\{\Vert\varphi_{0}\Vert_{H^3(\mathbb T)}\Vert\varphi_{x,\gamma}\Vert_{H^{m-1}(\mathbb T)}+\Vert\varphi_{0,x}\Vert_{H^{m}(\mathbb T)}\Vert\varphi_{x,\gamma}\Vert_{H^1(\mathbb T)}\right\}\,.
\end{split}
\end{equation}

{\it Estimate of $G'_8$}: Leibniz's formula and estimates \eqref{stima_comm_1}, \eqref{stima_comm_2} give
\begin{equation}\label{stima_8}
\begin{split}
\Vert & G'_{8,\gamma}(t)\Vert_{L^2(\mathbb T)}=\Vert -\left(\left[\phi_\gamma\,;\,\mathbb H\right]\partial_x^m\phi_{0,x}\right)_x\Vert_{L^2(\mathbb T)}\\
&\le \Vert \left[\phi_{x,\gamma}\,;\,\mathbb H\right]\partial_x^m\phi_{0,x}\Vert_{L^2(\mathbb T)}+\Vert \left[\phi_\gamma\,;\,\mathbb H\right]\partial_x\partial_x^{m}\phi_{0,x}\Vert_{L^2(\mathbb T)}\\
&\le C\left\{\Vert\phi_{x,\gamma}\Vert_{H^1(\mathbb T)}\Vert\partial_x^m\phi_{0,x}\Vert_{L^{2}(\mathbb T)}+\Vert\phi_{x,\gamma}\Vert_{H^{1}(\mathbb T)}\Vert\partial_x^m\phi_{0,x}\Vert_{L^2(\mathbb T)}\right\}\\
&\le C\Vert\varphi_{x,\gamma}\Vert_{H^1(\mathbb T)}\Vert\varphi_{0,x}\Vert_{H^{m}(\mathbb T)}\,.
\end{split}
\end{equation}

{\it Estimate of $G'_9$}: by Leibniz's formula, Lemma \ref{lemma_comm_2} and using that $\partial^{m-1}_x$ and $\mathbb H$ commute, we get
\begin{equation}\label{stima_9}
\begin{split}
\Vert & G'_{9,\gamma}(t)\Vert_{L^2(\mathbb T)}=\left\Vert -\left(\left[\partial^{m-1}_x\,;\,\left[\phi_\gamma\,;\,\mathbb H\right]\right]\phi_{0,xx}\right)_x\right\Vert_{L^2(\mathbb T)}=\left\Vert -\left(\left[\left[\partial_x^{m-1}\,;\,\phi_{\gamma}\right]\,;\,\mathbb H\right]\phi_{0,xx}\right)_x \right\Vert_{L^2(\mathbb T)}\\
&\le \left\Vert\left[\left[\partial_x^{m-1}\,;\,\phi_{x,\gamma}\right]\,;\,\mathbb H\right]\phi_{0,xx} \right\Vert_{L^2(\mathbb T)}+\left\Vert \left[\left[\partial_x^{m-1}\,;\,\phi_{\gamma}\right]\,;\,\mathbb H\right]\phi_{0,xxx} \right\Vert_{L^2(\mathbb T)}\\
&\le C\Vert\phi_{xx,\gamma}\Vert_{H^{m-1}(\mathbb T)}\Vert\phi_{0,x}\Vert_{H^1(\mathbb T)}+C\Vert\phi_{x,\gamma}\Vert_{H^{m-1}(\mathbb T)}\Vert\phi_{0,xx}\Vert_{H^1(\mathbb T)}\le C\Vert\varphi_{x,\gamma}\Vert_{H^m(\mathbb T)}\Vert\varphi_0\Vert_{H^3(\mathbb T)}\,.
\end{split}
\end{equation}

{\it Estimate of $G'_{10}$}: Leibniz's formula and estimates \eqref{stima_comm_1}, \eqref{stima_comm_2} give
\begin{equation}\label{stima_10}
\begin{split}
\Vert & G'_{10,\gamma}(t)\Vert_{L^2(\mathbb T)}=\Vert -\left(\left[\phi_0\,;\,\mathbb H\right]\partial_x^m\phi_{x,\gamma}\right)_x\Vert_{L^2(\mathbb T)}\\
&\le \Vert \left[\phi_{0,x}\,;\,\mathbb H\right]\partial_x^m\phi_{x,\gamma}\Vert_{L^2(\mathbb T)}+\Vert \left[\phi_0\,;\,\mathbb H\right]\partial_x\partial_x^{m}\phi_{x,\gamma}\Vert_{L^2(\mathbb T)}\\
&\le C\Vert\phi_{0,x}\Vert_{H^1(\mathbb T)}\Vert\partial_x^m\phi_{x,\gamma}\Vert_{L^{2}(\mathbb T)}\le C\Vert\varphi_{0}\Vert_{H^2(\mathbb T)}\Vert\varphi_{x,\gamma}\Vert_{H^{m}(\mathbb T)}\,.
\end{split}
\end{equation}

{\it Estimate of $G'_{11}$}: by Leibniz's formula, Lemma \ref{lemma_comm_2} and using that $\partial^{m-1}_x$ and $\mathbb H$ commute, we get
\begin{equation}\label{stima_11}
\begin{split}
\Vert & G'_{11,\gamma}(t)\Vert_{L^2(\mathbb T)}=\left\Vert -\left(\left[\partial^{m-1}_x\,;\,\left[\phi_0\,;\,\mathbb H\right]\right]\phi_{xx,\gamma}\right)_x\right\Vert_{L^2(\mathbb T)}=\left\Vert -\left(\left[\left[\partial_x^{m-1}\,;\,\phi_{0}\right]\,;\,\mathbb H\right]\phi_{xx,\gamma}\right)_x \right\Vert_{L^2(\mathbb T)}\\
&\le \left\Vert\left[\left[\partial_x^{m-1}\,;\,\phi_{0,x}\right]\,;\,\mathbb H\right]\phi_{xx,\gamma} \right\Vert_{L^2(\mathbb T)}+\left\Vert \left[\left[\partial_x^{m-1}\,;\,\phi_{0}\right]\,;\,\mathbb H\right]\phi_{xxx,\gamma} \right\Vert_{L^2(\mathbb T)}\\
&\le C\Vert\phi_{0,xx}\Vert_{H^{m-1}(\mathbb T)}\Vert\phi_{x,\gamma}\Vert_{H^1(\mathbb T)}+C\Vert\phi_{0,x}\Vert_{H^{m-1}(\mathbb T)}\Vert\phi_{xx,\gamma}\Vert_{H^1(\mathbb T)}\le C\Vert\varphi_{0,x}\Vert_{H^m(\mathbb T)}\Vert\varphi_{x,\gamma}\Vert_{H^2(\mathbb T)}\,.
\end{split}
\end{equation}
Using \eqref{stima_1}--\eqref{stima_11} above to estimate the right-hand side of \eqref{eq_dm-1phitt} yields
\begin{equation*}
\begin{split}
&\Vert\partial^{m-1}_x\varphi_{tt,\gamma}\Vert_{L^2(\mathbb T)}\le\vert\mu\vert\Vert\varphi_{x,\gamma}\Vert_{H^m(\mathbb T)}+C\left\{\Vert\varphi_0\Vert_{H^3(\mathbb T)}\Vert\varphi_{x,\gamma}\Vert_{H^m(\mathbb T)}\right.\\
&\qquad\qquad\qquad\quad \left.+\Vert\varphi_{0,x}\Vert_{H^{m}(\mathbb T)}\Vert\varphi_{x,\gamma}\Vert_{H^2(\mathbb T)}\right\}+\Vert \partial^{m-1}_xg_\gamma\Vert_{L^{2}(\mathbb T)}\,,
\end{split}
\end{equation*}
then integration in time and H\"older's inequality give
\begin{equation*}
\begin{split}
&\Vert\partial^{m-1}_x\varphi_{tt}\Vert_{L^2_\gamma(\mathbb R\times\mathbb T)}\le\vert\mu\vert\Vert\varphi_{x}\Vert_{L^2_\gamma(\mathbb R;H^m(\mathbb T))}+C\left\{\Vert\varphi_0\Vert_{L^\infty(\mathbb R;H^3(\mathbb T))}\Vert\varphi_{x}\Vert_{L^2_\gamma(\mathbb R;H^m(\mathbb T))}\right.\\
&\qquad\qquad\qquad\quad \left.+\Vert\varphi_{0,x}\Vert_{L^\infty(\mathbb R;H^{m}(\mathbb T))}\Vert\varphi_{x}\Vert_{L^2_\gamma(\mathbb R;H^2(\mathbb T))}\right\}+\Vert g\Vert_{L^2_\gamma(\mathbb R;H^{m-1}(\mathbb T))}\,,
\end{split}
\end{equation*}
which provides the desired result, in view of \eqref{ident_norme}.
\end{proof}
\section{Estimate of the second order derivative of $\mathbb{L}$}\label{sec_2der}
In order to apply the Nash-Moser method, we need to have a suitable estimate for the second order derivative of the nonlinear operator \eqref{operatore_nonlin} at a given state $\varphi_0$. From \eqref{operatore_lin1} one computes
\begin{equation*}
\begin{split}
\mathbb L^{\prime\prime}&[\varphi_0](\varphi, \psi):=\frac{d}{d\varepsilon}\left\{\mathbb L^\prime[\varphi_0+\varepsilon\psi]\varphi\right\}_{\vert\varepsilon=0}\\
&=\frac{d}{d\varepsilon}\big\{\varphi_{tt}-\mu\varphi_{xx}-\big(2\mathbb H\left[(\phi_{0,x}+\varepsilon\Psi_x)\phi_x\right]-\left[\phi\,;\,\mathbb H\right]\left(\phi_{0,xx}+\varepsilon\Psi_{xx}\right)-\left[\phi_0+\varepsilon\Psi\,;\,\mathbb H\right]\phi_{xx}\big)_x\big\}_{\vert\varepsilon=0}\\
&=\big(-2\mathbb H\left[\Psi_x\phi_x\right]+\left[\phi\,;\,\mathbb H\right]\Psi_{xx}+\left[\Psi\,;\,\mathbb H\right]\phi_{xx}\big)_x\,,
\end{split}
\end{equation*}
where it is set
\begin{equation*}
\phi_0:=\mathbb H[\varphi_0]\,,\quad\phi:=\mathbb H[\varphi]\,,\quad \Psi:=\mathbb H[\psi]\,.
\end{equation*}
We look now for an estimate of the second derivative $\mathbb L^{\prime\prime}[\varphi_0](\varphi, \psi)$ in the space $L^2_\gamma(\mathbb R; H^m(\mathbb T))$, with given integer $m$ and $\gamma$ sufficiently large. To this end we need to provide a suitable estimate of each term involved in the expression of  $\mathbb L^{\prime\prime}[\varphi_0](\varphi, \psi)$ above.

{\em Estimate for the 1st term $\left(-2\mathbb H\left[\Psi_x\phi_x\right]\right)_x$:} for fixed $t\in\mathbb R$ and $\gamma\ge 1$, we first commute $\mathbb H$ and $\partial_x$, then we use the continuity of $\mathbb H$ in $H^m(\mathbb T)$, Leibniz's formula, Lemma \ref{lemma_comm_1} (estimate \eqref{stima_prod_4}) and the Sobolev imbedding $H^1(\mathbb T)\hookrightarrow L^\infty(\mathbb T)$ to get
\begin{equation}\label{stima_punt_1}
\begin{split}
\Vert &\left(-2\mathbb H\left[\Psi_x\phi_x\right]\right)_x\Vert_{H^m(\mathbb T)}\le C\Vert\left(\Psi_x\phi_x\right)_x\Vert_{H^m(\mathbb T)}\le C\left\{\Vert\Psi_{xx}\phi_x\Vert_{H^m(\mathbb T)}+\Vert\Psi_{x}\phi_{xx}\Vert_{H^m(\mathbb T)}\right\}\\
&\le C\left\{\Vert\Psi_{xx}\Vert_{H^m(\mathbb T)}\Vert\phi_{x}\Vert_{L^\infty(\mathbb T)}+\Vert\Psi_{xx}\Vert_{L^\infty(\mathbb T)}\Vert\phi_{x}\Vert_{H^m(\mathbb T)}+\Vert\phi_{xx}\Vert_{H^m(\mathbb T)}\Vert\Psi_{x}\Vert_{L^\infty(\mathbb T)}\right.\\
&\qquad\left. +\Vert\phi_{xx}\Vert_{L^\infty(\mathbb T)}\Vert\Psi_{x}\Vert_{H^m(\mathbb T)}\right\}\\
&\le C\left\{\Vert\psi_{x}\Vert_{H^{m+1}(\mathbb T)}\Vert\varphi_{x}\Vert_{H^1(\mathbb T)}+\Vert\psi_{x}\Vert_{H^2(\mathbb T)}\Vert\varphi_{x}\Vert_{H^m(\mathbb T)}+\Vert\varphi_{x}\Vert_{H^{m+1}(\mathbb T)}\Vert\psi_{x}\Vert_{H^1(\mathbb T)}\right.\\
&\qquad\left. +\Vert\varphi_{x}\Vert_{H^2(\mathbb T)}\Vert\psi_{x}\Vert_{H^m(\mathbb T)}\right\}\\
&\le C\left\{\Vert\psi_{x}\Vert_{H^{m+1}(\mathbb T)}\Vert\varphi_{x}\Vert_{H^2(\mathbb T)}+\Vert\psi_{x}\Vert_{H^2(\mathbb T)}\Vert\varphi_{x}\Vert_{H^{m+1}(\mathbb T)}\right\}\,.
\end{split}
\end{equation}
Then we multiply the square of the $H^m(\mathbb T)-$norm by $e^{-2\gamma t}$, integrate over $\mathbb R_t$ and use H\"older's inequality to find
\begin{equation}\label{stima_2der_1}
\begin{split}
\Vert &\left(-2\mathbb H\left[\Psi_x\phi_x\right]\right)_x\Vert_{L^2_\gamma(\mathbb R;H^m(\mathbb T))}\\
&\le C\left\{\Vert\psi_{x}\Vert_{L^2_\gamma(\mathbb R; H^{m+1}(\mathbb T))}\Vert\varphi_{x}\Vert_{L^\infty(\mathbb R;H^2(\mathbb T))}+\Vert\psi_{x}\Vert_{L^\infty(\mathbb R;H^2(\mathbb T))}\Vert\varphi_{x}\Vert_{L^2_\gamma(\mathbb R;H^{m+1}(\mathbb T))}\right\}\,.
\end{split}
\end{equation}

{\em Estimates for the 2nd and the 3rd terms $\left(\left[\phi\,;\,\mathbb H\right]\Psi_{xx}\right)_x$, $\left(\left[\Psi\,;\,\mathbb H\right]\phi_{xx}\right)_x$:} since the role of $\varphi$ and $\psi$ in the two terms before is exchanged, it is enough to exhibit the estimate for one of them. We compute the estimate for the first one of the two. In view of Lemma \ref{lemma_comm_paolo} and the Sobolev continuity of $\mathbb H$, we find for fixed $t\in\mathbb R$:
\begin{equation}\label{stima_punt_2}
\begin{split}
\Vert &\left(\left[\phi\,;\,\mathbb H\right]\Psi_{xx}\right)_x\Vert_{H^m(\mathbb T)}\le C\Vert \left[\phi\,;\,\mathbb H\right]\Psi_{xx}\Vert_{H^{m+1}(\mathbb T)}\\
&\le C\Vert\partial^{m+1}_x\phi\Vert_{L^2(\mathbb T)}\Vert\Psi_{xx}\Vert_{H^1(\mathbb T)}\le C\Vert\varphi_{x}\Vert_{H^m(\mathbb T)}\Vert\psi_x\Vert_{H^2(\mathbb T)}\,.
\end{split}
\end{equation}
Then we multiply the square of the $H^m(\mathbb T)-$norm by $e^{-2\gamma t}$, integrate over $\mathbb R_t$ and use H\"older's inequality to find
\begin{equation}\label{stima_2der_2}
\Vert\left(\left[\phi\,;\,\mathbb H\right]\Psi_{xx}\right)_x\Vert_{L^2_\gamma(\mathbb R;H^m(\mathbb T))}\le C\Vert\varphi_{x}\Vert_{L^2_\gamma(\mathbb R;H^m(\mathbb T))}\Vert\psi_x\Vert_{L^\infty(\mathbb R;H^2(\mathbb T))}\,.
\end{equation}
The same arguments can be applied to estimate the third term $\left(\left[\Psi\,;\,\mathbb H\right]\phi_{xx}\right)_x$, where the role of $\varphi$ and $\psi$ is changed; thus
\begin{equation}\label{stima_2der_3}
\Vert\left(\left[\Psi\,;\,\mathbb H\right]\phi_{xx}\right)_x\Vert_{L^2_\gamma(\mathbb R;H^m(\mathbb T))}\le C\Vert\psi_{x}\Vert_{L^2_\gamma(\mathbb R;H^m(\mathbb T))}\Vert\varphi_x\Vert_{L^\infty(\mathbb R;H^2(\mathbb T))}\,.
\end{equation}
Collecting \eqref{stima_2der_1}, \eqref{stima_2der_2}, \eqref{stima_2der_3} we have proved the following proposition.
\begin{proposition}\label{prop_2der}
For all $m\geq 1$ there exists a constant $C_m>0$ such that, for all sufficiently smooth $\varphi$, $\psi$ and all $\gamma\geq 1$, the following estimate holds true
\begin{equation}\label{stima_der2}
\begin{split}
\Vert\mathbb L^{\prime\prime}[\varphi_0]&(\varphi, \psi)\Vert_{L^2_{\gamma}(\mathbb R;H^m(\mathbb T))}\\
&\leq C_m\left\{\Vert\varphi_{x}\Vert_{L^2_\gamma(\mathbb R;H^{m+1}(\mathbb T))}\Vert\psi_x\Vert_{L^\infty(\mathbb R;H^2(\mathbb T))} + \Vert\psi_{x}\Vert_{L^2_\gamma(\mathbb R;H^{m+1}(\mathbb T))}\Vert\varphi_x\Vert_{L^\infty(\mathbb R;H^2(\mathbb T))}\right\}\,.
\end{split}
\end{equation}
\end{proposition}
\begin{remark}\label{rem_der2}
Since \eqref{stima_der2} is obtained from the point-wise in $t$ estimates \eqref{stima_punt_1}, \eqref{stima_punt_2}, by repeating the same arguments above where the integration on $\mathbb R_t$ is replaced by integration on $(-\infty, T]$ (for given $T>0$), we get that \eqref{stima_der2} can be restated in the framework of the spaces $L^2_\gamma(-\infty,T;H^m(\mathbb T))$, $L^\infty(-\infty,T; H^m(\mathbb T))$; namely one has that
\begin{equation}\label{stima_der2_T}
\begin{split}
\Vert\mathbb L^{\prime\prime}[\varphi_0]&(\varphi, \psi)\Vert_{L^2_{\gamma}(-\infty,T;H^m(\mathbb T))}\\
&\leq C_m\left\{\Vert\varphi_{x}\Vert_{L^2_\gamma(-\infty,T;H^{m+1}(\mathbb T))}\Vert\psi_x\Vert_{L^\infty(-\infty,T;H^2(\mathbb T))} + \Vert\psi_{x}\Vert_{L^2_\gamma(-\infty,T;H^{m+1}(\mathbb T))}\Vert\varphi_x\Vert_{L^\infty(-\infty,T;H^2(\mathbb T))}\right\}\,,
\end{split}
\end{equation}
where $C_m>0$ is independent of $T$. Moreover, using suitable extension arguments (see Subsection \ref{stima_diff_1}), also estimates \eqref{stima_tame_1}, \eqref{stima_dm-1phitt} can be restated on $(-\infty, T]$ (for given $T>0$).
\end{remark}
\section{The nonlinear problem}\label{nlpb}
In this section we prove the existence of a solution to the nonlinear Cauchy problem
\begin{equation}\label{cp}
\begin{cases}
\varphi_{tt}-\mu\varphi_{xx}=\left(\mathbb H[\phi^2_x]-\left[\phi\,;\,\mathbb H\right]\phi_{xx}\right)_{x}\,, \quad {\rm in}\,\, [0,T]\times\mathbb{T}, \qquad\phi=\mathbb H[\varphi]\,,\\
\varphi(x,0)=\varphi^{(0)}(x) \quad {\rm in}\,\, \mathbb{T},\\
\varphi_t(x,0)=\varphi^{(1)}(x) \quad {\rm in}\,\, \mathbb{T},
\end{cases}
\end{equation}
as stated  in our main Theorem \ref{nonlin_th}.
In the following, we prove the result by applying the Nash--Moser's theorem for the resolution of the nonlinear problem above, see \cite{alinhacgerard07}, \cite{secchi-nash} for a thorough description of the method. In the following for the explanation of the Nash--Moser's method we will refer to the setting and notation of \cite{secchi-nash}.
\newline
Assume that $\varphi^{(0)}\in H^{\nu+1}(\mathbb T)$, $\varphi^{(1)}\in H^{\nu}(\mathbb T)$, for a general integer $\nu\ge 0$, and the sign condition \eqref{sign-cond} is satisfied.

We first consider a suitable lifting $\varphi^{a}:\mathbb R\times \mathbb{T} \rightarrow \mathbb{R}$ of the data $\varphi^{(0)}$, $\varphi^{(1)}$ such that
\begin{equation}\label{reg_phia}
\begin{split}
&\varphi^{a}\in \bigcap \limits_{k=0}^2 W^{k,\infty}(\mathbb R;H^{\nu+1-k}(\mathbb T)),\\
&\varphi^{a}_{|t=0}=\varphi^{(0)}, \quad \partial_t\varphi^{a}_{|t=0}=\varphi^{(1)}\quad {\rm in}\,\,\mathbb{T}\,, \\
&\mu- 2\mathbb{H}[\varphi^{a}]_x\geq 3\delta/4>0 \quad {\rm in}\,\,\,\,\mathbb R\times \mathbb{T}\,,\\
&\sum\limits_{k=0}^2\Vert\partial^k_t\varphi^a\Vert_{L^\infty(\mathbb R; H^{\nu+1-k}(\mathbb T))}\le C\left\{\Vert\varphi^{(0)}\Vert_{H^{\nu+1}(\mathbb T)}+\Vert\varphi^{(1)}\Vert_{H^\nu(\mathbb T)}\right\}\,,
\end{split}
\end{equation}
where $C>0$ is a suitable constant independent of $\varphi^{(0)}$, $\varphi^{(1)}$ and $\nu$.

We define $F^{a}:\mathbb R\times \mathbb{T} \rightarrow \mathbb{R}$ by setting
\begin{equation}\label{Fa}
F^{a}:=
\begin{cases}
-\mathbb L[\varphi^{a}] \quad  {\rm for}\quad t>0\,,\\
0 \quad  {\rm for}\quad t<0,
\end{cases}
\end{equation}
where $\mathbb L$ is the nonlinear operator \eqref{operatore_nonlin}. Using the regularity of $\varphi^{a}$ we get the following result.
\begin{lemma}\label{reg_Fa}
Let the functions $\varphi^a=\varphi^a(t,x)$, $F^a=F^a(t,x)$ be defined as in \eqref{reg_phia}, \eqref{Fa} and $\nu>1$. Then $F^{a}\in L^\infty(\mathbb R; H^{\nu-1}(\mathbb T))$ and we get the following estimate
\begin{equation}\label{stima_Fa}
\Vert F^a\Vert_{L^\infty(\mathbb R; H^{\nu-1}(\mathbb T)}\le\mathcal P(\Vert\varphi^{(0)}\Vert_{H^{\nu+1}(\mathbb T)}, \Vert\varphi^{(1)}\Vert_{H^{\nu}(\mathbb T)})\,,
\end{equation}
where $\mathcal P=\mathcal P(x,y)$ is a quadratic polynomial such that $\mathcal P(0,0)=0$.
\end{lemma}
\begin{proof}
From the regularity of $\varphi^a$ we obtain the following:
\begin{equation*}
\begin{split}
&\varphi^a\in W^{2,\infty}(\mathbb R; H^{\nu-1}(\mathbb T))\,\,\Rightarrow\,\,\partial_{tt}\varphi^a\in L^{\infty}(\mathbb R; H^{\nu-1}(\mathbb T))\\
&\varphi^a\in L^{\infty}(\mathbb R; H^{\nu+1}(\mathbb T))\,\,\Rightarrow\,\,\partial_{xx}\varphi^a\in L^{\infty}(\mathbb R; H^{\nu-1}(\mathbb T))\,;
\end{split}
\end{equation*}
moreover we compute, for $\phi^a:=\mathbb H\left[\varphi^a\right]$,
\begin{equation}\label{parte_nonlin}
\begin{split}
&\left(\mathbb H\left[\left(\phi^a_x\right)^2\right]-\left[\phi^a\,;\,\mathbb H\right]\phi^a_{xx}\right)_x=2\mathbb H\left[\phi^a_x\,\phi^a_{xx}\right]-\left[\phi^a_x\,;\,\mathbb H\right]\phi^a_{xx}-\left[\phi^a\,;\,\mathbb H\right]\phi^a_{xxx}\\
&\qquad=3\mathbb H\left[\phi^a_x\,\phi^a_{xx}\right]-\phi^a_x\,\mathbb H\left[\phi^a_{xx}\right]-\left[\phi^a\,;\,\mathbb H\right]\phi^a_{xxx}\,.
\end{split}
\end{equation}
Hence $\left(\mathbb H\left[\left(\phi^a_x\right)^2\right]-\left[\phi^a\,;\,\mathbb H\right]\phi^a_{xx}\right)_x\in L^\infty(\mathbb R; H^{\nu-1}(\mathbb T))$ in view of the continuity of $\mathbb H$ in Sobolev spaces, the inclusion $H^\nu(\mathbb T)\cdot H^{\nu-1}(\mathbb T) \subset H^{\nu-1}(\mathbb T)$, that holds since $\nu>1$ (cf. Lemma \ref{lemma_prod_alg}), and Lemma \ref{lemma_comm_ale_2} (for $m=\nu-1$, $p=2$, $v=\phi^a$ and $f=\phi^a_x$). The above calculations yield that $\mathbb L(\varphi^a)$, hence $F^a$, belong to $L^\infty(\mathbb R; H^{\nu-1}(\mathbb T))$.

As regards to the estimate \eqref{stima_Fa}, from \eqref{operatore_nonlin}, \eqref{Fa} and \eqref{parte_nonlin} and using the Sobolev continuity of the Hilbert transform we first get
\begin{equation}\label{stima_Fa_resto}
\begin{split}
\Vert F^a &\Vert_{L^\infty(\mathbb R; H^{\nu-1}(\mathbb T))}\le \Vert\mathbb L(\varphi^a)\Vert_{L^\infty(\mathbb R; H^{\nu-1}(\mathbb T))}\le\Vert\varphi^a_{tt}\Vert_{L^\infty(\mathbb R; H^{\nu-1}(\mathbb T))}+\vert\mu\vert\Vert\varphi^a_{xx}\Vert_{L^\infty(\mathbb R; H^{\nu-1}(\mathbb T))}\\
&+3\Vert\mathbb H\left[\phi^a_x\,\phi^a_{xx}\right]\Vert_{L^\infty(\mathbb R; H^{\nu-1}(\mathbb T))}+\Vert\phi^a_x\,\mathbb H\left[\phi^a_{xx}\right]\Vert_{L^\infty(\mathbb R; H^{\nu-1}(\mathbb T))}+\Vert\left[\phi^a\,;\,\mathbb H\right]\phi^a_{xxx}\Vert_{L^\infty(\mathbb R; H^{\nu-1}(\mathbb T))}\\
&\le\Vert\varphi^a_{tt}\Vert_{L^\infty(\mathbb R; H^{\nu-1}(\mathbb T))}+\vert\mu\vert\Vert\varphi^a\Vert_{L^\infty(\mathbb R; H^{\nu+1}(\mathbb T))}+3\Vert\phi^a_x\,\phi^a_{xx}\Vert_{L^\infty(\mathbb R; H^{\nu-1}(\mathbb T))}\\
&+\Vert\phi^a_x\,\mathbb \varphi^a_{xx}\Vert_{L^\infty(\mathbb R; H^{\nu-1}(\mathbb T))}+\Vert\left[\phi^a\,;\,\mathbb H\right]\phi^a_{xxx}\Vert_{L^\infty(\mathbb R; H^{\nu-1}(\mathbb T))}\,.
\end{split}
\end{equation}
For fixed $t\in\mathbb R$ from Lemma \ref{lemma_prod_alg} (with $s=\nu$ and $m=\nu-1$) and using once again the Sobolev continuity of $\mathbb H$, we get
\begin{equation*}
\Vert\phi^a_x(t)\,\phi^a_{xx}(t)\Vert_{H^{\nu-1}(\mathbb T)}\le C\Vert\phi^a_x(t)\Vert_{H^{\nu}(\mathbb T)}\Vert\phi^a_{xx}(t)\Vert_{H^{\nu-1}(\mathbb T)}\le C\Vert\phi^a(t)\Vert^2_{H^{\nu+1}(\mathbb T)}\le C \Vert\varphi^a(t)\Vert^2_{H^{\nu+1}(\mathbb T)}\,.
\end{equation*}
Then taking the supremum over $\mathbb R_t$ we obtain
\begin{equation}\label{stima_Fa_3}
\Vert\phi^a_x\,\phi^a_{xx}\Vert_{L^\infty(\mathbb R;H^{\nu-1}(\mathbb T))}\le C \Vert\varphi^a\Vert^2_{L^\infty(\mathbb R;H^{\nu+1}(\mathbb T))}\,.
\end{equation}
The same arguments above can be used to estimate the term $\Vert\phi^a_x\,\mathbb \varphi^a_{xx}\Vert_{L^\infty(\mathbb R; H^{\nu-1}(\mathbb T))}$ by
\begin{equation}\label{stima_Fa_4}
\Vert\phi^a_x\,\varphi^a_{xx}\Vert_{L^\infty(\mathbb R;H^{\nu-1}(\mathbb T))}\le C \Vert\varphi^a\Vert^2_{L^\infty(\mathbb R;H^{\nu+1}(\mathbb T))}\,.
\end{equation}
As for the subsequent commutator term $\Vert\left[\phi^a\,;\,\mathbb H\right]\phi^a_{xxx}\Vert_{L^\infty(\mathbb R; H^{\nu-1}(\mathbb T))}$, for fixed $t\in\mathbb R$ we apply Lemma \ref{lemma_comm_ale_2} (for $m=\nu-1$, $p=2$, $v=\phi^a$ and $f=\phi^a_x$), the imbedding $H^\nu(\mathbb T)\hookrightarrow H^1(\mathbb T)$ (as $\nu>1$) and the Sobolev continuity of $\mathbb H$ to get
\begin{equation*}
\begin{split}
\Vert\left[\phi^a(t)\,;\,\mathbb H\right]&\phi^a_{xxx}(t)\Vert_{H^{\nu-1}(\mathbb T)}\le C\Vert\partial^{\nu+1}_x\phi^a(t)\Vert_{L^2(\mathbb T)}\Vert\phi^a_x(t)\Vert_{H^1(\mathbb T)}\\
&\le C\Vert\phi^a(t)\Vert_{H^{\nu+1}(\mathbb T)}\Vert\phi^a_x(t)\Vert_{H^\nu(\mathbb T)}\le C\Vert\varphi^a(t)\Vert^2_{H^{\nu+1}(\mathbb T)}\,.
\end{split}
\end{equation*}
Again taking the supremum over $\mathbb R_t$ gives
\begin{equation}\label{stima_Fa_5}
\Vert\left[\phi^a\,;\,\mathbb H\right]\phi^a_{xxx}\Vert_{L^\infty(\mathbb R;H^{\nu-1}(\mathbb T))}\le C\Vert\varphi^a\Vert^2_{L^\infty(\mathbb R;H^{\nu+1}(\mathbb T))}\,.
\end{equation}
Gathering estimates \eqref{stima_Fa_resto}-\eqref{stima_Fa_5} we obtain
\begin{equation}\label{stima_Fa_6}
\begin{split}
\Vert F^a\Vert_{L^\infty(\mathbb R; H^{\nu-1}(\mathbb T))}\le\Vert\varphi^a_{tt}\Vert_{L^\infty(\mathbb R; H^{\nu-1}(\mathbb T))}+\vert\mu\vert\Vert\varphi^a\Vert_{L^\infty(\mathbb R; H^{\nu+1}(\mathbb T))}+C\Vert\varphi^a\Vert^2_{L^\infty(\mathbb R;H^{\nu+1}(\mathbb T))}\,.
\end{split}
\end{equation}
Combining \eqref{stima_Fa_6} with $\eqref{reg_phia}_4$ we finally get
\begin{equation*}
\begin{split}
\Vert  & F^a\Vert_{L^\infty(\mathbb R; H^{\nu-1}(\mathbb T))}\le (1+\vert\mu\vert)\left\{\Vert\varphi^{(0)}\Vert_{H^{\nu+1}(\mathbb T)}+\Vert\varphi^{(1)}\Vert_{H^\nu(\mathbb T)}\right\}\\
&+C\left\{\Vert\varphi^{(0)}\Vert_{H^{\nu+1}(\mathbb T)}+\Vert\varphi^{(1)}\Vert_{H^\nu(\mathbb T)}\right\}^2=:\mathcal P(\Vert\varphi^{(0)}\Vert_{H^{\nu+1}(\mathbb T)}, \Vert\varphi^{(1)}\Vert_{H^{\nu}(\mathbb T)})\,.
\end{split}
\end{equation*}
\end{proof}
\begin{remark}\label{rmk:1}
It is worth to remark that in view of the definition of $F^a$, for $\nu>1$ we have the following:
\begin{itemize}
\item[(i)] for $t<0$, $F^a\equiv 0$ then $F^a\in L^2(-\infty, 0; H^{\nu-1}(\mathbb T))$;
\item[(ii)] for every $T>0$, $L^\infty(0,T; \mathcal X)\subset L^2(0,T; \mathcal X)$ (where $\mathcal X$ is any Banach space), hence $F^a\in L^\infty(0,T; H^{\nu-1}(\mathbb T))\subset L^2(0,T; H^{\nu-1}(\mathbb T))$.
\end{itemize}
From (i), (ii) above easily follows that $F^a\in  L^2(-\infty,T; H^{\nu-1}(\mathbb T))$. Moreover for $\gamma\ge 1$ fixed we get also that $F^a\in  L^2_\gamma(-\infty,T; H^{\nu-1}(\mathbb T))$; indeed from \eqref{stima_Fa} we compute
\begin{equation*}
\begin{split}
\Vert & F^a\Vert^2_{L^2_\gamma(-\infty,T; H^{\nu-1}(\mathbb T))}=\int_{-\infty}^T e^{-2\gamma t}\Vert F^a(t)\Vert^2_{H^{\nu-1}(\mathbb T)}\,dt=\int_{0}^T e^{-2\gamma t}\Vert F^a(t)\Vert^2_{H^{\nu-1}(\mathbb T)}\,dt\\
&\le \Vert F^a\Vert^2_{L^{\infty}(\mathbb R; H^{\nu-1}(\mathbb T))}\int_0^T e^{-2\gamma t}\,dt=\frac{1}{2\gamma}\left(1-e^{-2\gamma T}\right)\Vert F^a\Vert^2_{L^{\infty}(\mathbb R; H^{\nu-1}(\mathbb T))}\\
&\le \frac{1}{2\gamma}\left(1-e^{-2\gamma T}\right)\mathcal P^2(\Vert\varphi^{(0)}\Vert_{H^{\nu+1}(\mathbb T)}, \Vert\varphi^{(1)}\Vert_{H^{\nu}(\mathbb T)})\,.
\end{split}
\end{equation*}
The estimate above shows that for fixed $\gamma\ge 1$ we can make the norm of $F^a$ in $L^2_\gamma(-\infty,T; H^{\nu-1}(\mathbb T))$ as small as we want by choosing  a suitable value of $T$ depending on the norms of the initial data $\varphi^{(0)}$, $\varphi^{(1)}$ respectively in $H^{\nu+1}(\mathbb T)$ and $H^{\nu}(\mathbb T)$.
\end{remark}

\subsection{An equivalent formulation of the problem \eqref{cp}}
In order to solve problem \eqref{cp} by the Nash-Moser's theorem, it is convenient to recast this problem in an equivalent form. We first look for a solution of the nonlinear problem
\begin{eqnarray}
\mathbb L[\varphi]=0\,, \quad {\rm in}\,\, [0,T]\times\mathbb{T}\,,\label{cp1}\\
\varphi_{\vert\,t=0}=\varphi^{(0)}\,\quad {\rm in}\,\, \mathbb{T},\label{cp2}\\
\partial_t\varphi_{\vert\,t=0}=\varphi^{(1)} \quad {\rm in}\,\, \mathbb{T},\label{cp3}
\end{eqnarray}
(see \eqref{operatore_nonlin} for the definition of $\mathbb L$) in the form of a perturbation of the function $\varphi^a(t,x)$ defined in the previous section, namely
\begin{equation}\label{forma_sol}
\varphi=\varphi^a+\varphi^\prime\,,
\end{equation}
where, according to $\eqref{reg_phia}_2$, we must have
\begin{equation}\label{ic}
\varphi^\prime_{\vert\,t=0}=0\,,\quad \partial_t\varphi^\prime_{\vert\,t=0}=0\,,\quad\mbox{in}\,\,\mathbb T\,.
\end{equation}
Replacing \eqref{forma_sol} into \eqref{cp1} we find
\begin{equation}\label{equiv}
\mathbb L[\varphi^a+\varphi^\prime]=0\quad\Leftrightarrow\quad\mathcal L[\varphi^\prime]=-\mathbb L [\varphi^a]\,,\quad\mbox{in}\,\,[0,T]\times\mathbb T\,,
\end{equation}
where we have set
\begin{equation}\label{oprt_L}
\mathcal L[\varphi^\prime]:=\mathbb L[\varphi^a+\varphi^\prime]-\mathbb L [\varphi^a]\,.
\end{equation}
In agreement with the functional setting introduced in Section \ref{not}, it is convenient to ``extend'' the problem \eqref{equiv} also for negative time; then, in view of \eqref{Fa}, we are led to solve the nonlinear problem
\begin{eqnarray}
\mathcal L[\varphi^\prime]=F^a\,,\quad\mbox{in}\,\,(-\infty,T]\times\mathbb T\,,\label{cp1_equiv}\\
\varphi^\prime=0\,,\quad\mbox{for}\,\,t<0\,.\label{cp2_equiv}
\end{eqnarray}

\subsection{The functional setting}\label{sec_funct_sett}
In this section we fix the functional setting where the Nash-Moser's theorem will be applied.
For $m\ge 1$, we define $X_m$ to be the space of measurable functions $\varphi:(-\infty,T]\times\mathbb T\rightarrow\mathbb R$ such that
\begin{itemize}
\item[i)] $\varphi$ has zero spatial mean, i.e. $\widehat{\varphi}(0)=\displaystyle\int_{\mathbb T}\varphi\,dx=0$;
\item[ii)] $\varphi_{\vert\{t<0\}}=0$;
\item[iii)] the function $\varphi$ enjoys the following regularity assumptions
\begin{equation}\label{Xm}
\varphi\in L^2_\gamma(-\infty, T; H^{m+2}(\mathbb T))\cap H^1_\gamma(-\infty, T; H^{m+1}(\mathbb T))\cap H^2_\gamma(-\infty, T; H^m(\mathbb T))\,,
\end{equation}
where $\gamma\ge 1$ will be fixed in a suitable way later on.
\end{itemize}
In view of the equality \eqref{ident_norme} the space $X_m$ can be provided with the norm
\begin{equation}\label{normaXm}
\Vert\varphi\Vert_{X_m}^2:=\Vert\varphi_{x}\Vert^2_{L^2_\gamma(-\infty, T; H^{m+1}(\mathbb T))}+\Vert\varphi_t\Vert^2_{L^2_\gamma(-\infty, T; H^{m+1}(\mathbb T))}+\Vert\varphi_{tt}\Vert^2_{L^2_\gamma(-\infty, T; H^{m}(\mathbb T))}\,.
\end{equation}
We also define $Y_m$ to be
\begin{equation*}
Y_m:=\{F:  (-\infty,T]\times\mathbb T\rightarrow\mathbb R\,:\,\,F_{\vert\,\{t<0\}}=0\,,\,\,F\in L^2_\gamma(-\infty, T; H^m(\mathbb T))\}\,,
\end{equation*}
with the natural norm
\begin{equation*}
\Vert F\Vert_{Y_m}:=\Vert F\Vert_{L^2_\gamma(-\infty, T; H^m(\mathbb T))}\,.
\end{equation*}
\begin{remark}\label{rmk:2}
From Remark \ref{rmk:1}, for $\nu>1$ we get that $F^a\in Y_{\nu-1}$ with $Y_{\nu-1}-$norm as small as we want for $T>0$ sufficiently small, depending on the initial data $\varphi^{(0)}$, $\varphi^{(1)}$.
\end{remark}

\subsection{Nash-Moser's theorem}\label{NM}
Now we are going to apply the Nash-Moser theorem to solve the problem \eqref{cp1_equiv}, \eqref{cp2_equiv} in the functional framework introduced in the previous section. Our goal is to find $\varphi^\prime\in X_m$ for a suitable $m\geq 0$ satisfying \eqref{cp1_equiv}--\eqref{cp2_equiv}.
\begin{remark}\label{rmk:3}
 The regularity assumptions defining the spaces $X_m$ (see \eqref{Xm}), imply that the solution to \eqref{cp1_equiv}, \eqref{cp2_equiv} belonging to  $X_m$ for some $m\geq 0$ is such that
\begin{equation*}
\begin{split}
\varphi^\prime \in H^1_\gamma(-\infty,T;H^{m+1}(\mathbb T))\hookrightarrow C((-\infty,T]; H^{m+1}(\mathbb T))\\
\partial_t\varphi^\prime \in H^1_\gamma(-\infty,T;H^{m}(\mathbb T))\hookrightarrow C((-\infty,T]; H^{m}(\mathbb T)).
\end{split}
\end{equation*}
Hence, the condition \eqref{cp2_equiv}  together with the continuity in time of $\varphi^\prime, \partial_t\varphi^\prime$ gives that \eqref{ic} are satisfied, i.e. $\varphi^\prime_{|\, t=0}=0, \partial_t\varphi^\prime_{|\, t=0}=0$.
\end{remark}

Here we closely follow the presentation of the method in \cite{secchi-nash}. According to \cite{secchi-nash} we will use the following notations
\begin{equation*}
X_\infty:=\bigcap\limits_{m\ge 0}X_m\,,\quad Y_{\infty}:=\bigcap\limits_{m\geq 0}Y_m\,,\quad H^\infty(\mathbb T):=\bigcap\limits_{m\geq 0}H^m(\mathbb T)\,,
\end{equation*}
where $X_m$, $Y_m$ are the functional spaces introduced in Section \ref{sec_funct_sett}, and $\{X_m\}_{m\ge 0}$, $\{Y_m\}_{m\ge 0}$ are decreasing families of Banach spaces satisfying the \emph{smoothing hypothesis}, see \cite[Definition 2.3]{secchi-nash}.

We first need to collect a number of properties of the nonlinear operator $\mathcal L$ in \eqref{oprt_L}, as well as its first and second derivatives at a given point $\varphi_0\in X_\infty$. In order to do so, here we assume that the initial data $\varphi^{(0)}$, $\varphi^{(1)}$ in \eqref{cp2}, \eqref{cp3} belong to $H^\infty(\mathbb T)$, so that $\varphi^a$ has $H^\infty$-regularity in $x$ and $F^a\in Y_\infty$. Under these assumptions, we have that for every $m\ge 0$
\begin{equation*}
\mathcal L:X_m\rightarrow Y_m\,.
\end{equation*}
The operator $\mathcal L$ must satisfy all the assumptions in \cite[Theorem 2.4]{secchi-nash}, see for the reader's convenience \ref{sec_N-M}, Theorem \ref{N-M_paolo}; more precisely we have to check that the first order differential $d\mathcal L[\varphi_0]$ and the second order differential $d^2\mathcal L[\varphi_0]$ obey suitable estimates (see the Assumptions 2.1, 2.2 in \cite{secchi-nash}), as long as $\varphi_0$ belongs to $U\cap X_{\infty}$, being $U$ a suitable bounded open neighborhood of $0$ in $X_{m_0}$ for some $m_0\ge 0$.

Let us observe that, for every $\varphi_0\in X_\infty$, from \eqref{oprt_L} one computes
\begin{equation}\label{diffL}
\begin{split}
d\mathcal L[\varphi_0](\varphi^\prime)=\mathbb L^\prime[\varphi^a+\varphi_0]\varphi^\prime\,,\\
d^2\mathcal L[\varphi_0](\varphi^\prime, \psi^\prime)=\mathbb L^{\prime\prime}[\varphi^a+\varphi_0](\varphi^\prime, \psi^\prime)\,.
\end{split}
\end{equation}

\subsubsection{Estimate for the first order differential of $\mathcal L$}\label{stima_diff_1}
In view of \eqref{diffL}, let us consider the equation
\begin{equation}\label{eq_diff}
\mathbb L^\prime[\varphi^a+\varphi_0]\varphi^\prime=g, \quad \mbox{in} \,\,(-\infty,T)\times \mathbb{T}.
\end{equation}
In order to have that the hypothesis of Nash-Moser's theorem is satisfied (see \cite[Assumption 2.2]{secchi-nash}) we need equation \eqref{eq_diff} to admit a unique solution in $X_\infty$ as long as $g\in Y_\infty$ (i.e. $d\mathcal L[\varphi_0](\varphi^\prime): X_{\infty} \rightarrow Y_{\infty}$ admits a right inverse operator). This comes as a consequence of Theorem \ref{teorema_2}.

In order to use the tame estimate \eqref{stima_tame_1} in Proposition \ref{prop_stima_tame}, we have to consider the equation \eqref{eq_diff} with the time extended to the whole real line (i.e. on $\mathbb{R}\times \mathbb{T}$). To this end, let us consider suitable extensions $\tilde g$ and $\tilde\varphi_0$ respectively of the source term $g$ and  the basic state $\varphi_0$ in such a way that
\begin{equation}\label{ext_cont}
\begin{split}
&\tilde g\in L^2_\gamma(\mathbb R;H^\infty(\mathbb T)), \quad \tilde\varphi_0\in L^\infty(\mathbb R; H^\infty(\mathbb T))\,,\\
&\Vert\tilde{g}\Vert_{L^2_\gamma(\mathbb R;H^m(\mathbb T))}\le C_{m} \Vert g\Vert_{L^2_\gamma(-\infty, T;H^m(\mathbb T))}\,,\\
&\Vert\tilde{\varphi}_0\Vert_{L^\infty(\mathbb R;H^m(\mathbb T))}\le C_{m} \Vert \varphi_0\Vert_{L^\infty(-\infty, T;H^m(\mathbb T))}\,,\\
&\Vert\tilde{\varphi}_{0,x}\Vert_{L^\infty(\mathbb R;H^m(\mathbb T))}\le C_{m} \Vert \varphi_{0,x}\Vert_{L^\infty(-\infty, T;H^m(\mathbb T))}\,,\quad\forall\,m\ge 0\,,\,\,\forall\,\gamma\ge 1\,,
\end{split}
\end{equation}
where for every $m$, $C_{m}$ is some positive constant depending on $m$ and independent of $\gamma$. The constant $C_{m}$ may be chosen independently of $T$.

Let us assume that the linearized time-extended equation
\begin{equation}\label{ext_eq}
\mathbb L^\prime[\varphi^a+\tilde\varphi_0]\tilde{\varphi}^\prime=\tilde g, \quad \mbox{in} \,\,\mathbb R\times \mathbb{T}
\end{equation}
has a unique solution: we want  to derive an energy estimate for it. Let $\tilde{\varphi}^\prime$ be such a solution. From Proposition \ref{prop_stima_tame}, by estimate \eqref{stima_tame_1} written for $m+1$ instead of $m$, we get that $\tilde\varphi^\prime$ of \eqref{ext_eq} satisfies the estimate
\begin{equation}\label{stima_diff}
\begin{split}
\gamma & \left\{\Vert\tilde{\varphi}^\prime_{t}\Vert^2_{L^2_\gamma(\mathbb R; H^{m+1}(\mathbb T))}+\Vert\tilde{\varphi}^\prime_{x}\Vert^2_{L^2_\gamma(\mathbb R;H^{m+1}(\mathbb T))}\right\}\\
&\le \frac{C}{\gamma}\left\{\Vert\varphi^a_x+\tilde{\varphi}_{0,x}\Vert^2_{L^\infty(\mathbb R;H^{m+3}(\mathbb T))}\Vert \tilde g\Vert^2_{L^2_\gamma(\mathbb R; H^2(\mathbb T))} +\Vert \tilde g\Vert^2_{L^2_\gamma(\mathbb R; H^{m+1}(\mathbb T))}\right\}\,,
\end{split}
\end{equation}
for all $\gamma\geq \gamma_1$ and $C$ and $\gamma_1$ depend increasingly and boundedly on $m,\delta$ and $\Vert\varphi^a+\tilde{\varphi}_0\Vert_{L^\infty(\mathbb R;H^5(\mathbb T))}$ and $\Vert\varphi^a_t+\tilde{\varphi}_{0,t}\Vert_{L^\infty(\mathbb R;H^2(\mathbb T))}$, provided that the sign condition
\begin{equation}\label{sign_phia_phi0}
\mu-2\mathbb{H}[\varphi^a+\tilde{\varphi}_0]_x\geq \delta/2\,,\quad\mbox{in}\,\,\mathbb R\times\mathbb T
\end{equation}
holds (see Proposition \ref{prop_stima_tame}). By construction, the function
\begin{equation*}
\varphi^\prime:=\tilde\varphi^{\prime}_{\vert\,\,(-\infty, T]}
\end{equation*}
provides a solution to the equation \eqref{eq_diff} and, from \eqref{ext_cont} and \eqref{stima_diff} we get
\begin{equation}\label{stima_diff_T}
\begin{split}
\gamma & \left\{\Vert\varphi^\prime_{t}\Vert^2_{L^2_\gamma(-\infty, T; H^{m+1}(\mathbb T))}+\Vert\varphi^\prime_{x}\Vert^2_{L^2_\gamma(-\infty,T;H^{m+1}(\mathbb T))}\right\}\\
&\le \gamma\left\{\Vert\tilde{\varphi}^\prime_{t}\Vert^2_{L^2_\gamma(\mathbb R; H^{m+1}(\mathbb T))}+\Vert\tilde{\varphi}^\prime_{x}\Vert^2_{L^2_\gamma(\mathbb R;H^{m+1}(\mathbb T))}\right\}\\
&\le \frac{C}{\gamma}\left\{\Vert\varphi^a_x+\tilde{\varphi}_{0,x}\Vert^2_{L^\infty(\mathbb R;H^{m+3}(\mathbb T))}\Vert \tilde g\Vert^2_{L^2_\gamma(\mathbb R; H^2(\mathbb T))} +\Vert \tilde g\Vert^2_{L^2_\gamma(\mathbb R; H^{m+1}(\mathbb T))}\right\}\\
&\le \frac{C}{\gamma}\left\{\left(\Vert\varphi^a_x\Vert^2_{L^\infty(\mathbb R;H^{m+3}(\mathbb T))}+\Vert\varphi_{0,x}\Vert^2_{L^\infty(-\infty, T;H^{m+3}(\mathbb T))}\right)\Vert g\Vert^2_{L^2_\gamma(-\infty,T; H^2(\mathbb T))} +\Vert g\Vert^2_{L^2_\gamma(-\infty,T; H^{m+1}(\mathbb T))}\right\}\,,
\end{split}
\end{equation}
for all $\gamma\ge\gamma_1$, being $\gamma_1$, $C>0$ defined as in \eqref{stima_diff} and again under the condition \eqref{sign_phia_phi0}.

Concerning the condition \eqref{sign_phia_phi0}, in view of $\eqref{reg_phia}_3$, \eqref{ext_cont} and the Sobolev imbeddings $H^1(\mathbb T)\hookrightarrow L^\infty(\mathbb T)$, $H^1_\gamma(-\infty, T; \mathcal X)\hookrightarrow L^\infty(-\infty,T; \mathcal X)$ (cf. Lemma \ref{sobolev-imb}) we get
\begin{equation}\label{sgn_restr}
\begin{split}
\mu &-2\mathbb{H}[\varphi^a+\tilde{\varphi}_0]_x\geq 3\delta/4-2\mathbb H[\tilde{\varphi}_0]_x\ge 3\delta/4-2\Vert\mathbb H[\tilde{\varphi}_0]_x\Vert_{L^\infty(\mathbb R\times\mathbb T)}\\
&\ge 3\delta/4-C\Vert\mathbb H[\tilde{\varphi}_0]_x\Vert_{L^\infty(\mathbb R;H^1(\mathbb T))}\ge 3\delta/4-C\Vert\tilde{\varphi}_{0,x}\Vert_{L^\infty(\mathbb R;H^1(\mathbb T))}\ge 3\delta/4-C\Vert\varphi_{0,x}\Vert_{L^\infty(-\infty,T;H^1(\mathbb T))}\\
&\ge 3\delta/4-C\Vert\varphi_{0,x}\Vert_{H^1_\gamma(-\infty,T;H^1(\mathbb T))}\ge 3\delta/4-C\Vert\varphi_0\Vert_{X_1}\,.
\end{split}
\end{equation}
From \eqref{sgn_restr} we deduce that condition \eqref{sign_phia_phi0}, and consequently the estimate \eqref{stima_diff_T}, are satisfied provided that
\begin{equation*}
\Vert\varphi_0\Vert_{X_1}\le\delta/4C\,.
\end{equation*}
On the other hand, from \eqref{ext_cont}, the Sobolev imbedding and $\eqref{reg_phia}_4$ we find that
\begin{equation}\label{CX4}
\begin{split}
\Vert &\varphi^a+\tilde{\varphi}_0\Vert_{L^\infty(\mathbb R;H^5(\mathbb T))}\le\Vert\varphi^a\Vert_{L^\infty(\mathbb R;H^5(\mathbb T))}+\Vert\tilde{\varphi}_0\Vert_{L^\infty(\mathbb R;H^5(\mathbb T))}\\
&\le \Vert\varphi^a\Vert_{L^\infty(\mathbb R;H^5(\mathbb T))}+C\Vert\varphi_0\Vert_{L^\infty(-\infty,T;H^5(\mathbb T))}\le \Vert\varphi^a\Vert_{L^\infty(\mathbb R;H^5(\mathbb T))}+C\Vert\varphi_0\Vert_{H^1_\gamma(-\infty,T;H^5(\mathbb T))}\\
&\le \Vert\varphi^a\Vert_{L^\infty(\mathbb R;H^5(\mathbb T))}+C\Vert\varphi_0\Vert_{X_4}\le C\left\{\Vert\varphi^{(0)}\Vert_{H^5(\mathbb T)}+\Vert\varphi^{(1)}\Vert_{H^4(\mathbb T)}+\Vert\varphi_0\Vert_{X_4}\right\} \,,
\end{split}
\end{equation}
and similarly
\begin{equation}\label{CX41}
\begin{split}
\Vert\varphi^a_t &+\tilde{\varphi}_{0\,,t}\Vert_{L^\infty(\mathbb R;H^2(\mathbb T))}\le \Vert\varphi^a_t\Vert_{L^\infty(\mathbb R;H^2(\mathbb T))}+ C \Vert\varphi_0\Vert_{X_2}\\
&\le C\left\{\Vert\varphi^{(0)}\Vert_{H^3(\mathbb T)}+\Vert\varphi^{(1)}\Vert_{H^2(\mathbb T)}+\Vert\varphi_0\Vert_{X_2} \right\} \,.
\end{split}
\end{equation}
Finally, we observe that the inequalities
\begin{equation}\label{imbeddingX}
\Vert\varphi\Vert_{X_1}\le\Vert\varphi\Vert_{X_2}\le\Vert\varphi\Vert_{X_4}\,,\quad\forall\varphi\in X_4\,,
\end{equation}
follow at once from the definition \eqref{normaXm} of the norm in $X_m$.

Let us define the bounded neighborhood $\mathcal U$ of $0$ in $X_4$ be setting
\begin{equation}\label{neighbrhd}
\mathcal U:=\left\{\varphi\in X_4\,:\,\,\Vert\varphi\Vert_{X_4}\le\delta/4C\right\}\,.
\end{equation}
Because of \eqref{CX4}-\eqref{imbeddingX}, one has
\begin{equation}\label{CX.1}
\begin{split}
\Vert &\varphi^a+\tilde{\varphi}_0\Vert_{L^\infty(\mathbb R;H^5(\mathbb T))}+\Vert\varphi^a_t +\tilde{\varphi}_{0\,,t}\Vert_{L^\infty(\mathbb R;H^2(\mathbb T))}\\
&\le C\left\{\Vert\varphi^{(0)}\Vert_{H^5(\mathbb T)}+\Vert\varphi^{(1)}\Vert_{H^4(\mathbb T)}\right\}+\delta/C\,,
\end{split}
\end{equation}
and the sign condition \eqref{sign_phia_phi0} holds true, as long as $\varphi_0\in\mathcal U\cap X_\infty$. According to the previous observations, let us set
\begin{equation}\label{Cgamma}
\begin{split}
\widehat{\gamma}_1=\widehat{\gamma}_1(m,\delta,\varphi^{(0)},\varphi^{(1)}):=\sup\limits_{\varphi_0\in\mathcal U\cap X_\infty} \gamma_1(m,\delta, \Vert\varphi^a+\tilde{\varphi}_0\Vert_{L^\infty(\mathbb R;H^5(\mathbb T))}, \Vert\varphi^a_t+\tilde{\varphi}_{0,t}\Vert_{L^\infty(\mathbb R;H^2(\mathbb T))})\,,\\
\widehat{C}_1=\widehat{C}_1(m,\delta,\varphi^{(0)},\varphi^{(1)}):=\sup\limits_{\varphi_0\in\mathcal U\cap X_\infty} C(m,\delta, \Vert\varphi^a+\tilde{\varphi}_0\Vert_{L^\infty(\mathbb R;H^5(\mathbb T))}, \Vert\varphi^a_t+\tilde{\varphi}_{0,t}\Vert_{L^\infty(\mathbb R;H^2(\mathbb T))})\,,
\end{split}
\end{equation}
where $\gamma_1$ and $C$ in the right-hand sides are the constants involved in \eqref{stima_diff_T}. From \eqref{stima_diff_T} and the above definitions we derive that
\begin{equation}\label{stima_diff_T1}
\begin{split}
\gamma & \left\{\Vert\varphi^\prime_{t}\Vert^2_{L^2_\gamma(-\infty, T; H^{m+1}(\mathbb T))}+\Vert\varphi^\prime_{x}\Vert^2_{L^2_\gamma(-\infty,T;H^{m+1}(\mathbb T))}\right\}\\
&\le \frac{\widehat{C}_1}{\gamma}\left\{\left(\Vert\varphi^a_x\Vert^2_{L^\infty(\mathbb R;H^{m+3}(\mathbb T))}+\Vert\varphi_{0,x}\Vert^2_{L^\infty(-\infty, T;H^{m+3}(\mathbb T))}\right)\Vert g\Vert^2_{L^2_\gamma(-\infty,T; H^2(\mathbb T))} +\Vert g\Vert^2_{L^2_\gamma(-\infty,T; H^{m+1}(\mathbb T))}\right\}\,,
\end{split}
\end{equation}
holds true for all $\varphi^{\prime}\in X_\infty$, $\varphi_0\in\mathcal U\cap X_{\infty}$ and $\gamma\ge\widehat{\gamma}_1$. Note that $\widehat{\gamma}_1$ and $\widehat{C}_1$ in \eqref{stima_diff_T1} now depend only on $m$, $\delta$ and the initial data $\varphi^{(0)}$, $\varphi^{(1)}$.

From now on, for every $m\ge 1$ let $\gamma$ be fixed such that $\gamma\ge\widehat{\gamma}_1$, being $\widehat{\gamma}_1$ defined as in \eqref{Cgamma}. For such $\gamma$, we consider the estimate \eqref{stima_dm-1phitt} (with $m+1$ instead of $m$), written for $\tilde{\varphi}^\prime$; we use \eqref{ext_cont} and Sobolev imbedding $H^1_\gamma(-\infty,T;\mathcal X)\hookrightarrow L^\infty(-\infty,T; \mathcal X)$ (see Lemma \ref{sobolev-imb}), to get
\begin{equation}\label{stima_dm-1tildephitt}
\begin{split}
&\Vert\varphi^\prime_{tt}\Vert_{L^2_\gamma(-\infty,T;H^{m}(\mathbb T))}\le\\
&\Vert\tilde{\varphi}^\prime_{tt}\Vert_{L^2_\gamma(\mathbb R;H^{m}(\mathbb T))}\le C_1\left\{\Vert\tilde{\varphi}^\prime_{x}\Vert_{L^2_\gamma(\mathbb R;H^{m+1}(\mathbb T))}+\Vert\varphi^a+\tilde{\varphi}_0\Vert_{L^\infty(\mathbb R;H^3(\mathbb T))}\Vert\tilde{\varphi}^\prime_{x}\Vert_{L^2_\gamma(\mathbb R;H^{m+1}(\mathbb T))}\right.\\
&\qquad\qquad\qquad\quad \left.+\Vert\varphi^a_x+\tilde{\varphi}_{0,x}\Vert_{L^\infty(\mathbb R;H^{m+1}(\mathbb T))}\Vert\tilde{\varphi}^\prime_{x}\Vert_{L^2_\gamma(\mathbb R;H^2(\mathbb T))}+\Vert \tilde g\Vert_{L^2_\gamma(\mathbb R;H^{m}(\mathbb T))}\right\}\\
&\le C_1\left\{\Vert{\widetilde{\varphi}}^\prime_{x}\Vert_{L^2_\gamma(\mathbb R;H^{m+1}(\mathbb T))}+\left(\Vert\varphi^a\Vert_{L^\infty(\mathbb R;H^3(\mathbb T))}+\Vert{\varphi}_0\Vert_{H^1_\gamma(-\infty,T;H^3(\mathbb T))}\right)\Vert{\widetilde{\varphi}}^\prime_{x}\Vert_{L^2_\gamma(\mathbb R;H^{m+1}(\mathbb T))}\right.\\
&\qquad\qquad\qquad\quad \left.+\left(\Vert\varphi^a_x\Vert_{L^\infty(\mathbb R;H^{m+1}(\mathbb T))}+\Vert{\varphi}_{0,x}\Vert_{H^1_\gamma(-\infty,T;H^{m+1}(\mathbb T))}\right)\Vert{\widetilde{\varphi}}^\prime_{x}\Vert_{L^2_\gamma(\mathbb R;H^2(\mathbb T))}+\Vert  g\Vert_{L^2_\gamma(-\infty,T;H^{m}(\mathbb T))}\right\}\\
&\le C_1\left\{\left(1+\Vert\varphi^a\Vert_{L^\infty(\mathbb R;H^3(\mathbb T))}+\Vert{\varphi}_0\Vert_{X_2}\right)\Vert{\widetilde{\varphi}}^\prime_{x}
\Vert_{L^2_\gamma(\mathbb R;H^{m+1}(\mathbb T))}\right.\\
&\qquad\qquad\qquad\quad \left.+\left(\Vert\varphi^a\Vert_{L^\infty(\mathbb R;H^{m+2}(\mathbb T))}+\Vert{\varphi}_{0}\Vert_{H^1_\gamma(-\infty,T;H^{m+2}(\mathbb T))}\right)\Vert{\widetilde{\varphi}}^\prime_{x}\Vert_{L^2_\gamma(\mathbb R;H^2(\mathbb T))}+\Vert  g\Vert_{L^2_\gamma(-\infty,T;H^{m}(\mathbb T))}\right\}\,,
\end{split}
\end{equation}
where $C_1$ only depends on $m$, $\delta$, $\mu$, see Proposition \ref{prop_stima_tame_2}. Now we use \eqref{stima_diff} together with the Sobolev imbedding and the definition of the spaces $X_m$ and $Y_m$ to bound, for fixed $\gamma$, $\Vert\widetilde{\varphi}^\prime_{x}\Vert^2_{L^2_\gamma(\mathbb R;H^{m+1}(\mathbb T))}$  and $\Vert\widetilde{\varphi}^\prime_{x}\Vert^2_{L^2_\gamma(\mathbb R;H^{2}(\mathbb T))}$ in the above inequality, finding
\begin{equation}\label{stima_diff_T_NM}
\begin{split}
 & \Vert\widetilde{\varphi}^\prime_{x}\Vert_{L^2_\gamma(\mathbb R;H^{m+1}(\mathbb T))}\\
&\le \widehat{C}_1\left\{\left(\Vert\varphi^a_x\Vert_{L^\infty(\mathbb R;H^{m+3}(\mathbb T))}+\Vert\varphi_{0,x}\Vert_{L^\infty(-\infty, T;H^{m+3}(\mathbb T))}\right)\Vert g\Vert_{L^2_\gamma(-\infty,T; H^2(\mathbb T))} +\Vert g\Vert_{L^2_\gamma(-\infty,T; H^{m+1}(\mathbb T))}\right\}\\
&\le \widehat{C}_1\left\{\left(\Vert\varphi^a\Vert_{L^\infty(\mathbb R;H^{m+4}(\mathbb T))}+\Vert\varphi_{0}\Vert_{H^1_\gamma(-\infty, T;H^{m+4}(\mathbb T))}\right)\Vert g\Vert_{Y_2} +\Vert g\Vert_{Y_{m+1}}\right\}\\
&\le \widehat{C}_1\left\{\left(\Vert\varphi^a\Vert_{L^\infty(\mathbb R;H^{m+4}(\mathbb T))}+\Vert\varphi_{0}\Vert_{X_{m+3}}\right)\Vert g\Vert_{Y_2} +\Vert g\Vert_{Y_{m+1}}\right\},
\end{split}
\end{equation}
which gives, for $m=1$
\begin{equation}\label{stima_diff_T_2_NM}
\Vert\widetilde{\varphi}^\prime_{x}\Vert_{L^2_\gamma(\mathbb R;H^{2}(\mathbb T))}\\
\le \widehat{C}_1\left\{\left(\Vert\varphi^a\Vert_{L^\infty(\mathbb R;H^{5}(\mathbb T))}+\Vert\varphi_{0}\Vert_{X_{4}}\right)\Vert g\Vert_{Y_2} +\Vert g\Vert_{Y_{2}}\right\}.
\end{equation}
The constant $\widehat{C}_1$ involved in \eqref{stima_diff_T_NM}, \eqref{stima_diff_T_2_NM} is the value defined in \eqref{Cgamma}.

Inserting  \eqref{stima_diff_T_NM} and \eqref{stima_diff_T_2_NM} in \eqref{stima_dm-1tildephitt} we get
\begin{equation}\label{stima_dm-phitt_fine}
\begin{split}
&\Vert\varphi^\prime_{tt}\Vert_{L^2_\gamma(-\infty,T;H^{m}(\mathbb T))}\le\\
&\le \widetilde{C}_1\bigg\{\big(1+\Vert\varphi^a\Vert_{L^\infty(\mathbb R;H^3(\mathbb T))}+\Vert{\varphi}_0\Vert_{X_2}\big)\bigg(\left(\Vert\varphi^a\Vert_{L^\infty(\mathbb R;H^{m+4}(\mathbb T))} + \Vert{\varphi}_0\Vert_{X_{m+3}}\right)\Vert g\Vert_{Y_2} + \Vert g\Vert_{Y_{m+1}}\bigg)\bigg.\\
&\quad\ \bigg.+\left(\Vert\varphi^a\Vert_{L^\infty(\mathbb R;H^{m+2}(\mathbb T))}+\Vert{\varphi}_{0}\Vert_{X_{m+1}}\right)\bigg(\left(\Vert\varphi^a\Vert_{L^\infty(\mathbb R;H^{5}(\mathbb T))} + \Vert{\varphi}_0\Vert_{X_{4}}\right)\Vert g\Vert_{Y_2} + \Vert g\Vert_{Y_{2}}\bigg)+\Vert  g\Vert_{Y_{m}}\bigg\}\,,
\end{split}
\end{equation}
with $\widetilde{C}_1=\widetilde{C}_1(m,\delta,\mu,\varphi^{(0)},\varphi^{(1)})$. Adding \eqref{stima_dm-phitt_fine} to \eqref{stima_diff_T}, and recalling \eqref{normaXm}, we get
\begin{equation}\label{stima_phiprimeXm}
\begin{split}
&\Vert\varphi^\prime\Vert_{X_m}\le\\
&\le C\bigg\{\left(1+\Vert\varphi^a\Vert_{L^\infty(\mathbb R;H^3(\mathbb T))}+\Vert{\varphi}_0\Vert_{X_2}\right)\bigg(\left(\Vert\varphi^a\Vert_{L^\infty(\mathbb R;H^{m+4}(\mathbb T))} + \Vert{\varphi}_0\Vert_{X_{m+3}}\right)\Vert g\Vert_{Y_2} + \Vert g\Vert_{Y_{m+1}}\bigg)\bigg.\\
&\quad \bigg.+\left(\Vert\varphi^a\Vert_{L^\infty(\mathbb R;H^{m+2}(\mathbb T))}+\Vert{\varphi}_{0}\Vert_{X_{m+1}}\right)\bigg(\left(\Vert\varphi^a\Vert_{L^\infty(\mathbb R;H^{5}(\mathbb T))} + \Vert{\varphi}_0\Vert_{X_{4}}\right)\Vert g\Vert_{Y_2} + \Vert g\Vert_{Y_{2}}\bigg)+\Vert  g\Vert_{Y_{m}}\bigg\}\,,
\end{split}
\end{equation}
with $C=C(m,\delta,\mu,\varphi^{(0)},\varphi^{(1)})$.

Now we observe that for $\varphi_0\in\mathcal U\cap X_\infty$ (see \eqref{neighbrhd}) in the above inequality the coefficients
\begin{equation*}
\left(1+\Vert\varphi^a\Vert_{L^\infty(\mathbb R;H^3(\mathbb T))}+\Vert{\varphi}_0\Vert_{X_2}\right)\,,\quad \left(\Vert\varphi^a\Vert_{L^\infty(\mathbb R;H^{5}(\mathbb T))} + \Vert{\varphi}_0\Vert_{X_{4}}\right)
\end{equation*}
can be bounded by some constant $C=C(\delta,\Vert\varphi^{(0)}\Vert_{H^5(\mathbb T)},\Vert\varphi^{(1)}\Vert_{H^4(\mathbb T)})$, see \eqref{CX.1}. Hence from \eqref{stima_phiprimeXm} we get
\begin{equation*}
\Vert\varphi^\prime\Vert_{X_m}\le C\left\{\left(\Vert\varphi^a\Vert_{L^\infty(\mathbb R;H^{m+4}(\mathbb T))} + \Vert{\varphi}_0\Vert_{X_{m+3}}\right)\Vert g\Vert_{Y_2} + \Vert g\Vert_{Y_{m+1}}\right\}\,,
\end{equation*}
where $C=C(m,\mu,\delta,\varphi^{(0)},\varphi^{(1)})$. Using also $\Vert g\Vert_{Y_2}\leq\Vert g\Vert_{Y_{m+1}}$ and that (see \eqref{reg_phia})
\begin{equation*}
\Vert\varphi^a\Vert_{L^\infty(\mathbb R;H^{m+4}(\mathbb T))}\le C\left\{\Vert\varphi^{(0)}\Vert_{H^{m+4}(\mathbb T)}+\Vert\varphi^{(1)}\Vert_{H^{m+3}(\mathbb T)}\right\}
\end{equation*}
for all $m\geq 1$, we finally get
\begin{equation}\label{stima_phiprimeXm_fine2}
\Vert\varphi^\prime\Vert_{X_m}\le \tilde C\left\{\Vert{\varphi}_0\Vert_{X_{m+3}}\Vert g\Vert_{Y_2} + \Vert g\Vert_{Y_{m+1}}\right\}\,
\end{equation}
where now $\tilde C$ depends only on $m$, $\delta$, $\mu$ and the initial data $\varphi^{(0)}$, $\varphi^{(1)}$; in particular $\tilde C$ is bounded for $m$ bounded and sufficiently regular initial data.

The above estimate implies that Assumption 2.2 in \ref{sec_N-M} (see also  \cite[Assumption 2.2]{secchi-nash}) is satisfied with $m_0=4$, $s=1$ and $s^\prime=3$ (recall that \eqref{stima_phiprimeXm_fine2} requires $\varphi_0\in\mathcal U\cap X_\infty$, where $\mathcal U$ was defined in \eqref{neighbrhd}, and $\Vert g\Vert_{Y_2}\le\Vert g\Vert_{Y_4}$).

\subsubsection{Estimate for the second order differential of $\mathcal L$}\label{stima_diff_2}
We need to derive an estimate (see  Assumption 2.1 in \ref{sec_N-M}; see also \cite[Assumption 2.1]{secchi-nash}) for the second order derivative $\mathbb L^{\prime\prime}[\varphi^a+\varphi_0](\varphi^\prime, \psi^\prime)$ (see \eqref{diffL}).

From estimate \eqref{stima_der2_T} (see Remark \ref{rem_der2}), for  fixed $\gamma\ge\widehat{\gamma}_1$ we find
\begin{equation*}
\begin{split}
\Vert\mathbb L^{\prime\prime}&[\varphi^a+\varphi_0](\varphi^\prime, \psi^\prime)\Vert_{L^2_{\gamma}(-\infty,T;H^m(\mathbb T))}\\
&\leq C_m\left\{\Vert\varphi^\prime_{x}\Vert_{L^2_\gamma(-\infty,T;H^{m+1}(\mathbb T))}\Vert\psi^\prime_x\Vert_{L^\infty(-\infty,T;H^2(\mathbb T))} + \Vert\psi^\prime_{x}\Vert_{L^2_\gamma(-\infty,T;H^{m+1}(\mathbb T))}\Vert\varphi^\prime_x\Vert_{L^\infty(-\infty,T;H^2(\mathbb T))}\right\}\\
&\leq C_m\left\{\Vert\varphi^\prime_x\Vert_{L^2_\gamma(-\infty,T;H^{m+1}(\mathbb T))}\Vert\psi^\prime\Vert_{H^1_\gamma(-\infty,T;H^3(\mathbb T))} + \Vert\psi^\prime_x\Vert_{L^2_\gamma(-\infty,T;H^{m+1}(\mathbb T))}\Vert\varphi^\prime\Vert_{H^1_\gamma(-\infty,T;H^3(\mathbb T))}\right\}\\
&\leq C_m\left\{\Vert\varphi^\prime\Vert_{X_m}\Vert\psi^\prime\Vert_{X_2} + \Vert\psi^\prime\Vert_{X_m}\Vert\varphi^\prime\Vert_{X_2}\right\}\,,
\end{split}
\end{equation*}
with the constant $C_m$ bounded for $m$ bounded (again we used the Sobolev imbedding $H^1_\gamma(-\infty,T;\mathcal X)\hookrightarrow L^\infty(-\infty,T;\mathcal X)$).

Estimate above is exactly Assumpion 2.1 in \ref{sec_N-M} (see also \cite[Assumption 2.1]{secchi-nash}) with $m_0\geq 2$, $r=0$, $r^\prime =0$.

We need that Assumptions 2.1 and 2.2 in \ref{sec_N-M} (see also \cite{secchi-nash}) are simultaneously satisfied; hence, from now on, we assume $m_0=4$.

\subsection{Proof of Theorem \ref{nonlin_th}}\label{proof-theorem}
We are now in the position to verify that all the assumptions in \cite[Theorem 2.4]{secchi-nash}  are satisfied (see \ref{sec_N-M}, Theorem \ref{N-M_paolo}) in order to get that  the nonlinear equation \eqref{cp1_equiv} has a solution.

From the values of $r,r^\prime, s, s^\prime$ obtained in Sections \ref{stima_diff_1} and \ref{stima_diff_2}, we compute
\begin{equation*}
m^\prime=m_0+\max{\{r,r^\prime\}} + \max{\{s,s^\prime\}}=4+\max{\{0,0\}} + \max{\{1,3\}}=7.
\end{equation*}
By (i) in \cite[Theorem 2.4]{secchi-nash} (see \ref{sec_N-M}, Theorem \ref{N-M_paolo}), there exists $\varepsilon$ such that, if $F^a\in Y_{m^\prime+s+1}=Y_9$ with
\begin{equation}\label{norma-piccola-nostro-caso}
\Vert F^a\Vert_{Y_9}=\Vert F^a\Vert_{L^2_\gamma(-\infty,T; H^9(\mathbb T))}\leq \varepsilon\,,
\end{equation}
there exists a solution $\varphi^\prime \in X_{m^\prime}=X_7$ of \eqref{cp1_equiv}, provided that Assumptions 2.1.and 2.2 are satisfied. In view of Remark \ref{rmk:1}, the assumptions $\varphi^{(0)}\in H^{11}(\mathbb T)$ and $\varphi^{(1)}\in H^{10}(\mathbb T)$ imply $F^a\in Y_{9}$. On the other hand, the smallness assumption \eqref{norma-piccola-nostro-caso} is guaranteed if we take a positive $T$ sufficiently small, depending on the size of the initial data $\varphi^{(0)}\in H^{11}(\mathbb T)$ and $\varphi^{(1)}\in H^{10}(\mathbb T)$, see again Remark \ref{rmk:2}. In view of \eqref{reg_phia} and \eqref{Xm}, the function $\varphi=\varphi^a_{\vert\,[0,T]}+\varphi^\prime_{\vert [0,T]}$ provides a solution of \eqref{cp} with the regularity required. The uniqueness  follows from standard arguments, based on the regularity of the solution.  This ends the proof of the statement (1).

Moreover, by (ii) in \cite[Theorem 2.4]{secchi-nash} (see also \ref{sec_N-M}, Theorem \ref{N-M_paolo}), if \eqref{norma-piccola-nostro-caso} holds and if in addition $F^a\in Y_{m^{\prime\prime}+s+1}=Y_{m^{\prime\prime}+2}$ with $m^{\prime\prime}>m^\prime=7$, then the solution $\varphi^\prime$ of \eqref{cp1_equiv} belongs to $X_{m^{\prime\prime}}$, again under the condition that Assumptions 2.1.and 2.2 are satisfied.

Observe now that if $\varphi^{(0)}\in H^{\nu+1}(\mathbb T)$ and $\varphi^{(1)}\in H^\nu(\mathbb T)$, in view of Remark \ref{rmk:1}, we know that $F^a\in Y_{\nu-1}$. Hence, in order to satisfy the above case (ii) of \cite[Theorem 2.4]{secchi-nash}, we need to require that
\begin{equation*}
\nu-1=m^{\prime\prime}+2>9
\end{equation*}
that is ensured by the assumption $\nu>10$. For such $\nu>10$, Nash-Moser's theorem implies that the solution of \eqref{cp1_equiv} belongs to $X_{m^{\prime\prime}}=X_{\nu-3}$, which yields for the solution $\varphi=\varphi^a_{\vert\,[0,T]}+\varphi^\prime_{\vert [0,T]}$ of \eqref{cp} the required regularity. Notice in particular that the final time $T$ is just defined from the requirement \eqref{norma-piccola-nostro-caso}, where the norm of $F^a$ in $L^2_\gamma(-\infty,T; H^{9}(\mathbb T))$ is involved; because of Remark \ref{rmk:2} the time $T$ may be chosen to depend on the norms of the data $\varphi^{(0)}$, $\varphi^{(1)}$ respectively in $H^{11}(\mathbb T)$ and $H^{10}(\mathbb T)$, in spite of the augmented regularity $\varphi^{(0)}\in H^{\nu+1}(\mathbb T)$, $\varphi^{(1)}\in H^{\nu}(\mathbb T)$.

\appendix
\section{Some commutator and product estimates}\label{stima_commutatore}
\begin{lemma}\label{lemma_comm}
For $s>1/2$ there exists a constant $C_s>0$ such that
\begin{eqnarray}
\Vert \left[\mathbb H\,;\,v\right]f\Vert_{L^2(\mathbb T)}\le C_s\Vert v\Vert_{H^s(\mathbb T)}\Vert f\Vert_{L^2(\mathbb T)}\,,\quad\forall\,v\in H^s(\mathbb T)\,,\,\,\forall\,f\in L^2(\mathbb T)\,;\label{stima_comm_1}\\
\Vert \left[\mathbb H\,;\,v\right]f_x\Vert_{L^2(\mathbb T)}\le C_s\Vert v_x\Vert_{H^s(\mathbb T)}\Vert f\Vert_{L^2(\mathbb T)}\,,\quad\forall\,v\in H^{s+1}(\mathbb T)\,,\,\,\forall\,f\in L^2(\mathbb T)\,;\label{stima_comm_2}\\
\left\Vert \left(\left[\mathbb H\,;\,\left[\mathbb H\,;\,v\right]\right]f_x\right)_x\right\Vert_{L^2(\mathbb T)}\le C_s\Vert v_{xx}\Vert_{H^s(\mathbb T)}\Vert f\Vert_{L^2(\mathbb T)}\,,\quad\forall\,v\in H^{s+2}(\mathbb T)\,,\,\,\forall\,f\in L^2(\mathbb T)\,,\label{stima_comm_3}
\end{eqnarray}
where $\left[\mathbb H\,;\,v\right]$ is the commutator between the Hilbert transform $\mathbb H$ and the multiplication by $v$.
\end{lemma}
\begin{proof}
We show here only the estimate \eqref{stima_comm_3}, the proof of estimates \eqref{stima_comm_1}, \eqref{stima_comm_2} following from similar arguments. Let us first recall that, whenever $f$ and $g$ are sufficiently smooth periodic functions on $\mathbb T$, there holds
\begin{equation}\label{convoluzione}
\widehat{f\cdot g}=\frac1{2\pi}\widehat{f}\ast\widehat{g}\,,\quad\forall\,k\in\mathbb Z\,,
\end{equation}
where $\widehat f\ast\widehat g$ is the discrete convolution of the sequences $\widehat f:=\left\{\widehat f(k)\right\}_{k\in\mathbb Z}$ and $\widehat g:=\left\{\widehat g(k)\right\}_{k\in\mathbb Z}$ defined by
\begin{equation}\label{conv}
\widehat f\ast\widehat g(k):=\sum\limits_{\ell}\widehat{f}(k-\ell)\widehat{g}(\ell)\,,\qquad\forall\,k\in\mathbb Z\,.
\end{equation}

Making use of \eqref{convoluzione}, \eqref{conv} for $k\neq 0$ we compute
\begin{align*}\label{stima_comm_3.1}
\big(\left(\left[\mathbb H\,;\,\left[\mathbb H\,;\,v\right]\right]f_x\right)_x\big)^{\wedge}&(k)=ik\big(\left[\mathbb H\,;\,\left[\mathbb H\,;\,v\right]\right]f_x\big)^{\wedge}(k)=ik\left\{\big(\mathbb H[\left[\mathbb H\,;\,v\right]f_x]\big)^{\wedge}(k)-\big(\left[\mathbb H\,;\,v\right]\mathbb H[f_x]\big)^{\wedge}(k)\right\}\\
&=ik\left\{-i\,{\rm sgn}\,k\,\big(\left[\mathbb H\,;\,v\right]f_x\big)^\wedge(k)+\big(v\,\mathbb H^2[f_x]\big)^\wedge(k)-\big(\mathbb H[v\,\mathbb H[f_x]]\big)^\wedge(k)\right\}\\
&=ik\left\{-i\,{\rm sgn}\,k\,\left(\widehat{\mathbb H[v\,f_x]}(k)-\widehat{v\,\mathbb H[f_x]}(k)\right)-\widehat{v\,f_x}(k)+i\,{\rm sgn}\,k\,\widehat{v\,\mathbb H[f_x]}(k)\right\}\\
&=ik\left\{-i\,{\rm sgn}\,k\,\left(-i\,{\rm sgn}\,k\,\widehat{v\,f_x}(k)-\frac1{2\pi}\sum\limits_{\ell}\widehat{v}(k-\ell)\widehat{\mathbb H[f_x]}(\ell)\right)-\widehat{v\cdot f_x}(k)\right.\\
&\quad\left.+\frac1{2\pi}\,i\,{\rm sgn}\,k\,\sum\limits_{\ell}\widehat{v}(k-\ell)\widehat{\mathbb H[f_x]}(\ell)\right\}\\
&=ik\left\{-({\rm sgn}\,k)^2\,\widehat{v\, f_x}(k)+\frac1{2\pi}i\,{\rm sgn}\,k\,\sum\limits_{\ell}\widehat{v}(k-\ell)\widehat{\mathbb H[f_x]}(\ell)-\widehat{v\, f_x}(k)\right.\\
&\quad\left. +\frac1{2\pi}\,i\,{\rm sgn}\,k\,\sum\limits_{\ell}\widehat{v}(k-\ell)\widehat{\mathbb H[f_x]}(\ell)\right\}\\
&=ik\left\{-2\widehat{v\, f_x}(k)+\frac1{\pi}i\,{\rm sgn}\,k\,\sum\limits_{\ell}\widehat{v}(k-\ell)\widehat{\mathbb H[f_x]}(\ell)\right\}\\
&=ik\left\{\frac1{\pi}i\,{\rm sgn}\,k\,\sum\limits_{\ell}\widehat{v}(k-\ell)\widehat{\mathbb H[f_x]}(\ell)-\frac1{\pi}\sum\limits_{\ell}\widehat{v}(k-\ell)\widehat{f_x}(\ell)\right\}\\
&=ik\left\{\frac1{\pi}i\,{\rm sgn}\,k\,\sum\limits_{\ell}\widehat{v}(k-\ell)\,\ell\,{\rm sgn}\,\ell\,\widehat{f}(\ell)-\frac1{\pi}\sum\limits_{\ell}\widehat{v}(k-\ell)i\,\ell\,\widehat{f}(\ell)\right\}\\
&=\frac1{2\pi}\sum\limits_{\ell}\Lambda_1(k,\ell)\,\widehat{v}(k-\ell)\widehat{f}(\ell)\,,
\end{align*}
where we also used $({\rm sgn}\,k)^2=1$, $k\,{\rm sgn}\,k=|k|$ and we have set
\begin{equation*}
\Lambda_1(k,\ell)=2(k\ell-|k|\,|\ell|)\,.
\end{equation*}
Observing that $\Lambda_1(k,\ell)=0$ for $k\ell\ge 0$ and using the numerical inequality $-2k\ell\le (k-\ell)^2$ we get
\begin{equation*}
\begin{array}{ll}
\displaystyle\left\vert\big(\left(\left[\mathbb H\,;\,\left[\mathbb H\,;\,v\right]\right]f_x\right)_x\big)^\wedge(k)\right\vert\le \frac1{2\pi}\sum\limits_{\ell\,:\,\,k\ell< 0}(-4k\ell)\,\vert\widehat{v}(k-\ell)\vert\,\vert\widehat{f}(\ell)\vert\le \frac1{\pi}\sum\limits_{\ell}(k-\ell)^2\,\vert\widehat{v}(k-\ell)\vert\,\vert\widehat{f}(\ell)\vert\\
\displaystyle\quad =\frac1{\pi}\sum\limits_{\ell}\,\vert (i(k-\ell))^2\widehat{v}(k-\ell)\vert\,\vert\widehat{f}(\ell)\vert=\frac1{\pi}\sum\limits_{\ell}\,\vert \widehat{v_{xx}}(k-\ell)\vert\,\vert\widehat{f}(\ell)\vert=\frac1{\pi}\left(\vert\widehat{v_{xx}}\vert\ast\vert\widehat{f}\vert\right)(k)\,.
\end{array}
\end{equation*}
Because of Parseval's identity and using Young's inequality we derive
\begin{align*}
\Vert &\left(\left[\mathbb H\,;\,\left[\mathbb H\,;\,v\right]\right]f_x\right)_x\Vert^2_{L^2(\mathbb T)}=\frac1{2\pi}\sum\limits_{k\in\mathbb Z}\left\vert \big(\left(\left[\mathbb H\,;\,\left[\mathbb H\,;\,v\right]\right]f_x\right)_x\big)^\wedge(k)\right\vert^2\le\frac{1}{2\pi^2}\sum\limits_{k\in\mathbb Z}\left\vert\left(\vert\widehat v_{xx}\vert\ast\vert\widehat f\vert\right)(k)\right\vert^2\\
&\le \frac{1}{2\pi^2}\left(\sum\limits_{k\in\mathbb Z}\vert\widehat v_{xx}(k)\vert\right)^2\sum\limits_{k\in\mathbb Z}\vert\widehat{f}(k)\vert^2=\frac{1}{\pi}\left(\sum\limits_{k\in\mathbb Z}\vert\widehat v_{xx}(k)\vert\right)^2\Vert f\Vert^2_{L^2(\mathbb T)}\,.
\end{align*}
Estimate \eqref{stima_comm_3} follows at once from the above, since for $s>1/2$
\begin{equation}\label{imm_sobolev}
\sum\limits_{k\in\mathbb Z}\vert\widehat v_{xx}(k)\vert\le c_s\Vert v_{xx}\Vert_{H^s(\mathbb T)}
\end{equation}
holds true with some positive constant $c_s$, depending only on $s$.
\end{proof}
\begin{lemma}\label{lemma_comm_paolo}
For every integer $m\ge 1$ there exists a constant $C_m>0$ such that for all functions $v\in H^m(\mathbb T)$ and $f\in H^1(\mathbb T)$
\begin{equation*}
\Vert \left[\mathbb H\,;\,v\right]f\Vert_{H^m(\mathbb T)}\le C_m\Vert \partial^m_x v\Vert_{L^{2}(\mathbb T)}\Vert f\Vert_{H^1(\mathbb T)}\,.
\end{equation*}
\end{lemma}
\begin{proof}
In view of \eqref{hilbert1} we compute
\begin{equation*}
\begin{array}{ll}
\displaystyle\widehat{\left[\mathbb H\,;\,v\right]f}(k)=\widehat{\mathbb H[vf]}(k)-\widehat{v\mathbb H[f]}(k)=-i\,{\rm sgn}\,k\,\widehat{vf}(k)-\frac1{2\pi}\sum\limits_{\ell}\widehat{v}(k-\ell)\widehat{\mathbb H[f]}(\ell)\\
\displaystyle\quad =-i\,{\rm sgn}\,k\,\frac1{2\pi}\sum\limits_{\ell}\widehat{v}(k-\ell)\widehat{f}(\ell)+\frac1{2\pi}\sum\limits_{\ell}i\,{\rm sgn}\,\ell\,\widehat{v}(k-\ell)\widehat{f}(\ell)\\
\displaystyle\quad =\frac1{2\pi}\sum\limits_{\ell}\Lambda(k,\ell)\,\widehat{v}(k-\ell)\widehat{f}(\ell)\,,
\end{array}
\end{equation*}
where
\begin{equation}\label{valori_lambda}
\Lambda(k,\ell):=i\,\left(-{\rm sgn}\,k+{\rm sgn}\,\ell\right)=
\begin{cases}
-2i\,,\qquad\mbox{if}\,\,k>0\,\,\mbox{and}\,\,\ell<0\,,\\
\mp i\,,\qquad\mbox{if}\,\,\pm k>0\,\,\mbox{and}\,\,\ell=0\,,\\
0\,,\qquad\mbox{if}\,\,k\ell>0\,,
\end{cases}
\end{equation}
and
\begin{equation}\label{simmetria_lambda}
\Lambda(\ell,k)=-\Lambda(k,\ell)\,,\quad\forall\,(k,\ell)\in\mathbb Z^2\,.
\end{equation}
In view of \eqref{valori_lambda} and \eqref{simmetria_lambda} we have more explicitly that
\begin{equation}\label{casi}
\widehat{\left[\mathbb H\,;\,v\right]f}(k)=\begin{cases}\displaystyle\frac{-i}{2\pi}\widehat{v}(k)\widehat{f}(0)-\frac{i}{\pi}\sum\limits_{\ell<0}\widehat{v}(k-\ell)\widehat{f}(\ell)\,,\quad\mbox{if}\,\,k>0\,,\\
\displaystyle-\frac{i}{2\pi}\sum\limits_{\ell}{\rm sgn}\,\ell\,\widehat{v}(-\ell)\widehat{f}(\ell)\,,\quad\mbox{if}\,\,k=0\,,\\
\displaystyle\frac{i}{2\pi}\widehat{v}(k)\widehat{f}(0)+\frac{i}{\pi}\sum\limits_{\ell>0}\widehat{v}(k-\ell)\widehat{f}(\ell)\,,\quad\mbox{if}\,\,k<0\,.
\end{cases}
\end{equation}
On the other hand, in view of \eqref{normaHs}
\begin{equation}\label{stima_paolo_1}
\begin{split}
\Vert &\left[\mathbb H\,;\,v\right]f\Vert^2_{H^m(\mathbb T)}=\frac1{2\pi}\sum\limits_{k\in \mathbb Z}(1+|k|)^{2m}\vert\widehat{\left[\mathbb H\,;\,v\right]f}(k)\vert^2\\
&\le C_m\left\{\vert\widehat{\left[\mathbb H\,;\,v\right]f}(0)\vert^2+\sum\limits_{k\in \mathbb Z\setminus\{0\}}|k|^{2m}\vert\widehat{\left[\mathbb H\,;\,v\right]f}(k)\vert^2\right\}\,.
\end{split}
\end{equation}
From \eqref{casi} we get for $k=0$:
\begin{equation}\label{coeff_0}
\vert \widehat{\left[\mathbb H\,;\,v\right]f}(0)\vert\le\frac{1}{2\pi}\sum\limits_{\ell\neq 0}\vert\widehat{v}(-\ell)\vert\,\vert\widehat{f}(\ell)\vert\le \frac{1}{2\pi}\sum\limits_{\ell\neq 0}\vert -i\ell\vert^m \vert\widehat{v}(-\ell)\vert\vert\widehat{f}(\ell)\vert\le\frac1{2\pi}\left(\vert\widehat{\partial_x^m v}\vert\ast\vert \widehat f\vert\right)(0)\,.
\end{equation}
For $k>0$ we obtain
\begin{equation}\label{coeff_k+}
\begin{split}
\vert k\vert^m\,\vert & \widehat{\left[\mathbb H\,;\,v\right]f}(k)\vert\le \frac{1}{2\pi}\vert k\vert^m \vert\widehat{v}(k)\vert\,\vert\widehat{f}(0)\vert+\frac{1}{\pi}\sum\limits_{\ell<0}\vert ik\vert^m\,\vert\widehat{v}(k-\ell)\vert\,\vert\widehat{f}(\ell)\vert\\
&\le \frac{1}{2\pi}\vert(ik)^m\widehat{v}(k)\vert\,\vert\widehat{f}(0)\vert+\frac{1}{\pi}\sum\limits_{\ell<0}\vert i(k-\ell)\vert^m\,\vert\widehat{v}(k-\ell)\vert\,\vert\widehat{f}(\ell)\vert\\
&\le\frac1{\pi}\sum\limits_{\ell\le 0}\vert\widehat{\partial^m_x v}(k-\ell)\vert\,\vert\widehat{f}(\ell)\vert\le\frac1{\pi}(\vert\widehat{\partial^m_x v}\vert\ast\vert\widehat{f}\vert)(k)\,,
\end{split}
\end{equation}
where we used that $\vert k\vert<\vert k-\ell\vert$ for $k>0$ and $\ell<0$. The same estimate as above can be extended to the coefficients $\vert\widehat{\left[\mathbb H\,;\,v\right]f}(k)\vert$, for $k<0$, by repeating the same arguments and since the inequality $\vert k\vert<\vert k-\ell\vert$ is still true for $k<0$ and $\ell>0$.

Using \eqref{coeff_0}, \eqref{coeff_k+} to estimate the right-hand side of \eqref{stima_paolo_1}, by Young's inequality and \eqref{imm_sobolev} (with $f$ instead of $v_{xx}$ and $s=1$) we obtain
\begin{equation*}\label{stima_paolo_2}
\begin{split}
\Vert &\left[\mathbb H\,;\,v\right]f\Vert^2_{H^m(\mathbb T)}\le C_m\sum\limits_{k\in \mathbb Z}\left\vert\left(\vert\widehat{\partial^m_x v}\vert\ast\vert\widehat{f}\vert\right)(k)\right\vert^2\le C_m\left(\sum\limits_{k\in \mathbb Z}\vert\widehat{\partial^m_x v}(k)\vert^2\right)\left(\sum\limits_{k\in \mathbb Z}\vert\widehat{f}(k)\vert\right)^2\\
&\le C_m\Vert\partial^m_x v\Vert^2_{L^2(\mathbb T)}\Vert f\Vert^2_{H^1(\mathbb T)}\,.
\end{split}
\end{equation*}
\end{proof}

Following the same arguments as in the case of Lemma \ref{lemma_comm_paolo}, we may prove the following more general result.
\begin{lemma}\label{lemma_comm_ale_2}
For all integers $m, p\ge 1$ there exists a constant $C_{m,p}>0$ such that for all functions $v\in H^{m+p}(\mathbb T)$ and $f\in H^1(\mathbb T)$
\begin{equation*}
\Vert \left[\mathbb H\,;\,v\right]\partial^p_x f\Vert_{H^m(\mathbb T)}\le C_{m,p}\Vert \partial^{m+p}_x v\Vert_{L^{2}(\mathbb T)}\Vert f\Vert_{H^1(\mathbb T)}\,.
\end{equation*}
\end{lemma}

The following product and commutator estimates are a consequence of the well known Gagliardo-Nirenberg inequalities, see f.i. \cite{majda}.
\begin{lemma}\label{lemma_comm_1}
For every integer $m\ge 1$ there exists a constant $C_m>0$ such that the following holds:
\begin{itemize}
\item[i.] for all functions $f\,,v\in H^m(\mathbb T)\cap L^\infty(\mathbb T)$
\begin{equation}
\Vert vf\Vert_{H^m(\mathbb T)}\le C_m\left\{\Vert v\Vert_{L^\infty(\mathbb T)}\Vert f\Vert_{H^m(\mathbb T)}+\Vert v\Vert_{H^m(\mathbb T)}\Vert f\Vert_{L^\infty(\mathbb T)}\right\}\,;\label{stima_prod_4}
\end{equation}
\item[ii.] for all positive integers $k\le m$ and functions $f\,,v\in H^m(\mathbb T)\cap L^\infty(\mathbb T)$
\begin{equation}
\Vert \left[\partial_x^k\,;\,v\right]f\Vert_{L^2(\mathbb T)}\le C_m\left\{\Vert v\Vert_{L^\infty(\mathbb T)}\Vert f\Vert_{H^k(\mathbb T)}+\Vert v\Vert_{H^k(\mathbb T)}\Vert f\Vert_{L^\infty(\mathbb T)}\right\}\,;\label{stima_comm_4}
\end{equation}
\item[iii.] for all positive integers $k\le m$ and functions $f\in H^{m-1}(\mathbb T)\cap L^\infty(\mathbb T)$ and $v\in H^m(\mathbb T)$ such that $v_x\in L^\infty(\mathbb T)$
\begin{equation}
\Vert \left[\partial_x^k\,;\,v\right]f\Vert_{L^2(\mathbb T)}\le C_m\left\{\Vert v_x\Vert_{L^\infty(\mathbb T)}\Vert f\Vert_{H^{k-1}(\mathbb T)}+\Vert v\Vert_{H^k(\mathbb T)}\Vert f\Vert_{L^\infty(\mathbb T)}\right\}\,.\label{stima_comm_5}
\end{equation}
\end{itemize}
\end{lemma}
\begin{lemma}\label{lemma_prod_alg}
For all real numbers $s>1/2$ and $m\in\mathbb N$, the set inclusion $H^s(\mathbb T)\cdot H^m(\mathbb T)\subset H^m(\mathbb T)$ holds with continuous imbedding. In particular, for $m=s\ge1$ the space $H^s(\mathbb T)$ is an algebra for the point-wise product of functions.
\end{lemma}
\begin{lemma}\label{lemma_comm_2}
For every integer $m\ge 1$ there exists a positive constant $C_m$ such that for all functions $v\in H^{m+1}(\mathbb T)$, $f\in H^2(\mathbb T)$
\begin{equation}\label{stima_comm_6}
\left\Vert \left[\mathbb H\,;\,\left[\partial_x^m\,;\,v\right]\right]\partial^2_x f\right\Vert_{L^2(\mathbb T)}\le C_m\Vert v_x\Vert_{H^m(\mathbb T)}\Vert f_x\Vert_{H^1(\mathbb T)}\,.
\end{equation}
\end{lemma}
\begin{proof}
We proceed again as in the proof of Lemma \ref{lemma_comm}, by computing explicitly the Fourier coefficients of the function $\left[\mathbb H\,;\,\left[\partial_x^m\,;\,v\right]\right]\partial^2_x f$. For every integer $k$ we find
\begin{equation}\label{stima_comm_6_1}
\begin{split}
\left(\left[\mathbb H\,;\,\left[\partial_x^m\,;\,v\right]\right]\partial_x^2 f\right)^{\wedge}(k)=\bigg(\mathbb H\left[\left[\partial_x^m\,;\,v\right]\partial^2_x f\right]-\left[\partial_x^m\,;\,v\right]\mathbb H\left[\partial^2_x f\right]\bigg)^{\wedge}(k)\\
=-i{\rm sgn}\,k\,\left(\left[\partial_x^m\,;\,v\right]\partial^2_x f\right)^{\wedge}(k)-\left(\left[\partial_x^m\,;\,v\right]\partial_x^2\mathbb H\left[f\right]\right)^{\wedge}(k)\,.
\end{split}
\end{equation}
Let us first develop the Fourier coefficients of $\left[\partial_x^m\,;\,v\right]\partial^2_x f$; we compute
\begin{equation*}
\begin{split}
&\left(\left[\partial_x^m\,;\,v\right]\partial^2_x f\right)^{\wedge}(k)=\left(\partial_x^m\left(v\partial_x^2 f\right)-v\partial^{m+2}_x f\right)^{\wedge}(k)\\
&\qquad=(ik)^m\left(v\partial_x^2 f\right)^{\wedge}(k)-\frac1{2\pi}\sum\limits_{\ell}\widehat{v}(k-\ell)\left(\partial^{m+2}_x f\right)^{\wedge}(\ell)\\
&\qquad=\frac1{2\pi}(ik)^m\sum\limits_{\ell}\widehat{v}(k-\ell)(i\ell)^2\widehat{f}(\ell)-\frac1{2\pi}\sum\limits_{\ell}\widehat{v}(k-\ell)(i\ell)^{m+2}\widehat{f}(\ell)\\
&\qquad=-\frac1{2\pi}i^m\sum\limits_{\ell}\ell^2(k^m-\ell^m)\widehat{v}(k-\ell)\widehat{f}(\ell)=-\frac1{2\pi}i^m\sum\limits_{\ell\neq 0}\ell^2(k^m-\ell^m)\widehat{v}(k-\ell)\widehat{f}(\ell)\,.
\end{split}
\end{equation*}
Substituting the above expression in \eqref{stima_comm_6_1} we get
\begin{equation*}
\begin{split}
&\left(\left[\mathbb H\,;\,\left[\partial_x^m\,;\,v\right]\right]\partial_x^2 f\right)^{\wedge}(k)=\\
&\qquad =\frac{i^{m+1}}{2\pi}{\rm sgn}\,k\,\sum\limits_{\ell\neq 0}\ell^2(k^m-\ell^m)\widehat{v}(k-\ell)\widehat{f}(\ell)      +\frac1{2\pi}i^m\sum\limits_{\ell\neq 0}\ell^2(k^m-\ell^m)\widehat{v}(k-\ell)\widehat{\mathbb H\left[f\right]}(\ell)\\
&\qquad =\frac{i^{m+1}}{2\pi}{\rm sgn}\,k\,\sum\limits_{\ell\neq 0}\ell^2(k^m-\ell^m)\widehat{v}(k-\ell)\widehat{f}(\ell)      -\frac{i^{m+1}}{2\pi}\sum\limits_{\ell\neq 0}\ell^2(k^m-\ell^m)\widehat{v}(k-\ell)\,{\rm sgn}\,\ell\,\widehat{f}(\ell)\\
&\qquad =\frac{i^{m+1}}{2\pi}\sum\limits_{\ell\neq 0}\left({\rm sgn}\,k-{\rm sgn}\,\ell\right)\ell^2(k^m-\ell^m)\widehat{v}(k-\ell)\,\widehat{f}(\ell)\,.
\end{split}
\end{equation*}
In view of \eqref{segno} we have for $k>0$
\begin{equation*}\label{casi+}
{\rm sgn}\,k-{\rm sgn}\,\ell=\begin{cases}0\,,\qquad\mbox{if}\,\,\ell>0\,,\\ 2\,,\qquad\mbox{if}\,\,\ell<0\,;\end{cases}
\end{equation*}
for $k=0$
\begin{equation*}\label{casi0}
{\rm sgn}\,k-{\rm sgn}\,\ell=\begin{cases}-1\,,\qquad\mbox{if}\,\,\ell>0\,,\\ 1\,,\qquad\mbox{if}\,\,\ell<0\,;\end{cases}
\end{equation*}
for $k<0$
\begin{equation*}\label{casi-}
{\rm sgn}\,k-{\rm sgn}\,\ell=\begin{cases}-2\,,\qquad\mbox{if}\,\,\ell>0\,,\\ 0\,,\qquad\mbox{if}\,\,\ell<0\,.\end{cases}
\end{equation*}
Consequently we obtain
\begin{equation}\label{stima_comm_6_3}
\begin{split}
\left(\left[\mathbb H\,;\,\left[\partial_x^m\,;\,v\right]\right]\partial_x^2 f\right)^{\wedge}(k)=\begin{cases}\displaystyle\frac{i^{m+1}}{\pi}\sum\limits_{\ell<0}\ell^2(k^m-\ell^m)\widehat{v}(k-\ell)\,\widehat{f}(\ell)\,,\quad\mbox{if}\,\,k>0\,,\\ \displaystyle-\frac{i^{m+1}}{\pi}\sum\limits_{\ell>0}\ell^2(k^m-\ell^m)\widehat{v}(k-\ell)\,\widehat{f}(\ell)\,,\quad\mbox{if}\,\,k<0\,,\\
\displaystyle\frac{i^{m+1}}{2\pi}\sum\limits_{\ell\neq 0}{\rm sgn}\,\ell\,\ell^{m+2}\widehat{v}(-\ell)\,\widehat{f}(\ell)\,,\quad\mbox{if}\,\,k=0\,.\end{cases}
\end{split}
\end{equation}
In view of the above formulas, for $k=0$ we immediately get:
\begin{equation*}\label{stima_comm_6_4}
\begin{split}
&\left\vert\left(\left[\mathbb H\,;\,\left[\partial_x^m\,;\,v\right]\right]\partial_x^2 f\right)^{\wedge}(0)\right\vert\le\frac1{2\pi}\sum\limits_{\ell\neq 0}\vert(-i\ell)^m\widehat{v}(-\ell)\vert\vert(i\ell)^2\widehat{f}(\ell)\vert\\
&\quad =\frac1{2\pi}\sum\limits_{\ell\neq 0}\vert\widehat{\partial_x^m v}(-\ell)\vert\widehat{\partial_x^2 f}(\ell)\vert\le\frac1{2\pi}\left(\vert\widehat{\partial_x^mv}\vert\ast\vert\widehat{\partial^2_x f}\vert\right)(0)\,.
\end{split}
\end{equation*}

If $k>0$ we expand the factor $(k^m-\ell^m)$ in the corresponding expression in \eqref{stima_comm_6_3} by Newton's formula to get
\begin{equation*}\label{stima_comm_6_5}
\begin{split}
&\left(\left[\mathbb H\,;\,\left[\partial_x^m\,;\,v\right]\right]\partial_x^2 f\right)^{\wedge}(k)=\frac{i^{m+1}}{\pi}\sum\limits_{\ell<0}\ell^2(k^m-\ell^m)\widehat{v}(k-\ell)\,\widehat{f}(\ell)\\
&\quad=\frac{i^{m+1}}{\pi}\sum\limits_{\ell<0}\left(\sum\limits_{h=0}^{m-1}\binom{m}{h}(k-\ell)^{m-h}\ell^h\right)\widehat{v}(k-\ell)\,\ell^2\,\widehat{f}(\ell)\,.
\end{split}
\end{equation*}
For $k>0$ and $\ell<0$ one has $0<\vert\ell\vert=-\ell\le k-\ell$, hence there exists $C>0$, depending only on $m$, such that
\begin{equation*}
\left\vert\sum\limits_{h=0}^{m-1}\binom{m}{h}(k-\ell)^{m-h}\ell^h\right\vert\le C(k-\ell)^{m}
\end{equation*}
and
\begin{equation*}\label{stima_comm_6_6}
\begin{split}
&\left\vert\left(\left[\mathbb H\,;\,\left[\partial_x^m\,;\,v\right]\right]\partial_x^2 f\right)^{\wedge}(k)\right\vert\le\frac{C}{\pi}\sum\limits_{\ell<0}\vert (i(k-\ell))^m\widehat{v}(k-\ell)\vert\,\vert\widehat{(i\ell)^2 f}(\ell)\vert\le\frac{C}{\pi}\left(\vert\widehat{\partial_x^m v}\vert\ast\vert\widehat{\partial_x^2 f}\vert\right)(k)\,.
\end{split}
\end{equation*}

If $k<0$ and $\ell>0$ we have again $\vert\ell\vert=\ell<\ell-k=\vert k-\ell\vert$ and all the arguments used in the case $k>0$ can be repeated to get the same estimates as above for the Fourier coefficients $\left(\left[\mathbb H\,;\,\left[\partial_x^m\,;\,v\right]\right]\partial_x^2 f\right)^{\wedge}(k)$.

Then  we end the proof of the estimate \eqref{stima_comm_6} by arguing as in the proof of Lemma \ref{lemma_comm}, by using Parseval and Young's inequalities.
\end{proof}

\section{Nash-Moser's theorem}\label{sec_N-M}
In this section, for reader's convenience, we report the assumptions and the Nash-Moser's theorem in \cite{secchi-nash} (we adopt the same notation in \cite{secchi-nash}) that we apply in order to get the existence of the solution of the nonlinear problem \eqref{cp}.

\smallskip

Let $\{X_m\}_{m\geq 0}$ and $\{Y_m\}_{m\geq 0}$ be two decreasing families of Banach spaces, each satisfying the {\it smoothing hypothesis} (see \cite{secchi-nash}).  Let $\mathcal L:X_m\rightarrow Y_m$ be twice differentiable for every $m\geq 0$,  with $\mathcal L (0)=0$.
\medskip

{\bf Assumption 2.1:} For all $u\in U\cap X_{\infty}$, where $U$ is a bounded open neighborhood of $0$ in $X_{m_0}$ for some $m_0\geq 0$, the function $\mathcal L:X_m\rightarrow Y_m$ satisfies the tame estimate
\begin{equation*}\label{assumption2.1}
\Vert d^2\mathcal L (u)(v_1,v_2)\Vert_{Y_m} \leq C\left( \Vert v_1\Vert_{X_{m+r}} \Vert v_2\Vert_{X_{m_0}} + \Vert v_1\Vert_{X_{m_0}} \Vert v_2\Vert_{X_{m+r}} +\Vert v_1\Vert_{X_{m_0}} \Vert v_2\Vert_{X_{m_0}} \left(1+\Vert u\Vert_{X_{m+r^\prime}} \right)\right)\,,
\end{equation*}
for all $m\geq 0$ and for all $v_1,v_2\in X_{\infty}$, for some fixed integers $r,r^\prime \geq 0$. The constant $C$ is bounded for $m$ bounded.

\medskip

{\bf Assumption 2.2:} For all $u\in U\cap X_{\infty}$, there exists a linear mapping $\Psi(u):Y_\infty \rightarrow X_\infty$ such that $d\mathcal L(u)\Psi(u)=Id$ and satisfying the tame estimate
\begin{equation*}\label{assumtion2.2}
\Vert\Psi(u)g\Vert_{X_m}\leq C\left( \Vert g\Vert_{Y_{m+s}} + \Vert g\Vert_{Y_{m_0}}\Vert u\Vert_{X_{m+s^\prime}}\right)
\end{equation*}
for all $m\geq 0$ and  some fixed integers $s,s^\prime \geq 0$. The constant $C$ is bounded for $m$ bounded.

\medskip

The Nash-Moser's theorem requires a family of smoothing operators (see \cite[Definition 2.3]{secchi-nash}). We recall here the statement of \cite[Theorem 2.4]{secchi-nash} which is a suitable version of Nash-Moser's theorem.
\medskip

\begin{theorem}\label{N-M_paolo}
Let $\{X_m\}_{m\geq 0}$ and $\{Y_m\}_{m\geq 0}$ be two decreasing families of Banach spaces, each satisfying the smoothing hypothesis, and assume that both Assumptions 2.1 and 2.2 hold. Let $m^\prime$ be a positive integer such that $m^\prime\geq m_0+\max{\{r,r^\prime\}}+ \max{\{s,s^\prime\}}$.
\begin{itemize}
\item[(i)] There exists a constant $0<\varepsilon \leq 1$ such that if $f\in Y_{m^\prime+s+1}$ with
    \begin{equation*}\label{norma-piccola}
    \Vert f\Vert_{Y_{m^\prime+s+1}}\le\varepsilon,
  \end{equation*}
the equation $\mathcal L(u)=f$ has a solution $u\in X_{m^\prime}$, in the sense that there exists a sequence $\{u_n\}\subset X_\infty$ such that $u_n\rightarrow u$ in $X_{m^\prime}$, $\mathcal L(u_n) \rightarrow f$ in $Y_{m^\prime+s}$, as $n\rightarrow \infty$.
\item[(ii)] If (i) holds and if there exists $m^{\prime\prime}>m^\prime$ such that $f\in Y_{m^{\prime\prime}+s+1}$, then the solution constructed $u\in X_{m^{\prime\prime}}$.
    \end{itemize}
\end{theorem}


\end{document}